\documentclass[english]{article}
\usepackage[T1]{fontenc}
\usepackage[latin9]{inputenc}
\usepackage{amsmath}
\usepackage{amsthm}
\usepackage{amssymb}

\makeatletter

\providecommand{\tabularnewline}{\\}

\theoremstyle{plain}
\newtheorem{thm}{Theorem}[section]
  \theoremstyle{plain}
  \newtheorem{lem}[thm]{\protect\lemmaname}
  \theoremstyle{remark}
  \newtheorem{conclusion}[thm]{\protect\conclusionname}
  \theoremstyle{plain}
  \newtheorem{prop}[thm]{\protect\propositionname}
  \theoremstyle{plain}
  \newtheorem{cor}[thm]{\protect\corollaryname}
  \theoremstyle{remark}
  \newtheorem{rem}[thm]{\protect\remarkname}
  \theoremstyle{remark}
  \newtheorem{defi}[thm]{\protect\defname}
   \theoremstyle{remark}
  \newtheorem{example}[thm]{\protect\exname}
  \numberwithin{equation}{section}

\newcommand{\hide}[1]{}

\DeclareMathOperator{\supp}{supp}

\DeclareMathOperator{\ldim}{\underline{dim}}

\DeclareMathOperator{\edim}{dim_e}
\DeclareMathOperator{\RP}{\mathbf{P}}
\DeclareMathOperator{\RW}{\mbox{\tiny{RW}}}
\def\Prob{{\mathbb{P}}}
\def\conv{\mbox{\LARGE{$.$}}}

\def\ov{\overline}
\def\weakstar{weak$^{{\textstyle *}}$\ }
\def\R{{\mathbb{R}}}
\def\RR{{\mathbb{R}}}
\def\N{{\mathbb{N}}}

\def\C{{\mathbb{C}}}
\def\Gam{\Gamma}
\def\lam{\lambda}
\def\gam{\gamma}
\def\eps{\varepsilon}
\def\Dk{{\mathcal{D}}}
\def\E{{\mathbb{E}}}
\def\Ek{{\mathcal{E}}}
\def\Fk{{\mathcal{F}}}
\def\Q{{\mathbb Q}}
\def\Ak{{\mathcal{A}}}
\def\Bk{{\mathcal{B}}}
\def\Pk{{\mathcal{P}}}
\def\be{\begin{equation}}
\def\ee{\end{equation}}
\def\es{\emptyset}
\def\seq{\subseteq}

\makeatother

\usepackage{babel}
  \providecommand{\conclusionname}{Conclusion}
  \providecommand{\corollaryname}{Corollary}
  \providecommand{\lemmaname}{Lemma}
  \providecommand{\propositionname}{Proposition}
  \providecommand{\remarkname}{Remark}

\providecommand{\defname}{Definition}
\providecommand{\exname}{Example}

\makeatletter
\def\blfootnote{\xdef\@thefnmark{}\@footnotetext}
\makeatother

\begin{document}

\title{On the dimension of Furstenberg measure for $SL_{2}(\mathbb{R})$
random matrix products}

\author{Michael Hochman and Boris Solomyak}
\maketitle
\begin{abstract}
Let\blfootnote{M.H. supported by ERC grant 306494, B.S. supported by the Israel Science Foundation (grant 396/15)}\blfootnote{AMS 2010 subject classification 37F35} $\mu$ be a measure on $SL_{2}(\mathbb{R})$ generating a non-compact and totally irreducible subgroup, and let $\nu$ be the associated stationary (Furstenberg) measure for the action on the projective line. We prove
that if $\mu$ is supported on finitely many matrices with algebraic
entries, then 
\[
\dim\nu=\min\{1,\frac{h_{\RW}(\mu)}{2\chi(\mu)}\}
\]
where $h_{\RW}(\mu)$ is the random walk entropy of $\mu$, $\chi(\mu)$ is the Lyapunov exponent for the random matrix product associated with $\mu$, and $\dim$
denotes pointwise dimension. In particular, for every $\delta>0$,
there is a neighborhood $U$ of the identity in $SL_{2}(\mathbb{R})$
such that if a measure $\mu\in\mathcal{P}(U)$ is supported on algebraic matrices with all atoms of size at least $\delta$, and generates a group which is non-compact and totally irreducible, then its stationary measure $\nu$ satisfies  $\dim\nu=1$. 
\end{abstract}

\tableofcontents

\section{Introduction}

Let $\mu$ be a probability measure on the group $G=SL_2(\R)$. The 
linear action of $G$ on $\mathbb{R}^{2}$ induces
an action on the projective space $\RP$, and this action
admits stationary measures, that is, probability measures $\nu$ on $\RP$ satisfying $\mu\conv\nu=\nu$. Here and throughout, $\mu\conv\nu$ denotes the push-forward
of $\mu\times\nu$ by the action map $(A,x)\mapsto Ax$. Let $G_\mu$ denote the subgroup of $G$ generated by\footnote{Usually $G_\mu$ is defined as the smallest closed subgroup supporting $\mu$, and one assumes it is non-compact. Our definition was chosen so that when $\mu$ is finitely supported, $G_\mu$  is the countable group generated by its atoms, rather than the closure of this group. The usual non-compactness assumption then translates to unboundedness of $G_\mu$.} $\supp\mu$; assuming that $G_\mu$ is unbounded and totally irreducible (i.e. does not preserve any finite set in $\RP$), there exists a unique stationary measure, called the  Furstenberg
measure of $\mu$, which is central in the study of the asymptotic behavior of random matrix products, and is important in many other problems. For instance, in special cases $\nu$ is related to the density of states measure in the Anderson-Bernoulli model of random Schr\"{o}dinger operators (see Bourgain \cite{Bourgain2012,Bourgain2014}), and it has recently become a central ingredient in the dimension theory of self-affine sets and measures \cite{brny2015,FalconerKempton2016,Rapaport2015}.

An important question in applications is to determine how ``large'' the stationary measure is. One version of this question asks when it is absolutely continuous. This happens trivially when $\mu$ itself is absolutely continuous, but can also happen in non-trivial ways, e.g. it may occur even when $\mu$ is supported on finitely many points. That such examples exist was shown by \cite{BaranyPollicottSimon2012}, using a transversality
argument, and explicit examples were recently constructed by Bourgain \cite{Bourgain2012b,Bourgain2014} (and even more recently by Benoit and Quint
in higher dimensions \cite{BenoistQuint2016}). In general, it remains
a difficult problem to determine the analytical and smoothness properties
of $\nu$, but doing so is  important in many applications. 

In this paper we focus on the problem of determining  the dimension of $\nu$, where we say that $\nu$ has (exact) dimension $\alpha$  if $\nu(B_{r}(x))=r^{\alpha+o(1)}$ as $r\to 0$,
for $\nu$-a.e.\ $x$. This problem is complementary to the one on absolute continuity in the sense that it is interesting when absolute continuity fails (since when it holds, the dimension is one). There is a classical upper bound for the dimension, explained below, and our main result  is that, under
explicit and rather mild conditions, this upper bound is attained.

In order to state our result we begin with some notation. Fix $\mu\in\mathcal{P}(G)$ with stationary measure $\nu$, and let $\chi=\chi(\mu)$ denote the Lyapunov exponent of $\mu$,
which is the almost sure value of the limit 
\[
\chi=\chi(\mu)=\lim_{n\rightarrow\infty}\frac{1}{n}\log\left\Vert X_{n}\ldots X_{1}\right\Vert, 
\]
where $(X_{n})$ is a sequence of i.i.d.\ random variables with distribution
$\mu$, and the logarithm, here and throughout the paper, is in base $2$. We note that if $G_\mu$ is unbounded and totally irreducible, then $\chi>0$. Set
\[
\lambda=\lambda(\mu)=2^{\chi(\mu)},
\]
so that $\left\Vert X_{n}\ldots X_{1}\right\Vert = \lambda^{n+o(n)}=2^{(\chi+o(1))n}$ a.s. Then $\lambda$ is the maximal asymptotic expansion rate (or Lipschitz constant) of the matrix $X_{n}\ldots X_{1}$ acting on $\mathbb{R}^{2}$, but when the matrices act on $\RP$, the essential rate of expansion is asymptotically $\lambda^{2n}=2^{2\chi n}$, see Section \ref{subsec:expansion-estimates}
below. 

Let $H(\cdot)$ denote the Shannon entropy of a discrete measure or
random variable. Let $\mu^{\star n}$ denote the $n$-th self convolution
of $\mu$ in $G$, i.e. the distribution at time
$n$ of the random walk on $G$ started at time zero at $1_{G}$ and
driven by $\mu$. The random walk entropy of a discrete measure $\mu$ is then defined
as
\[
h_{\RW}(\mu)=\lim_{n\rightarrow\infty}\frac{1}{n}H(\mu^{\star n}).
\]
The limit is known to exist, see e.g.\ \cite{KaimanovichVershik1983},
and quantifies the degree of non-freeness of the semigroup $G_\mu^+$ generated\footnote{Again, here we mean the countable semigroup generated by $\supp\mu$, we do not take the closure. }
by $\supp\mu$: we always have $h_{\RW}(\mu)\leq H(\mu)$, with equality
if and only if the semigroup generated by $\supp\mu$ is generated
freely by it. 

Fix a left-invariant Riemannian metric  $d(\cdot,\cdot)$  on $G$.
We say that $\Ak\subset SL_2(\R)$ is \emph{Diophantine }if there is a constant $c>0$
such that every  pair of sequences $A_{1},\ldots,A_{n}$ and
$A'_{1},\ldots,A'_{n}$ with $A_{i},A'_{i}\in \Ak$ satisfies
$$
A_{1}\ldots A_{n}\neq A'_{1}\ldots A'_{n} \ \Longrightarrow\ d(A_{1}\ldots A_{n},A'_{1}\ldots A'_{n})>c^{n}.
$$
(this property of $\mathcal{A}$ is independent of the metric chosen, see Section \ref{subsec:Left-invariant-metric-on-G}).

\begin{thm}
\label{thm:main}Let $\mu$ be a finitely supported measure on $G$ with $G_\mu$ unbounded and totally irreducible. Then the unique $\mu$-stationary measure $\nu$
is exact dimensional, and if $\supp \mu$ is Diophantine, then
\begin{equation}
\dim\nu=\min\Bigl\{1,\frac{h_{\RW}(\mu)}{2\chi(\mu)}\Bigr\}. \label{eq:dim-formula}
\end{equation}
If, in addition,  $\supp\mu$ generates a free semigroup,
then
\[
\dim\nu=\min\Bigl\{1,\frac{H(\mu)}{2\chi(\mu)}\Bigr\}.
\]
\end{thm}

\begin{rem}
  \begin{enumerate}

  \item Exact dimensionality holds with no assumption on $\mu$ besides positivity
of the Lyapunov exponent, a fact which is known to the experts (and is related to the exact dimensionality of self-similar measures established by Feng and Hu
\cite{FengHu2009}), but has apparently not appeared in the literature. We provide a proof in Theorem \ref{thm:-exact-dimensionality} below.

\item The right-hand side of (\ref{eq:dim-formula}) is always an upper
  bound for $\dim\nu$.
  
  \item In \cite{Ledrappier1983}, Ledrappier showed\footnote{In \cite{Ledrappier1983} a different notion of dimension was
used for $\nu$, but since the dimension of $\nu$ is exact, it coincides
with the usual one.} that for any $\mu$ one has the formula 
\[
\dim\nu=\min\Bigl\{1,\frac{h_{F}(\mu)}{2\chi(\mu)}\Bigr\},
\]
where $h_{F}(\mu)=\int\int\log\frac{dA\nu}{d\nu}(x)\,dA\nu(x)d\mu(A)$
is the Furstenberg entropy of $\nu$. This is a more general result,
since it requires no assumptions whatsoever on  $\mu$, but it is of
less practical use than (\ref{eq:dim-formula}), because $h_{F}(\mu)$
is in general difficult to compute. One consequence of Theorem
\ref{thm:main} is that when $\supp \mu$ is Diophantine and $h_{\RW}(\mu)/2\chi(\mu)\leq1$,
we obtain its value, $h_{F}(\mu)=h_{\RW}(\mu)$.

\item Even when $\dim\nu=1$, Theorem \ref{thm:main} does not imply absolute
continuity. It is possible that this can be
proved, at least in some cases, by a refinement of our methods and those of Varj\'{u} \cite{Varju2016}, but this requires further investigation.
  \end{enumerate}
\end{rem}

In general, the Lyapunov exponent of $\mu$ is hard to calculate, but
when $\mu$ is supported close to the identity of $G$ the
exponent is clearly close to zero, and if the random walk entropy
of $\mu$ is not too small, the right-hand side in Theorem \ref{thm:main}
will be $1$. In bounding the random walk entropy we can rely on uniform
expansion bounds in $G$, due to Breuillard and Gelander
\cite{BreuillardGelander2008}, and associated spectral gap derived
in \cite{Breuillard2011}. Specifically, it follows from these results
that if $G_\mu$ 
is non-amenable, and if $\mu$ is finitely supported with $\mu(g)>\delta_{0}$
for every $g\in \supp\mu$, then there is a constant $c=c(\delta_{0})$ such
that $h_{\RW}(\mu)\geq c$ (see Section \ref{section:appl}). Moreover, unboundedness and total irreducibility of $G_\mu$ imply that it is non-amenable (see Section \ref{section:appl}). Combining these
facts with Theorem \ref{thm:main} we obtain
\begin{thm} \label{thm:appl1}
For every $\delta_{0}>0$ there exists an identity neighborhood $U\subseteq G$ such that if $\mu$  is a Diophantine measure supported on $U$ with $G_\mu$ unbounded and totally irreducible, and all of $\mu$'s atoms have mass at least $\delta_0$, then the stationary measure $\nu$ satisfies $\dim\nu=1$.
\end{thm}

One may wonder whether the assumption that $\mu$ is Diophantine is unnecessary, and the only assumptions needed are positive Lyapunov exponent and non-atomic stationary measure. There is some analogy between this and the conjecture that there is a left neighborhood of $1$ in which the corresponding Bernoulli convolutions are absolutely continuous (see \cite{PeresSchlagSolomyak2000} for background). In the Bernoulli convolutions setting, Varj\'{u} recently established that for every $h>0$ there is a left neighborhood of $1$ in which the corresponding Bernoulli convolution is absolute continuous if the parameter is algebraic and has height less than $h$ \cite{Varju2016}. Theorem \ref{thm:appl1} is in the same spirit but avoids the height assumption, although, instead of absolute continuity, it gives only  full dimension.

Finally, we mention that using similar methods to the ones in this paper, the following result for infinitely supported
measures can be proved.

\begin{thm} \label{thm:infsupp}
Let $\mu\in\mathcal{P}(G)$ be supported on a compact
set, have positive Lyapunov exponent and  non-atomic stationary measure $\nu$. 
If $\dim\mu>0$ then $\dim\nu=1$.
\end{thm}

We do not include a proof of this; a closely related result has appeared recently in \cite{Hochman2016}, and the reader is referred there for details. It remains a challenge to determine whether the same holds when $\mu$ is continuous (i.e. has no atoms), but of dimension zero. For some related questions see \cite{Hochman2016}.

All of the results above are valid more generally when $\mu$ is a measure on the group $\widetilde{G}$ of $2\times 2$ matrices of determinant $\pm 1$. In fact it can be derived from the $G$ case: indeed the action of $\widetilde{G}$ factors through that of $G$, and the fibers of the factor map  $\widetilde{G}\to G$ has two points. Thus if we start with a measure on $\widetilde{G}$ and project it to $G$, then neither the random walk entropy, nor amenability of $G_\mu$,  is affected; and the results for $G$ may be lifted to $\widetilde{G}$.

The proof of Theorem \ref{thm:main} is an outgrowth of methods from
\cite{Hochman2014,Hochman2015}, see also \cite{BreuillardVarju2015}, which dealt with the dimension of
self-similar measures in Euclidean space. The stationary measure $\nu$
is in many respects like a self-similar measure: when $\mu$ is finitely
supported, stationarity implies that $\nu$  decomposes into ``copies''
of itself via $\nu=\sum_{A\in\supp\mu}\mu(A)\cdot A\nu$. Two key
differences compared to the self-similar case are that, first, $A$
does not contract $\RP$, but only most of it; and the action of $A$
on $\RP$ is not linear. The latter in particular makes it impossible
to view $\nu$ at small scales as a convolution of a scaled copy of
$\mu$ with another measure, which was crucial to the argument in
\cite{Hochman2014}. However, the linearization method used in \cite{Hochman2015}
in the study of multidimensional self-similar measures can be applied
in the present setting to show that locally there is an approximate
convolution structure to $\nu$. This allows us to apply theorems
about convolutions of measures on $\mathbb{R}$, and carry through
the argument from the self-similar case. Much of our work will go
towards controlling the linearization, and we hope this paper can
serve as an exposition of the method. 

\subsection*{Organization}

The first two sections are devoted to developing basic properties of
the action. Specifically, in Section \ref{sec:Geometry-of-the-action} we
examine the metric properties of the action, estimating the amount
of expansion/contraction when an element of $g$ acts on $\RP$ or
an element of $\RP$ acts on $G$ by evaluation, and associated quantities;
and in Section \ref{sec:Furstenberg-measure} we recall the theorems
of Furstenberg and Oseledets, along with some variants, and prove
exact dimensionality of the stationary measure.

In Section \ref{sec:Entropy-and-dimension} we begin to discuss entropy,
introducing suitable partitions of $\RP$ and $G$, reviewing entropy-related
methods, establishing convergence of some entropies to the dimension
of $\nu$, and proving a reduction from Theorem \ref{thm:main} to
a statement about entropy growth under convolutions. In Section
\ref{sec:Inverse-theorem,-linearization-etc}, we recall from \cite{Hochman2014}
the inverse theorem for convolutions on $\mathbb{R}$ and develop
the linearization argument which allows us to transfer it to the present
setting and complete the proof of the main theorem. Finally, Section \ref{section:appl} contains the proof Theorem~\ref{thm:appl1} and several examples, as well as a discussion of the Diophantine property and other assumptions of our theorems.

\subsection*{Acknowledgment} We are grateful Emanuel Breuillard for drawing our attention to the results on uniform expansion in Lie groups. Also thanks to ICERM for their hospitality and their support during the spring 2016 semester program on dimension and dynamics.

\subsection*{Notation}

We summarize here our main notation and conventions.

\noindent \vfill{}
\begin{tabular}{ll}
\hline 
& \tabularnewline
$G$, $d(\cdot,\cdot)$ & $G=SL_{2}(\mathbb{R})$ with a left-invariant Riemannian metric $d(\cdot,\cdot)$ \tabularnewline
$\RP$, $d_{\RP}(\cdot,\cdot)$ & Real projective line with rotation-invariant measure $d_{\RP}(\cdot,\cdot)$ of diameter $1$.\tabularnewline
$\mu$ & Probability measure on $G$.\tabularnewline
$\nu$ & The stationary measure on $\RP$ associated with $\mu$ \tabularnewline
$\chi$ & The Lyapunov exponent of $\mu$ \tabularnewline
$\lambda$ & $=2^\chi$ (except in Section  \ref{sec:Geometry-of-the-action}) \tabularnewline
$G^+_\mu$, $G_\mu$ & The semigroup and group generated by $\supp\mu$.\tabularnewline
$B_{r}(x)$  & Open ball of radius $r$\tabularnewline
$\mathcal{P}(X)$ & Space of probability measures on $X$.\tabularnewline
$\overline{x}$ & The line (i.e. point in $\RP$) determined by $x\in\mathbb{R}^{2}\setminus\{0\}$. 
 \tabularnewline
$\angle_{u}(x)$ & Normalized angle between $x$ and $u$ (by default $u=(1,0)$).\tabularnewline
$\theta\conv \eta$  & Push forward of $\theta\times\eta$ through the action map $G\times\RP\rightarrow\RP$ \tabularnewline
$\theta\star\theta'$, $\theta^{\star n}$ & Convolution in $G$\tabularnewline
$\eta\ast\eta'$, $\eta^{\ast n}$ & Convolution in $\mathbb{R}$ (also applied to measures on $\RP\cong[0,1)$).\tabularnewline
$A^{*}$  & Transpose of a matrix $A$ (all our matrices are real)\tabularnewline
$\lambda_{A}^{+},\lambda_{A}^{-}$ & Singular values of $A$ (eigenvalues of $(A^{*}A)^{1/2}$), with $\lambda_{A}^{+}\geq\lambda_{A}^{-}$\tabularnewline
$u_{A}^{+},u_{A}^{-}$ & Singular vectors of $A$ (eigenvectors of $A^{*}A$ corresponding
to $\lambda_{A}^{+},\lambda_{A}^{-}$)\tabularnewline
$v_{A}^{+},v_{A}^{-}$ & $v_{A}^{+}=Au_{A}^{+}/\lambda_{A}^{+}$, $v_{A}^{-}=Au_{A}^{-}/\lambda_{A}^{-}$\tabularnewline
$\mathcal{D}_{n}$, $\mathcal{D}_{n}^{G}$ & level-$n$ dyadic partition of $\RP$ (or $\mathbb{R}$) and $G$
(see Section \ref{subsec:Dyadic-partitions-and-entropy})\tabularnewline
$S_{t}$ & Scaling map: $S_{t}(x)=2^{t}x$\tabularnewline
$T_{s}$ & Translation map: $T_{s}(x)=x+s$\tabularnewline
$\mu_{x,n}$ & Component measures, Section \ref{subsec:Component-measures}\tabularnewline
$\mu_A$ & Conditional measure, $\mu_A=\frac{1}{\mu(A)} \mu|_{A}$ \tabularnewline
$\mathbb{P}_{i\in I}$, $\mathbb{E}_{i\in I}$ & See Section \ref{subsec:Component-measures}\tabularnewline
$H(\mu,\mathcal{B})$ & Shannon entropy \tabularnewline
 $H(\mu,\mathcal{B}|\mathcal{C})$ & Conditional Shannon entropy \tabularnewline
$\dim\eta$ & Exact dimension of a measure $\eta$ (if exists).\tabularnewline
 & \tabularnewline
\hline 
\end{tabular}

\section{\label{sec:Geometry-of-the-action}Geometry of the action}

In this section we develop some elementary (geo)metric properties
of the $G$-action on $\RP$. In particular we estimate the contraction
properties of the maps $\overline{x}\mapsto g\overline{x}$ for $g\in G$,
and of the evaluation maps $g\mapsto g\overline{x}$ for $\overline{x}\in\RP$,
and variations on them.

\subsection{\label{subsec:Projective-space}Projective space and induced action}

Let $\RP=\mathbb{RP}^{1}$ denote the 1-dimensional projective space,
i.e. $(\mathbb{R}^{2}\setminus\{0\})/\sim$ where $\sim$ is the relation
of colinearity,  $x\sim y$ if and only if $x=cy$ for some $c\in\mathbb{R}$.
For $x\in\mathbb{R}^{2}\setminus\{0\}$ we write $\overline{x}\in\RP$
for its equivalence class, and generally denote elements of $\RP$
by $\overline{x}$, with $x$ an implicit representative. Later on we shall not distinguish notationally between elements of $\RR^2\setminus\{0\}$ and $\RP$, implicitly converting a vector $0\neq x\in\RR^2$ to the point $\ov{x}\in\RP$, and elements of $\RP$ with unit-vector reprsentatives. But for clarity the distinction is maintained in the following few sections.

Denote by $d_{\RP}(\cdot,\cdot)$ the rotation-invariant metric
on $\RP{}$ given for $\overline{x},\overline{y}\in\RP$
by 
\[
d_{\RP}(\overline{x},\overline{y})=\frac{1}{\pi}\left|\arcsin\left(\left(1-(\frac{\left\langle x,y\right\rangle }{\left\Vert x\right\Vert \left\Vert y\right\Vert })^{2}\right)^{1/2}\right)\right|.
\]
With this metric $\RP$ is isometric to $S^{1}$ with a metric proportional
to arc-length, normalized so that the total circumference is $1$.

For a unit vector $u$ let $u^{\perp}$ denote its rotation by $\pi/2$.
We obtain linear coordinates for $\RP$ by taking the (normalized)
angle that $\overline{x}$ forms with $u$; in fact for any representative
$x=(x_{1},x_{2})\in\mathbb{R}^{2}$ we define 
\[
\angle_{u}(x)=\frac{1}{\pi}\arctan\frac{\left\langle x,u\right\rangle }{\left\langle x,u^{\perp}\right\rangle }.
\]
This does not depend on the representative so we may write $\angle_{u}(\overline{x})$.
This map is discontinuous on the line $\mathbb{R}u^{\perp}$ (or at
${u}^{\perp}$ in $\RP$), but becomes continuous if we identify
the points $0,1$ in the range, i.e. take the range to be $\mathbb{R}/\mathbb{Z}$
instead of $[0,1)$. Note that for any $u$ the distance in $\RP$
is given by 
\[
d_{\RP}(\overline{x},\overline{y})=|\angle_{u}(x)-\angle_{u}(y)|,
\]
assuming the distance is less than $1/2$.
We write $\angle(w)=\angle_{(1,0)}(w)$ for the angle formed with
the $x$-axis. A section (partial inverse) map to $\angle_{u}$ is
given by the map $\gamma_{u}:[0,1)\rightarrow S^{1}$,
\[
\gamma_{u}(\theta)=\cos\pi\theta\cdot u+\sin\pi\theta\cdot u^{\perp}.
\]
The map is not continuous as a map $[0,1)\to S^{1}$ but is continuous and
well-defined as a map from $\mathbb{R}/\mathbb{Z}\rightarrow\RP$.
These maps provide a system of charts for $\RP{}$ with
connecting maps given by translation. 

Now suppose that $A$ is an invertible $2\times 2$ matrix and $\overline{x}\in\RP$. Then we can define $A\overline{x}=\overline{Ax}$, and this is independent of the representative $x$ because if $\overline{x}=\overline{y}$ then $x=ty$ for some $0\neq t\in\RR$, hence $Ay=Atx=t(Ax)$, so $\overline{Ay}=\overline{Ax}$. One similarly checks that if $B$ is another matrix then $B(A\overline{x})=(BA)\overline{x}$, and we obtain a well defined action of $SL_2(\RR)$ on $\RP$. The action is easily seen to be continuous. It is not faithful, since $A$ and $-A$ act in the same way, but the stabilizer is the two-point group $\{-1,1\}$, so locally in $G$, the action is faithful.

We do not distinguish notationally between a matrix $A$ and the induced map of $\RP$, denoting the latter also by $A$. In cases where there may be ambiguity we shall introduce suitable notation locally. For the next few sections this will not be a problem, as we are distinguishing explicitly between vectors and elements of $\RP$.

\subsection{\label{subsec:Singular-values-and-vectors}Some linear algebra: singular
values and singular vectors}

Let $A$ be a $d\times d$ real matrix and denote its transpose by
$A^{*}$. Then $A^{*}A$ is symmetric and positive-definite, so we
may list its eigenvalues, with multiplicities, in decreasing order,
writing them as $\lambda_{1}^{2}\geq\ldots\geq\lambda_{d}^{2}$. The
numbers $\lambda_{1}\geq\ldots\geq\lambda_{d}$ are called the \emph{singular
values }of $A$ (these are the square roots of the eigenvalues). Again
using the symmetry of $A^{*}A$, we can find an orthonormal basis
$u_{1},\ldots,u_{d}$ of eigenvectors of $A^{*}A$, with $u_{i}$
corresponding to $\lambda_{i}^{2}$, and assuming the singular values
are distinct, this basis is unique up to multiplication of the vectors
by $-1$. We refer to $u_{1},\ldots,u_{d}$ as the
\emph{singular vectors}\footnote{They are sometimes called the \emph{right singular vectors }of $A$. }\emph{
}of $A$. Note that
\[
\left\langle Au_{i},Au_{j}\right\rangle =\left\langle u_{i},A^{*}Au_{j}\right\rangle =\left\langle u_{i},\lambda_{j}^{2}u_{j}\right\rangle =\lambda_{j}^{2}\delta_{i,j},
\]
so $\{Au_{i}\}_{i=1}^{d}$ forms an orthogonal basis of $\mathbb{R}^{d}$,
with $\left\Vert Au_{i}\right\Vert =\lambda_{i}$. 

The geometric interpretation of these numbers is well-known: denoting
by $B_{1}(0)$ the Euclidean unit ball of $\mathbb{R}^{d}$, and $v_{i}=Au_{i}/\lambda_{i}$,
we find that
\begin{eqnarray*}
A(B_{1}(0)) & = & A\left\{ \sum a_{i}u_{i}\,:\,\sum a_{i}^{2}\leq1\right\} \\
 & = & \left\{ \sum a_{i}Au_{i}\,:\,\sum a_{i}^{2}\leq1\right\} \\
 & = & \left\{ \sum b_{i}v_{i}\,:\,\sum(b_{i}/\lambda_{i})^{2}\leq1\right\}.
\end{eqnarray*}
Since $\{v_{i}\}$ is an orthonormal basis of $\mathbb{R}^{d}$, this
shows that the image of the unit ball is an ellipsoid with principal
axes of lengths $\lambda_{1},\ldots,\lambda_{d}$ in directions $v_{1},\ldots,v_{d}$.

For a $2\times2$ matrix $A\in SL_{2}(\mathbb{R})$  we write $\lambda_{A}^{+}=\lambda_{1}=\left\Vert A\right\Vert $
and $\lambda_{A}^{-}=\lambda_{2}$, and similarly $u_{A}^{+}=u_{1}$,
$u_{A}^{-}=u_{2}$ and $v_{A}^{+}=v_{1}$, $v_{A}^{+}=v_{2}$. We
sometimes drop the subscript $A$ when it is clear from the context.
We note that $\lambda_{A}^{+}\cdot\lambda_{A}^{-}=\sqrt{\det A^{*}A}=\det A$,
so for $A\in SL_{2}(\mathbb{R})$ we have $\lambda_{A}^{+}\lambda_{A}^{-}=1$. 

The singular values of $A$ are the same as those of $A^{*}$.

The singular vectors of $A$ and $A^{*}$ do not have such
a simple relation, but when the singular values are large, the singular vectors are related by the following.
\begin{lem}
\label{lem:singular-vectors-of-transpose}Let $A\in SL_{2}(\mathbb{R})$
and write $\lambda=\lambda_{A}$, $u^{\pm}=u_{A}^{\pm}$. Then the
singular vectors $w^{\pm}=u_{A^{*}}^{\pm}$ of $A^{*}$ satisfy 
\begin{eqnarray*}
d_{\RP}(\overline{A^{*}w^{+}},{u^{+}}) & = & O(\lambda^{-2}),\\
d_{\RP}(\overline{A^{*}w^{-}},{u^{-}}) & = & O(\lambda^{-2}).
\end{eqnarray*}
\end{lem}
\begin{proof}
We have
\[
|\left\langle A^{*}w^{+},u^{-}\right\rangle |=|\left\langle w^{+},Au^{-}\right\rangle |\leq\left\Vert w^{+}\right\Vert \left\Vert Au^{-}\right\Vert =\lambda^{-}\left\Vert w^{+}\right\Vert  = \lambda^{-}.
\]
Therefore, writing $A^{*}w^{+}=a^{+}u^{+}+a^{-}u^{-}$, we conclude
that $|a^{-}|\leq\lambda^{-1}$, and using $\left\Vert A^{*}w^{+}\right\Vert =\lambda$
we have
\[
\frac{A^{*}w^{+}}{\left\Vert A^{*}w^{+}\right\Vert }=O(\lambda^{-2})u^{-}+\sqrt{1-O(\lambda^{-4})}\cdot u^{+},
\]
implying the claim. The distance for $\overline{A^{*}w^{-}},{u^{-}}$ follows by orthogonality.
\end{proof}

\noindent {\em Note.} Above  $u^{\pm}$ denote both unit singular vectors and the corresponding elements of $\RP$. Observe that a singular vector of  an $SL_2(\R)$-matrix with $\lambda>1$ is defined uniquely up to sign, so the corresponding element of $\RP$ is defined uniquely.

\subsection{\label{subsec:Left-invariant-metric-on-G}Left invariant metric on
$G$}

Let $G=SL_{2}(\mathbb{R})$, the group of $2\times2$ matrices of determinant
$1$. It is a $3$-dimensional real Lie group.

We endow $G$ with a left-invariant Riemannian metric $d(\cdot,\cdot)$,
so that 
\[
d(hg_{1},hg_{2})=d(g_{1},g_{2}),
\]
hence
\begin{eqnarray*}
B_{r}(h) & = & h\cdot B_{r}(1_{G}).
\end{eqnarray*}
Let us compare $d$ to the norm metric. Suppose that $g,g'\in G$.
Write $h=g^{-1}g'$, then
\[
g-g'=g(1_G-h).
\]
Now, there is an $r_{0}$ such that the metrics $d$ and the norm-metric
are bi-Lipschitz equivalent on $B_{r_{0}}(1_{G})$ (actually, this is true for any $r_0>0$), and in particular there is
a constant $\alpha$ such that $\|1_{G}-u\|\leq\alpha\cdot d(1_{G},u)$
for all $u\in B_{r_{0}}(1_{G})$. It follows that if $d(g,g')<r_{0}$,
then 
\[
\left\Vert g-g'\right\Vert \leq\left\Vert g\right\Vert \left\Vert 1_{G}-h\right\Vert \leq\left\Vert g\right\Vert \cdot\alpha\cdot d(1_{G},h)=\left\Vert g\right\Vert \cdot\alpha\cdot d(g,g').
\]
A similar calculation gives $\left\Vert g-g'\right\Vert \geq\left\Vert g^{-1}\right\Vert^{-1}\cdot\alpha^{-1}\cdot d(g,g')$. Using the fact that $\|g^{-1}\| = \|g\|$ for $g\in SL_2(\R)$,
we obtain
\begin{conclusion} \label{concl}
There exist an $r_{0}>0$ and $\alpha>0$ such that if $g,g'\in G$
and $d(g,g')<r_{0}$ then 
\[
\alpha^{-1}\left\Vert g\right\Vert ^{-1}\leq\frac{\left\Vert g-g'\right\Vert }{d(g,g')}\leq\alpha\left\Vert g\right\Vert. 
\]
In particular, a set $\mathcal{A}\subseteq G$ is Diophantine in the sense given in the  introduction if and only if there exists a constant $c>0$ such that for every pair of sequences $A_1,\ldots,A_n$ and $A'_1,\ldots,A'_n$ in $\mathcal{A}$, if $A_1\ldots A_n\neq A'_1\ldots A'_n$ then $\Vert A_1\ldots A_n- A'_1\ldots A'_n\Vert >c^n$. This also shows that the property of being Diophantine is independent of the left-invariant metric $d$ we choose.
\end{conclusion}

\subsection{\label{subsec:expansion-estimates}Expansion estimates and linearization}

Let $g\in G$, write $\lambda=\left\Vert g\right\Vert $, and let $u^{\pm}=u_{g}^{\pm}$
and $v^{\pm}=v_{g}^{\pm}$ be the elements of $\RP$ corresponding to the singular vectors of $g$ and their images under  the action of $g$ (see notation in Section \ref{subsec:Singular-values-and-vectors}).
Consider the map $\widehat{g}:\RP\rightarrow\RP$ with the coordinates
$\angle_{u^{+}}$ in the domain and the coordinates $\angle_{v^{+}}$
in the range, i.e. the map
\begin{eqnarray*}
\widehat{g}:\mathbb{R}/\mathbb{Z} & \rightarrow & \mathbb{R}/\mathbb{Z};\\
\theta & \mapsto & \angle_{v^{+}}\circ g\circ\gamma_{u^{+}}(\theta).
\end{eqnarray*}
We have
\begin{eqnarray*}
g(\gamma_{u^{+}}(\theta)) & = & g(\cos\pi\theta\cdot u^{+}+\sin\pi\theta\cdot u^{-})\\
 & = & (\lambda\cos\pi\theta\cdot v^{+},\lambda^{-1}\sin\pi\theta\cdot v^{-}),
\end{eqnarray*}
so
\begin{eqnarray*}
\widehat{g}(\theta) & = & \angle_{v^{+}}(\lambda\cos\pi\theta\cdot v^{+},\lambda^{-1}\sin\pi\theta\cdot v^{-})\\
 & = & \frac{1}{\pi}\arctan\frac{\lambda^{-1}\sin\pi\theta}{\lambda\cos\pi\theta}\\
 & = & \frac{1}{\pi}\arctan\frac{1}{\lambda^{2}}\tan\pi\theta.
\end{eqnarray*}
Calculating the derivative,
\[
\frac{d}{d\theta}\widehat{g}(\theta)=\frac{1}{\lambda^{2}\cos^{2}\pi\theta+\lambda^{-2}\sin^{2}\pi\theta}\,.
\]
This shows immediately that
\begin{lem}
\label{lem:g-contracts-on-all-of-RP}For $g\in G$ the induced map
$\widehat{g}:\RP\rightarrow\RP$ expands by at most $\lambda_{g}^{2}$
and contracts by at most $\lambda_{g}^{-2}$. 
\end{lem}
The upper and lower bound in the lemma are very far from each other,
but a much tighter estimate can be obtained if we exclude a small
part of $\RP$. Indeed, given $\varepsilon>0$, the ratio of $\cos\pi\theta$
and $\sin\pi\theta$ is bounded away from $0$ for $\theta\not\in(\frac{1}{2}-\eps, \frac{1}{2}+\eps)$, so 
\[
\bigl|\frac{d}{d\theta}\widehat{g}(\theta)\bigr|=\Theta_{\varepsilon}\bigl(\frac{1}{\lambda^{2}}\bigr)\qquad\mbox{for } \theta\not\in\bigl(\frac{1}{2}-\eps, \frac{1}{2}+\eps\bigr).
\]
Differentiating further, a similar calculation shows that 
\[
\bigl|\frac{d^{2}}{d\theta^{2}}\widehat{g}(\theta)\bigr|=O_{\varepsilon}\bigl(\frac{1}{\lambda^{2}}\bigr)\qquad\mbox{for } \theta\not\in\bigl(\frac{1}{2}-\eps, \frac{1}{2}+\eps\bigr).
\]

Let us say that a map $f$ between metric spaces scales by $a>0$ with distortion $b>0$ if $$b^{-1}<\frac{d(f(x),f(y))}{ad(x,y)}<b$$ for all $x,y$ in its domain. 

Recalling that distance in $\RP$ is given by the difference of angles,
and using elementary calculus, we obtain, noting that $\theta=\frac{1}{2}$ corresponds to $u^-=u^-_g$:
\begin{lem}
\label{lem:g-contracts-on-most-of-RP}For $g\in G$ the induced map
$\widehat{g}:\RP\rightarrow\RP$ scales $\RP\setminus B_{\varepsilon}(u_{g}^{-})$
by $\lambda_{g}^{-2}$ with distortion $\Theta_{\varepsilon}(1)$,
and furthermore, assuming $x,x_{0}\in\RP\setminus B_{\varepsilon}(u_{g}^{-})$, and using the coordinates given by the angle,
\[
\widehat{g}(x)=\widehat{g}(x_{0})+\widehat{g}'(x_{0})\cdot(x-x_{0})+O_{\varepsilon}\Bigl(\frac{(x-x_{0})^{2}}{\lambda_g^{2}}\Bigr).
\]
\end{lem}

\subsection{Separating $G$ by the action on $\RP$}

For a $k$-tuple ${x}=({x}_{1},\ldots,{x}_{k})\in\RP{}^{k}$
write $f_{{x}}:G\rightarrow\RP{}^{k}$ for the
map $g\mapsto(g{x}_{1},\ldots,g{x}_{k})$. Endow
the range with the supremum product metric and let matrices act pointwise
on $k$-tuples. Observe that for any $g,h\in G$ and ${x}\in\RP{}^{k}$,
\[
f_{{x}}(hg)=hf_{{x}}(g).
\]
For a $3$-tuple ${x}$ of distinct points, $f_{{x}}$
is a smooth injection and is bi-Lipschitz to its image in any sufficiently small compact
neighborhood of the identity in $G$. By another compactness argument,
for every $\varepsilon>0$, the bi-Lipschitz constants of the map can
be bounded independently of ${x}$ (but depending on $\varepsilon$)
as long as the coordinates of ${x}$ are $\varepsilon$-separated,
i.e. $d_{\RP}({x}_{i},{x}_{j})\geq\varepsilon$ for
$i\neq j$. 

Now for fixed $g_{0}\in G$ and $r>0$ consider $f_{{x}}$
restricted to $B_{r}(g_{0})=g_{0}\cdot B_{r}(1_{G})$. Assume that
the $3$-tuple ${x}=({x}_{1},{x}_{2},{x}_{3})$
is $\varepsilon$-separated and that ${x}_{i}\in\RP{}\setminus B_{\varepsilon}(u_{g_{0}}^{-})$.
Assume further that $r$ is small enough that $B_r(1_G)$ satisfies the conclusion of the previous paragraph and no $h\in B_{r}(1_G)$
moves points in $\RP{}$ by more than $\varepsilon/4$,
so $h{x}_{i}\in\RP{}\setminus B_{\varepsilon/2}(u^{-}_{g_{0}})$.
This implies that if $h\in B_{r}(1_G)$ and $g=g_{0}h$, then $g_{0}$
acts by contracting the coordinates of ${x}$ and $h{x}$
by $\Theta_{\varepsilon}(\left\Vert g_{0}\right\Vert ^{2})$, hence
$g_{0}^{-1}$  acts on $g_{0}{x}=f_{{x}}(g_{0})$
and $g{x}=g_{0}h{x}=f_{{x}}(g)$  by
expanding by $\Theta_\eps(\left\Vert g_{0}\right\Vert ^{2})$. Thus, for
$g\in B_{r}(g_{0})$ and $h=g_{0}^{-1}g\in B_{r}(1_{G})$, we have
\begin{eqnarray*}
d_{\RP}(f_{{x}}(g),f_{{x}}(g_{0})) & = & \Theta_{\varepsilon}(\left\Vert g_{0}\right\Vert ^{-2})\cdot d_{\RP}(g_{0}^{-1}f_{{x}}(g),g_{0}^{-1}f_{{x}}(g_{0}))\\
 & = & \Theta_{\varepsilon}(\left\Vert g_{0}\right\Vert ^{-2})\cdot d_{\RP}(f_{{x}}(h),f_{{x}}(1_{G}))\\
 & = & \Theta_{\varepsilon}(\left\Vert g_{0}\right\Vert ^{-2})\cdot d(h,1_{G})\mbox{ (because \ensuremath{f_{{x}}} is \ensuremath{\Theta(1)}-bi-Lip on \ensuremath{B_{r}(1_{G})})}\\
 & = & \Theta_{\varepsilon}(\left\Vert g_{0}\right\Vert ^{-2})\cdot d(g_{0}h,g_{0}1_{G})\mbox{ (because \ensuremath{d} is left-invariant)}\\
 & = & \Theta_{\varepsilon}(\left\Vert g_{0}\right\Vert ^{-2})\cdot d(g,g_{0}).
\end{eqnarray*}
We have proved:
\begin{lem}
\label{lem:separation-of-G-by-3-orbits}For any $\varepsilon>0$ and
$0<r<r(\varepsilon)$, if $g_{0}\in G$ and ${x}=({x}_{1},{x}_{2},{x}_{3})\in\RP^{3}$
is $\varepsilon$-separated and ${x}_{i}\notin B_{\varepsilon}(u_{g_{0}}^{-})$,
then the map $g\mapsto g{x}$ scales by $\left\Vert g_{0}\right\Vert ^{-2}$
with distortion $O_{\varepsilon}(1)$ on $B_{r}(g_{0})$.
\end{lem}

\section{\label{sec:Furstenberg-measure}Furstenberg measure}

In this section we review basic results from the theory of random
matrix products and Furstenberg measure, and set up some notation
that will be used later. We then prove that the Furstenberg measure
is exact dimensional. 

\subsection{\label{subsec:The-theorems-of-Furst-and-Oseledets}The theorems of
Furstenberg and Oseledets}

We review some classical results on random matrix products, and prove some quantitative variants.

\begin{thm}
[Furstenberg] \label{thm:Furstenberg}If $\mu\in\mathcal{P}(G)$
with $G_{\mu}$ unbounded and totally irreducible, then there
exists a unique $\mu$-stationary measure $\nu$ on $\RP{}$,
and it is non-atomic. Furthermore, if $X_{1},X_{2},\ldots$ denotes an
i.i.d. sequence of matrices with marginal $\mu$, then with probability
one $X_{1}X_{2}\ldots X_{n}\nu$ converges weakly to a random Dirac
mass $\delta_{{z}}$, and $\nu=\mathbb{E}(\delta_{{z}})$. 
\end{thm}
The unique stationary measure is called the \emph{Furstenberg measure}
of $\mu$. 

\medskip

For our application, the weak convergence in the theorem is too coarse.
What we will actually need is information on how far $X_{1}\ldots X_{n}\nu$
typically is from $\delta_{{z}}$, and more generally, for
fixed ${u}\in\RP$, how far $X_{1}X_{2}\ldots X_{n}{u}$
is from ${z}$. We will derive this information from another
classical result:
\begin{thm}
[Furstenberg, Oseledets] \label{thm:Oseledets} Let $\mu\in\mathcal{P}(G)$
with $G_{\mu}$ unbounded and totally irreducible, and let $X_{1},X_{2},\ldots$
be i.i.d. matrices with marginal $\mu$. Then the almost sure limit
$\chi=\lim_{n\rightarrow\infty}\frac{1}{n}\log\left\Vert X_{n}^{*}\ldots X_{1}^{*}\right\Vert $
exists and is a.s.\ constant and is positive,
and, writing $\lambda=2^\chi$, we have

\begin{enumerate}
\item[{\rm (i)}] With probability one, there exists a (random) orthonormal pair $u^{+},u^{-}\in\mathbb{R}^{2}$
such that $\left\Vert X_{n}^{*}\ldots X_{1}^{*}u^{+}\right\Vert =\lambda^{n(1+o(1))}$
and $\left\Vert X_{n}^{*}\ldots X_{1}^{*}u^{-}\right\Vert =\lambda^{-n(1+o(1))}$. 
\item[{\rm (ii)}] Writing $Y_{n}=X_{n}^{*}\cdot\ldots\cdot X_{1}^{*}$ and ${u}_{n}^{+}={u}_{Y_{n}}^{+}$,
${u}_{n}^{-}={u}_{Y_{n}}^{-}$, with probability
one we have ${u}_{n}^{+}\rightarrow{u}^{+}$ and
${u}_{n}^{-}\rightarrow{u}^{-}$ in $\RP$ and $d_{\RP}({u}_{n}^{+},{u}^{+})=d_{\RP}({u}_{n}^{-},{u}^{-})=\lambda^{-2n(1+o(1))}$.
\item[{\rm (iii)}] Writing $Z_{n}=Y_{n}^{*}=X_{1}X_{2}\ldots X_{n}$ and $w_{n}^{+}=Z_{n}u_{Z_{n}}^{+}$,
$w_{n}^{-}=Z_{n}u_{Z_n}^{-}$, with probability one we have $w_{n}^{+}\rightarrow u^{+}$
and $w_{n}^{-}\rightarrow u^{-}$ in $\RP$, and $d_{\RP}({w}_{n}^{+},{u}^{+})=d_{\RP}({w}_{n}^{-},{u}^{-})=\lambda^{-2n(1+o(1))}$.
\end{enumerate}
\end{thm}
Part (i) is standard; part (ii) is not usually given in the statement
but it follows from some of the standard proofs (e.g. \cite{Raghunathan1979}).
In order to derive (iii), note that by Lemma \ref{lem:singular-vectors-of-transpose}
we have $d_{\RP}({w}_{n}^{+},{u}_{n}^{+})=O(\left\Vert Y_{n}\right\Vert ^{-2})$.
By (i), $\left\Vert Y_{n}\right\Vert =\lambda^{n(1+o(1))}$ and by
(ii) $d_{\RP}({u}_{n}^{+},{u}^{+})=\lambda^{-2n(1+o(1))}$.
Combining the last three bounds gives (iii).

The version we require of Theorem \ref{thm:Furstenberg}, which was alluded to earlier, is the following:

\begin{sloppypar}
\begin{prop}
\label{prop:quantitative-Furstenberg}Let
$\mu$, $X_{1},X_{2},\ldots$ and $\lambda,\chi,{u}^{\pm},{u}_{n}^{\pm},{w}_{n}^{\pm}$
be as in the last theorem. Then $\nu=\mathbb{E}(\delta_{{u}^{+}})$
and $\nu=\lim_{n\rightarrow\infty}\mathbb{E}(\delta_{{w}_{n}^{+}})$
(in the weak$^{\,{\textstyle *}}$ sense). Furthermore, if $\varepsilon_{n}\rightarrow0$
slowly enough, then for every ${z}\in\RP$, 
\begin{equation}
\mathbb{P}\left(d_{\RP}(X_{1}\ldots X_{n}{z},u^{+})<2^{-2(\chi-\varepsilon_{n})n}\right)\rightarrow1\qquad\mbox{as }n\rightarrow\infty.\label{eq:3}
\end{equation}
\end{prop}
\end{sloppypar}

\begin{proof}
Fix $\delta>0$. By Lemma \ref{lem:g-contracts-on-most-of-RP}, the
(random) set $I_{n}^{\delta}=\RP{}\setminus B_{\delta}({u}_{X_{1}\ldots X_{n}}^{-})$
is mapped by $X_{1}\ldots X_{n}$ into a ball of radius $O_{\delta}(\left\Vert X_{1}\ldots X_{n}\right\Vert^{-2} )$
around ${w}_{n}^{+}$, and by (iii) of Oseledets's theorem,
with probability one, $d_{\RP}({w}_{n}^{+},{u}^{+})<2^{-2(\chi+o(1))n}$
as $n\rightarrow\infty$. By the same theorem also $\left\Vert X_{1}\ldots X_{n}\right\Vert^{-2} =2^{-2(\chi+o(1))n}$
as $n\rightarrow\infty$. It follows that with probability one, if
${z}_{n}\in I_{n}^{\delta}$ then 
\begin{equation} \label{qua:Furst}
d_{\RP}(X_{1}\ldots X_{n}{z}_{n},{u}^{+})<2^{-2(\chi+o(1))n}\qquad\mbox{as }n\rightarrow\infty.
\end{equation}
Therefore by choosing $\varepsilon_{n}\rightarrow0$ slowly enough, using Egorov's Theorem and Borel-Cantelli Lemma,
we ensure that with probability one, for every sequence ${z}_{n}\in I_{n}^{\varepsilon_{n}}$, for $n$ sufficiently large,
\begin{equation}
d_{\RP}(X_{1}\ldots X_{n}{z}_{n},{u}^{+})<2^{-2(\chi-\varepsilon_{n})n}.\label{eq:2}
\end{equation}

Let $\nu$ denote the stationary measure. By total irreducibility
of $G_{\mu}$ we know that $\nu$ is non-atomic. Since $\varepsilon_{n}\rightarrow0$
it follows that $\nu(B_{\varepsilon_{n}}({z}))\rightarrow0$
as $n\rightarrow\infty$, uniformly in ${z}\in\RP$, and
in particular $\nu(I_{n}^{\varepsilon_{n}})\rightarrow1$ as $n\rightarrow\infty$
uniformly over the probability space (note again that $I_{n}^{\varepsilon_{n}}$
is random). Using $\nu=\mathbb{E}(X_{1}\ldots X_{n}\nu)$, which is
immediate from stationarity, the above implies
\[
\left\Vert \nu-\mathbb{E}(X_{1}\ldots X_{n}(\nu|_{I_{n}^{\varepsilon_{n}}}))\right\Vert \rightarrow0.
\]
(here $\Vert\cdot\Vert$ denotes total variation). But (\ref{eq:2}) implies that with probability one,
\[
X_{1}\ldots X_{n}(\nu|_{I_{n}^{\varepsilon_{n}}})\rightarrow\delta_{{u}^{+}}\qquad\mbox{as }n\rightarrow\infty,
\]
where convergence is in the \weakstar sense, so $\mathbb{E}(X_{1}\ldots X_{n}(\nu|_{I_{n}^{\varepsilon_{n}}}))\rightarrow\mathbb{E}(\delta_{{u}^{+}})$
as $n\rightarrow\infty$. Combining this with the previous limit gives
$\nu=\mathbb{E}(\delta_{{u}^{+}})$. Since almost surely
$\delta_{{w}_{n}^{+}}\rightarrow\delta_{{u}^{+}}$
as $n\rightarrow\infty$, we also obtain $\nu=\lim_{n\rightarrow\infty}\mathbb{E}(\delta_{{w}_{n}^{+}})$.

  Finally, (3.1) will follow from (3.3) once we show that for each
      fixed
      $z$ and $\delta>0$, \[
          \lim_{n\to\infty}\mathbb{P}(u^-_{X_1\ldots X_n} \in
      B_{\delta}(z)) =
      \delta'
      \]
      where $\delta'\to 0$ as $\delta\to 0$. Now, the distribution of
      $X_1\ldots
      X_n$ is the same as that of $X_n\ldots X_1$ (because the
      sequeunce is
      i.i.d.). So we must prove \[
          \lim_{n\to\infty}\mathbb{P}(u^-_{X_n\ldots X_1} \in
      B_{\delta}(z)) =
      \delta'
      \]
      According to Theorem 3.1, the probability in the last equation
      converges to
      $\nu^*(B_\delta(z))$, where $\nu^*$ is the Furstenberg measure
      for the
      random product of  $X_n^*$. Since these matrices also generate
      an unbounded
      totally irreducibly group when the original ones do, it follows
      that $\nu^*$
      is continuous, and so taking $\delta'=\sup_z
      \nu^*(B_\delta(z))$, the claim
      is proved.
\end{proof}

\subsection{\label{subsec:Symbolic-coding}Symbolic coding }

Let $\Omega_{0}=\supp\mu\subseteq G$ and $\Omega=(\Omega_{0})^{\mathbb{N}}$,
endowed with the product structure and the product measure $\widehat{\mu}=\mu^{\mathbb{N}}$.
For a word $w\in\Omega_{0}^{*}$ we write 
\begin{eqnarray*}
A_{w} & = & w_{1}\cdot\ldots\cdot w_{n};\\
\lambda_{w} & = & \left\Vert A_{w}\right\Vert \quad=\quad\lambda_{A_{w}}^{+}.
\end{eqnarray*}
Assume that $\mu$ satisfies the assumptions of Theorem \ref{thm:Furstenberg}
and let $\nu$ be the stationary measure. Then by the same theorem,
the map $\pi:\Omega\rightarrow \RP$, $w\mapsto\pi w$,
given by 
\[
\delta_{\pi w}=\lim_{n\rightarrow\infty}w_{1}\ldots w_{n}\nu\qquad\widehat{\mu}\mbox{-a.e. }w,
\]
is defined $\widehat{\mu}$ -a.e., and 
\[
\nu=\int\delta_{\pi w}d\widehat{\mu}(w).
\]
We denote by $S$ the left shift on $\Omega$.

\subsection{\label{subsec:Exact-dimensionality-of-Furst-measure}Exact dimensionality
of Furstenberg measure}

The Furstenberg entropy of $\nu$ is defined by
\begin{eqnarray*}
h_{F}(\nu) & = & \int\int\log\frac{dA\nu}{d\nu}(x)\,dA\nu(x)\,d\mu(A).
\end{eqnarray*}
It is known that $0\leq h_{F}(\nu)\leq H(\mu)$. The quantity $h_{F}(\mu)$
can also be expressed as $-\int\int\log\frac{dA^{-1}\nu}{d\nu}(x)d\nu(x)d\mu(A)$.
The equality of the two expressions can be obtained by applying $A^{-1}$
to the inner integral above.

Recall that $\nu$ is said to have dimension $\alpha$ if $\nu(B_{r}(x))=r^{(1+o(1))\alpha}$
as $r\rightarrow0$, at $\nu$-a.e.\ $x$. Ledrappier \cite{Ledrappier1983} showed that $\log\nu(B_{r}(x))/\log r\rightarrow\alpha=h_{F}(\nu)/2\log\lambda$
in $\nu$-probability as $r\rightarrow0$.
\begin{thm}
\label{thm:-exact-dimensionality} Assume that $\mu$ satisfies the assumptions of Theorem \ref{thm:Oseledets}
and let $\nu$ be the stationary measure. Then $\nu$ is exact dimensional and the local dimension
is $\nu$-a.e.\ equal to $h_{F}(\nu)/2\log\lambda$.
\end{thm}
Let us first explain the main idea of the proof. By a version of the
Besicovitch differentiation theorem, we know that
\begin{equation}
\frac{dA\nu}{d\nu}(x)=\lim_{I\rightarrow x}\frac{A\nu(I)}{\nu(I)}\qquad\nu\mbox{-a.e. }x, \label{eq:measure-differentiation}
\end{equation}
where the limit is over intervals $I=(a,b)$ containing $x$ as $b-a\rightarrow0$.
Now suppose that we fix a $\widehat{\mu}$-typical $w=(A_{1},A_{2},\ldots)\in\Omega$
and set $x=\pi w$. Let $r=\lambda^{-2N}$ for a large $N$. We wish
to estimate $\nu(B_{r}(x))$. We may write 
\[
\nu(B_{r}(x))=\left(\prod_{n=1}^{N}\frac{\nu((A_{1}\ldots A_{n-1})^{-1}B_{r}(x))}{\nu((A_{1}\ldots A_{n})^{-1}B_{r}(x))}\right)\cdot\nu((A_{1}\ldots A_{N})^{-1}B_{r}(x)).
\]
\begin{sloppypar}
\noindent Now, if we ignore the error terms in Oseledets's theorem one expects
$(A_{1}\ldots A_{n})^{-1}B_{r}(x)$ to be a neighborhood of $(A_{1}\ldots A_{n})^{-1}x$
of diameter $\lambda^{2n}r=\lambda^{2(n-N)}$, and in particular, $(A_{1}\ldots A_{N})^{-1}B_{r}(x)$
has diameter $O(1)$. Also note that $(A_{1}\ldots A_{n})^{-1}x$ is
just $\pi(S^{n}w)$, and that $\nu((A_{1}\ldots A_{n})^{-1}B_{r}(x))=A_{n}\nu((A_{1}\ldots A_{n-1})^{-1}B_{r}(x))$.
Therefore, taking logarithms in the previous equation and multiplying
by $-1/N$, we obtain 
\end{sloppypar}
\[
-\frac{1}{N}\log\nu(B_{\lambda^{-2N}}(x))=\frac{1}{N}\sum_{n=1}^{N}\log\frac{A_{n}\nu(B_{\lambda^{2(n-N)}}(\pi(S^{n-1}w)))}{\nu(B_{\lambda^{2(n-N)}}(\pi(S^{n-1}w)))}+O(\frac{1}{N}).
\]
Now, when $1\ll n\leq N$ the ball $B_{\lambda^{2(n-N)}}$is very
small, so by (\ref{eq:measure-differentiation}) one can expect the
$n$-th term in the sum above to be very close to $\frac{dA_{n}\nu}{d\nu}(\pi(S^{n-1}w))$.
Making this substitution, the average above becomes the ergodic average
of the function $w\mapsto\frac{dw_{1}\nu}{d\nu}(\pi w)$, and we find
that $-\frac{1}{N}\log\nu(B_{\lambda^{-2N}}(x))\rightarrow h_{F}(\nu)$
as $N\rightarrow\infty$. After normalizing properly, this is what
the theorem claims.

The remainder of this section is devoted to making the sketch above
precise. We begin with a more detailed discussion of the differentiation
theorem for measures. First, we note that (\ref{eq:measure-differentiation})
is well known when the convergence $I\rightarrow x$ is restricted
to right-handed neighborhoods $[x,x+r)$ of $x$, or to left-handed
neighborhoods; and together these two limits give (\ref{eq:measure-differentiation}).
Next, write
\begin{eqnarray*}
f_{A}(I) & = & \frac{A\nu(I)}{\nu(I)}.
\end{eqnarray*}
We shall require the following maximal-type inequalities, which are
essentially taken from the proof of the differentiation theorem:
\begin{lem}
\label{lem:maximal-inequality}For every $A\in\Omega_0$ and $t>0$,
\begin{eqnarray*}
\nu(\{x\,:\,\sup_{I\,:\,x\in I}f_{A}(I)>t\}) & \leq & 2t^{-1}
\end{eqnarray*}
and 
\begin{eqnarray*}
A\nu(\{x\,:\,\inf_{I\,:\,x\in I}f_{A}(I)<t^{-1}\}) & \leq & 2t^{-1}.
\end{eqnarray*}
\end{lem}
\begin{proof}
Let $E\subseteq\RP{}$ denote the set of $x$ such that
there exists an interval $I_{x}$ containing $x$ with $f_A(I_x) > t$, that is, $A\nu(I_x) > t\nu(I_x)$.
We want to show that $\nu(E)\leq2/t$. Write $I_x = (x - r_x^-, x+r_x^+)$, and let $E^{+},E^{-}\subseteq E$
denote the sets of $x$ such that, respectively, $(x-r_{x}^{-},x]$
and $[x,x+r_{x}^{+})$ contain at least half the $A\nu$-mass of $I_{x}$.
Then $E=E^{-}\cup E^{+}$ so it suffices to show that $\nu(E^{\pm})\leq1/t$.
Let us show this for $E^{+}$. Given $\varepsilon>0$, we can find
$R>0$ and a compact subset $E_{R}^{+}\subseteq E^{+}$ such that
$\nu(E_{R}^{+})>(1-\varepsilon)\nu(E)$ and $r_{x}^{+}\geq R$ for
$x\in E_{R}^{+}$. 
We can choose a finite sequence $\{x_{i}\}\subseteq E_{R}^{+}$ such that
$E_{R}^{+}\subseteq\bigcup_{i}[x_{i},x_{i}+r_{x_{i}}^{+})$ and this union is disjoint.
Hence 
\[
1\geq\sum_{i}A\nu([x_{i},x_{i}+r_{x_{i}}^{+}))\geq\sum_{i}t\nu([x_{i},x_{i}+r_{x_{i}}^{+}))\geq t\nu(E_{R}^{+})\geq t(1-\varepsilon)\nu(E^{+}).
\]
The claim follows by taking $\varepsilon\rightarrow0$. For the second
inequality, reverse the roles of $\nu$ and $A\nu$ and argue in
the same way. 
\end{proof}
\begin{cor}
\label{cor:maximal-inequality}
\begin{eqnarray*}
\widehat{\mu}(w\in\Omega\,:\,\sup_{I\,:\,\pi w\in I}f_{w_{1}}(I)\geq t) & \leq & 2t^{-1},\\
\widehat{\mu}(w\in\Omega\,:\,\inf_{I\,:\,\pi w\in I}f_{w_{1}}(I)\leq t^{-1}) & \leq & 2t^{-1}.
\end{eqnarray*}
\end{cor}
\begin{proof}
Both follow from the previous lemma using the relations $\nu=\pi\widehat{\mu}$,
$\pi(\widehat{\mu}_{[A]})=A\nu$ and $\widehat{\mu}=\sum_{A\in\mathcal{A}}\mu(A)\widehat{\mu}_{[A]}$,
and by decomposing $\Omega$ into $\bigcup_{A\in\mathcal{A}}[A]$.
\end{proof}
Let $F^{\pm}:\Omega\times\mathbb{R}^{+}\rightarrow\mathbb{R}^{+}$
be the functions
\begin{eqnarray*}
F^{+}(w,r) & = & \sup\{\log f_{w_{1}}(I)\,:\,\pi w\in I\mbox{ and }|I|\leq r\},\\
F^{-}(w,r) & = & \inf\{\log f_{w_{1}}(I)\,:\,\pi w\in I\mbox{ and }|I|\leq r\},
\end{eqnarray*}
and denote
\[
\Delta F(w,r)=F^{+}(w,r)-F^{-}(w,r).
\]
For $\widehat{\mu}$-a.e. $w$ we have $F^{+}(w,r)\rightarrow\log\frac{dA_{1}\nu}{d\nu}(\pi w)$
and $F^{-}(w,r)\rightarrow\log\frac{dA_{1}\nu}{d\nu}(\pi w)$ as $r\rightarrow0$,
hence $\Delta F(w,r)\rightarrow0$ (this relies on $-\infty<\frac{dA_{1}\nu}{d\nu}(\pi w)<\infty$,
which is the case for $\widehat{\mu}$-a.e. $w$). Note that $F^{+}(w,r)\geq F^{-}(w,r)$,
so $\Delta F\geq0$, and that 
\[
r_{1}\leq r_{2}\qquad\implies\qquad\Delta F(w,r_{1})\leq\Delta F(w,r_{2}).
\]
Also, observe that for $\widehat{\mu}$-a.e. $w\in\Omega$ and every
interval $I$ containing $\pi w$, both $\log\frac{dA_{1}\nu}{d\nu}(\pi w)$
and $\log f_{w_{1}}(I)$ lie between $F^{-}(w,|I|)$ and $F^{+}(w,|I|)$,
hence
\begin{equation}
\left|\log\frac{dA_{1}\nu}{d\nu}(\pi w)-\log f_{w_{1}}(I)\right|\leq\Delta F(w,|I|).\label{eq:approximation-of-RD-derivative}
\end{equation}

\begin{cor}
\label{cor:maximal-inequality-2}The function $w\mapsto\sup_{r>0}\Delta F(w,r)$
is in $L^{1}(\widehat{\mu})$.
\end{cor}
\begin{proof}
By interchanging sups and logarithms,
\[
\sup_{r>0}\Delta F(w,r)=\log\sup_{I\,:\,\pi w\in I}f_{w_{1}}(I)-\log\inf_{I\,:\,\pi w\in I}f_{w_{1}}(I),
\]
so it suffices to show that each term on the right hand side is in
$L^{1}$. These arguments are the same and we give details only for
the first term. For any non-negative function $G:\Omega\rightarrow\mathbb{R}$
we have $\int Gd\widehat{\mu}=\int_{0}^{\infty}\widehat{\mu}(G\geq t)dt$,
so, setting 
\[
G(w)=\log\sup_{I\,:\,\pi w\in I}f_{w_{1}}(I),
\]
the previous corollary implies 
$$\widehat{\mu}(w\,:\,G(w)\geq t) =  \widehat{\mu}(w\,:\,\sup_{I\,:\,\pi w\in I}f_{w_{1}}(I)\geq e^{t})
 \le  2e^{-t}.
 $$
Thus $\int_{0}^{\infty}\widehat{\mu}(G\geq t)dt<\infty$, as desired.
\end{proof}
Finally, in the sketch above we eventually arrived at an ergodic average.
The justification for this move is a variant of Maker's theorem.
\begin{thm}
\label{thm:Maker}Let $(X,$$\mathcal{F},\theta,T)$ be an ergodic
measure-preserving system. Let $G_{t}:X\rightarrow\mathbb{R}$ be a measurable
1-parameter family of measurable functions (i.e. $(t,x)\mapsto G_{t}(x)$
is measurable) such that $\sup_{t}|G_{t}|\in L^{1}$, and suppose
that 
\[
G=\lim_{t\rightarrow0}G_{t}
\]
exists a.e. Let $t_{N,n}:X\rightarrow\mathbb{R}$ be functions with
the property that for $\theta$-a.e. $x$ and every $\varepsilon>0$,
for large enough $N$, 
\[
|t_{N,n}|<\varepsilon\qquad\mbox{for }1\leq n\leq(1-\varepsilon)N.
\]
Then
\[
\frac{1}{N}\sum_{n=1}^{N}G_{t_{N,n}(x)}(T^{n}x)=\int Gd\theta\qquad\theta\mbox{-a.e. }x.
\]
\end{thm}
The proof is a minor modification of the standard Maker's theorem \cite{Maker1940}
and we omit it.

\begin{proof}
[Proof of exact dimensionality]  Fix a $\widehat{\mu}$-typical $w=(A_{1},A_{2},\ldots)\in\Omega$
and let $x=\pi(w)$, so that $x$ is also the expanding direction
of the sequence of products $A_{1}\ldots A_{n}$ in the Oseledets
theorem. Let $\lambda_{n}=\left\Vert A_{1}\ldots A_{n}\right\Vert $
denote the larger singular values of $A_{1}\cdot\ldots\cdot A_{n}$,
and $u_{n}^{\pm}$ the corresponding singular vectors. In particular,
$u_{n}^{+}\rightarrow x$. 
Let $\delta>0$ be such that $\nu(B_\delta(y))>0$ for all $y\in\RP$ (such $\delta$ exists since $\nu$ is non-atomic). 
Assume that $w$ satisfies the Oseledets theorem and (\ref{qua:Furst}) with parameter $\delta$. More precisely, we assume that for some $\eps_n\to 0$,
\begin{eqnarray*}
\lambda^{n(1-\eps_{n})}\quad\leq & \lambda_{n} & \leq\quad\lambda^{n(1+\eps_{n})},
\end{eqnarray*}
and $A_{1}\ldots A_{n}$ maps $I_n^\delta:= \RP\setminus B_\delta(u_n^-)$ into a neighborhood
of $x$ of radius $(\lambda^{-2})^{n(1-\eps_{n})}$. 

Fix a large $N$ (which we mostly suppress in our notation) and set 
\[
r=r_{N}=\lambda^{-2(1-\eps_{N})N}
\]
 and 
\begin{eqnarray*}
I_{0} & = & B_{r}(x); \\
I_{n} & = & (A_{1}\ldots A_{n})^{-1}I_{0}.
\end{eqnarray*}
Then $I_{0}$ is a neighborhood of $x$ and $I_{n}$ is a neighborhood
of 
\[
x_{n}=(A_{1}\ldots A_{n})^{-1}x=\pi(S^{n}w).
\]
By Lemma \ref{lem:g-contracts-on-all-of-RP}, every interval in $\RP$
is expanded under $A_{1}\ldots A_{n}$ by at most $\lambda_{n}^{2}\leq\lambda^{2n(1+\eps_{n})}$,
so 
\[
|I_{n}|\leq\lambda_{n}^{2}|I_{0}|\leq2\lambda^{-2(N-n)}\cdot\lambda^{2N(\eps_{N}+\eps_{n}n/N)}.
\]
Writing 
\[
\widetilde{\eps}_{N}=\frac{\log\lam}{N}+\eps_{N}+\sup_{1\leq n\leq N}\eps_{n}\frac{n}{N},
\]
we have $\widetilde{\eps}_{N}\rightarrow0$, and have obtained the
bound
\begin{equation}
|I_{n}|\leq\lambda^{-2((1-\widetilde{\eps}_{N})N-n)}.\label{eq:upper-bound-for-length-of-I-n}
\end{equation}
We do not require a lower bound for $|I_{n}|$, but shall want one
for $|I_{N}|$. Let $U=I_n^\delta=\RP\setminus B_{\delta}(u_{N}^{-})$.
Then by our choice of $\eps_{N}$ we know that $I_{0}\supseteq A_{1}\ldots A_{N}U$,
whereby 
\[
I_{N}=(A_{1}\ldots A_{N})^{-1}I_{0}\supseteq(A_{1}\ldots A_{N})^{-1}(A_{1}\ldots A_{N}U)=U.
\]
By definition of $\delta$ there exists $c>0$ such that
\begin{equation}
\nu(I_{N})>c\qquad\mbox{for all }N.\label{eq:lower-bound-on-nu-I-N}
\end{equation}
Finally, we note that
\[
\ldim(\nu,x)=\lim_{N\rightarrow\infty}-\frac{\log\nu(B_{\lambda^{-2N}}(x))}{2N\log\lambda}=\frac{1}{2\log\lambda}\lim_{N\rightarrow\infty}-\frac{\log\nu(B_{r_{N}}(x))}{N}.
\]
Therefore we need to estimate $\frac{1}{N}\log\nu(B_{r_{N}}(x))=\frac{1}{N}\log\nu(I_{0})$.

Assuming the parameters above have been fixed, write

\begin{eqnarray*}
\nu(I_{0}) & = & \left(\prod_{n=1}^{N}\frac{A_{1}\ldots A_{n-1}\nu(I_{0})}{A_{1}\ldots A_{n}\nu(I_{0})}\right)\cdot(A_{1}\ldots A_{N}\nu(I_{0}))\\
 & = & \left(\prod_{n=0}^{N}\frac{\nu(I_{n-1})}{A_{n}\nu(I_{n-1})}\right)\cdot\nu(I_{N}).
\end{eqnarray*}
Taking logarithms,
\[
-\log\nu(I_{0})=-\log\nu(I_{N})+\sum_{n=1}^{N}\log f_{A_{n}}(I_{n-1}).
\]
By (\ref{eq:approximation-of-RD-derivative}), we have 
\begin{eqnarray*}
\left|\log\frac{dA_{n}\nu}{d\nu}(\pi S^{n-1}w)-\log f_{A_{n}}(I_{n-1})\right| & \leq & \Delta F(S^{n-1}w,|I_{n-1}|).
\end{eqnarray*}
Together with the previous equation we get 
\begin{eqnarray*}
\left|-\frac{1}{N}\log\nu(I_{0})-\frac{1}{N}\sum_{n=1}^{N}\log\frac{dA_{n}\nu}{d\nu}(\pi(S^{n-1}w))\right| & \leq & \frac{1}{N}\log\nu(I_{N})\\
 &  & +\frac{1}{N}\sum_{n=1}^{N}\Delta F(S^{n-1}w,|I_{n}|).
\end{eqnarray*}
By the ergodic theorem, the average on the left-hand side converges
$\widehat{\mu}$-a.s. to $h_{F}(\nu)$, so we will be done once we
show that the right-hand side tends to $0$. Indeed, by (\ref{eq:lower-bound-on-nu-I-N})
the first term is $O(1/N)$. As for the second term, by (\ref{eq:upper-bound-for-length-of-I-n})
we know that we have $|I_{n}|<t_{N,n}$ for $t_{N,n}=\lambda^{-2((1-\widetilde{\eps}_{N})N-n)}$,
which clearly satisfies the hypotheses of Theorem \ref{thm:Maker}.
This shows that $\frac{1}{N}\sum_{n=1}^{N}\Delta F(S^{n-1}w,|I_{n}|)\rightarrow0$
a.s. (Here we use the fact that $\lim_{r\to 0} \Delta F(w,r) = 0$  for $\widehat{\mu}$-a.e.\ $w$ by (\ref{eq:measure-differentiation}).) This completes the proof.
\end{proof}

\section{\label{sec:Entropy-and-dimension}Entropy and dimension}

As in the introduction, we fix $\mu\in\mathcal{P}(G)$
with Lyapunov exponent $\chi=\log\lambda>0$, and let
$\nu=\mu\conv\nu$ be the unique stationary measure. 

\subsection{\label{subsec:Expansion-and-re-scaling}Expansion and re-scaling}

We often identify $\RP$ with $[0,1)$ via $\RP\ni(x,y)\mapsto\frac{1}{2}+\frac{1}{\pi}\arctan y/x\in[0,1)$.
One benefit is that we can now apply scaling to $\RP$, and we define
$S_{t}:\mathbb{R}\rightarrow\mathbb{R}$ by 
\[
S_{t}x=2^{t}x.
\]
Note that $S_{s+t}=S_{s}S_{t}$. 

The basic geometric fact about the action of $g\in G$ on $\RP$ is
that on the complement of an $\varepsilon$-ball it scales by a
factor of $\left\Vert g\right\Vert ^{-2}$ with bounded distortion. Thus, locally it acts like $S_{-2\log\left\Vert g\right\Vert +O_{\varepsilon}(1)}$,
followed by a translation.

 In particular if $\eta\in\mathcal{P}(\RP)$ is a
probability measure without atoms, then for $g$ far enough from the
identity all but a negligible fraction of the support of $\eta$
is contracted roughly by a factor of $\left\Vert g\right\Vert ^{2}$.

\subsection{\label{subsec:Dyadic-partitions-and-entropy}Dyadic partitions}

Let $\mathcal{D}_{n}=\{[k/2^{n},(k+1)/2^{n})\;:\;0\leq k<2^{n}\}$
denote the level-$n$ dyadic partition of $[0,1)$. We transfer $\mathcal{D}_{n}$
to $\RP$ using the usual identification. When $t$ is not an integer, we write $\Dk_t = \Dk_{\lfloor t \rfloor}$. Also, 
$\mathcal{D}_{n}(x)$ denotes the unique
element of $\mathcal{D}_{n}$ containing it.

We also want dyadic-like partitions $\mathcal{D}_{n}^{G}$ on $G$.
Specifically, we want a family of partitions $\mathcal{D}_{1}^{G},\mathcal{D}_{2}^{G},\ldots$
of $G$ into measurable sets such that for some constant $M$ the
following holds:
\begin{enumerate}
\item[(i)] $\mathcal{D}^G_{n+1}$ refines $\mathcal{D}^G_{n}$.
\item[(ii)] Every $D\in\mathcal{D}^G_{n}$ contains at most $M$ elements of $\mathcal{D}^G_{n+1}$.
\item[(iii)] Every $D\in\mathcal{D}^G_{n}$ contains a ball of radius $\frac{1}{M}\cdot2^{-n}$
and is contained in a ball of radius $M\cdot2^{-n}$ (recall that we are using the left-invariant metric).
\end{enumerate}
There are various ways to get such a system. If we replace
$2^{-n}$ by $r_{0}^{n}$ for some small $r_{0}\in (0,1)$,
then such partitions $\Ek_n^G$ can be constructed in the general setting of
a doubling metric space (see e.g.\ \cite{Kaenmaki2012} and references
therein). For the sake of tradition we stick with cells of size roughly $2^{-n}$; we can obtain them by setting $\Dk_n^G = \Ek^G_{\lfloor -n/ \log r_0\rfloor}$.

Alternatively, we can use the structure of $G = SL_2(\R)$ and construct the partitions explicitly, or use local charts to pull back the standard Dyadic partition on $\RR^3$, taking care to ``stitch'' the partitions together where the charts meet so that the properties are preserved. We leave such possibilities to the interested reader.

We extend the notation $\mathcal{D}_t$ for non-integer $t$ and $\mathcal{D}_n(x)$, discussed above for $\RP$, to the partition of $G$.

\subsection{\label{subsec:Preliminaries-on-entropy}Preliminaries on entropy}

We write $H(\eta,\mathcal{A})=-\sum_{A\in\mathcal{A}}\eta(A)\log\eta(A)$
for the entropy of a probability measure $\eta$ with respect to a
partition $\mathcal{A}$. Here the logarithm is in base 2 and $0\log 0 =0$. The conditional entropy with respect to a countable partition $\Fk$ is
$$
H(\eta,\Ak|\Fk) = \sum_{F\in \Fk} \eta(F)\cdot H(\eta_F,\Ak),
$$
where $\eta_F$ is the conditional measure on $F$.
For a probability measure $\eta$ on $[0,1)$ or $\RP$,
the quantity
\[
\frac{1}{n}H(\eta,\mathcal{D}_{n})
\]
is called the scale-$n$ entropy of $\eta$. This quantity is between
$0$ and $1$ and gives a finite-scale approximation of the dimension.

For a discrete probability measure $\mu$ we write $H(\mu)$ for the entropy with respect to the partition into points, and for a probability vector $\alpha = (\alpha_1,\ldots,\alpha_k)$, we write $H(\alpha) = -\sum_{i=1}^k \alpha_i \log \alpha_i$. We will rely on standard properties of entropy, for more details see  \cite[Section 3.1]{Hochman2014} and \cite[Section 2.4]{Hochman2016}. The properties below hold for measures on the real line (or on $\RP$), or on the group $G$; in the latter case the entropy is considered relative to the partitions $\Dk_n^G$.

\begin{lem} \label{lem:entropy_prop} As a function of the measure, entropy satisfies:
  \begin{enumerate}
\item[{\rm (i)}] $H(\cdot, \Ak)$ is concave.
\item[{\rm (ii)}] $H(\cdot, \Ak)$ is ``almost-convex'': if $\alpha$ is a probability vector, then
$$
H(\sum \alpha_i \mu_i, \Ak) \le \sum \alpha_i H(\mu_i,\Ak) + H(\alpha).
$$
\end{enumerate}
Properties (i) and (ii) also hold for the conditional entropy $H(\cdot, \Ak|\Fk)$.
\end{lem}

We note how entropy is affected by scaling and translation: for $t,u\in \R$ we have
\begin{equation} \label{ent:scale}
H(S_t \mu, \Dk_{n-t}) = H(\mu,\Dk_n) + O(1)
\end{equation}
and
\[H(T_u\mu,\Dk_n) = H(\mu,\Dk_n) + O(1).\]
If $\mu$ is supported on a set of diameter $2^{-(n+c)}$, then
\begin{equation} \label{ent:supp}
H(\mu,\Dk_n) = O_c(1),
\end{equation}
and hence, for $m>n$,  and if again $\mu$ is supported on a set of diameter $2^{-(n+c)}$, then using the basic identity $H(\mu,\Dk_m|\Dk_n) = H(\mu,\Dk_m) - H(\mu,\Dk_n)$ we have
\begin{equation} \label{ent:supp1}
H(\mu,\Dk_m|\Dk_n) = H(\mu,\Dk_m)-O_c(1).
\end{equation}

Next we collect some useful estimates for the entropy.
\begin{lem} \label{lem:entropy_perturb}
\begin{enumerate}
\item[{\rm (i)}] If each $E\in \Ek$ intersects at most $k$ elements of $\Fk$ and vice versa, then $|H(\mu,\Ek) - H(\mu,\Fk)| = O(\log k)$.
\item[{\rm (ii)}] If $f,g:\ \RP\to \RP$ and $|f(x) - g(x)| \le C2^{-m}$ for $x\in \RP$, then
$$
|H(f\mu,\Dk_m) - H(g\mu,\Dk_m)| \le O_C(1).
$$
\item[{\rm (iii)}]
If $f$ is bi-Lipschitz with constant $C$, then
$$
H(f\mu,\Dk_n) = H(\mu,\Dk_n) + O(\log C).
$$
\item[{\rm (iv)}] for every $\eps>0$ there is $\delta>0$ such that if $\eta,\theta$ are two measures and $\|\eta-\theta\|< \delta$, then for any finite partition $\Ak$ with $k$ elements,
\[
|H(\eta,\Ak) - H(\theta,\Ak)| < \eps.
\]
\end{enumerate}
\end{lem}

Combining the last two lemmas we also obtain

\begin{lem} \label{lem:entropy_distort}
Suppose that $f$ scales $\supp\mu$ by $u>0$ with distortion $C$. Then 
$$
H(f\mu,\Dk_{n-\log u}) = H(\mu, \Dk_{n}) + O_C(1).
$$
\end{lem}

\subsection{\label{subsec:Convergence-of-entropy-to-dimension}Convergence of
entropy to dimension}

Write 
\[
\alpha=\dim\nu.
\]
We have seen that $\nu$ is exact dimensional (Theorem \ref{thm:-exact-dimensionality}).
This implies that\footnote{The limit $\lim_{n\rightarrow\infty}\frac{1}{n}H(\eta,\mathcal{D}_{n})$
(which in general may not exist, or may be distinct from the dimension)
is sometimes called the entropy dimension of $\eta$, and denoted
$\edim\eta$.} 
\[
\lim_{n\rightarrow\infty}\frac{1}{n}H(\nu,\mathcal{D}_{n})=\alpha.
\]
Below we establish that a number of other natural entropies converge
to $\alpha$ as well. Before doing so we need a simple lemma.
\begin{lem}
\label{lem:coupling-bound}Let $\xi,\zeta$ be random variables defined
on a probability space $(\Omega,\mathcal{F},\mathbb{P})$, and with
values in $\RP$. Suppose that
$\mathbb{P}(d_{\RP}(\xi,\zeta)<2^{-(1-\varepsilon)n})=1-\delta$.
Write $\eta=\mathbb{E}(\delta_{\xi})$ and $\theta=\mathbb{E}(\delta_{\zeta})$.
Then 
\begin{equation}
\left|H(\theta,\mathcal{D}_{n})-H(\eta,\mathcal{D}_{n})\right|=O((\varepsilon+\delta )n+H(\delta)),\label{eq:5}
\end{equation}
and in fact
\begin{equation}
H(\mathbb{P},\xi^{-1}\mathcal{D}_{n}\,|\,\zeta^{-1}\mathcal{D}_{n})=O((\varepsilon+\delta )n+H(\delta)).\label{eq:6}
\end{equation}
\end{lem}
\begin{proof}
Suppose first that $\delta=0$, so that $d_{\RP}(\xi,\zeta)<2^{-(1-\varepsilon)n}$
always holds. Consider the random pair $(\xi,\zeta)\in\RP^{2}$
and the measure $\tau=\mathbb{E}(\delta_{(\xi,\zeta)})$. Let $\mathcal{T}=\{\emptyset,\RP\}$
denote the trivial partition of $\RP$ and note that $H(\eta,\mathcal{D}_{n})=H(\tau,\mathcal{D}_{n}\times\mathcal{T})$
and $H(\theta,\mathcal{D}_{n})=H(\tau,\mathcal{T}\times\mathcal{D}_{n})$.
Now, by the assumption that $d_{\RP}(\xi,\zeta)<2^{-(1-\varepsilon)n}$,
we see that given $\xi$, or the atom of $\mathcal{D}_n$ to which it belongs, there are $O(2^{\varepsilon n})$
possible atoms of $\Dk_n\times \Dk_n$ to which $\zeta$ can belong. Thus 
\begin{eqnarray*}
H(\tau,\mathcal{D}_{n}\times\mathcal{D}_{n}) & = & H(\tau,\mathcal{D}_{n}\times\mathcal{D}_{n}|\mathcal{D}_{n}\times\mathcal{T})+H(\tau,\mathcal{D}_{n}\times\mathcal{T})\\
 & = & O(\varepsilon n)+H(\eta,\mathcal{D}_{n}),
\end{eqnarray*}
and similarly $H(\tau,\mathcal{D}_{n}\times\mathcal{D}_{n})=O(\varepsilon n)+H(\theta,D_{n})$.
Hence assuming $\delta=0$ we have found that
\[
\left|H(\theta,\mathcal{D}_{n})-H(\eta,\mathcal{D}_{n})\right|=O(\varepsilon n).
\]

\begin{sloppypar}
For $\delta>0$, let $\Omega$ denote the event that $d_{\RP}(\xi,\zeta)<2^{-(1-\varepsilon)n}$,
so $\mathbb{P}(\Omega)=1-\delta$. Let $\eta'=\mathbb{E}(\delta_{\xi}|\Omega)$,
$\eta''=\mathbb{E}(\delta_{\xi}|\Omega^{c})$ and let $\theta'=\mathbb{E}(\delta_{\zeta}|\Omega)$
and $\theta''=\mathbb{E}(\delta_{\zeta}|\Omega^{c})$, so that $\eta=(1-\delta)\eta'+\delta\eta''$
and $\theta=(1-\delta)\theta'+\delta\theta''$. Thus, by Lemma~\ref{lem:entropy_prop} (i) and (ii) we have
\[
\left|H(\eta,\mathcal{D}_{n})-H(\eta',\mathcal{D}_{n})\right|<H(\delta)+\delta\left|H(\eta'',\mathcal{D}_{n})-H(\eta', \mathcal{D}_{n})\right|\leq H(\delta)+2\delta n,
\]
where we used that $|\Dk_n|=2^n$. Similarly, we have $\left|H(\theta,\mathcal{D}_{n})-H(\theta',\mathcal{D}_{n})\right|<H(\delta)+2\delta n$.
Finally, by the previous paragraph we know that $\left|H(\theta',\mathcal{D}_{n})-H(\eta',\mathcal{D}_{n})\right|=O(\varepsilon n)$.
Combining these three bounds gives (\ref{eq:5}).
\end{sloppypar}

For the second statement, we can apply the first statement to the
random variables $\xi'=(\xi,\xi)$, $\zeta'=(\xi,\zeta)$ and measures
$\eta'=\mathbb{E}(\delta_{\xi'})$, $\theta'=\mathbb{E}(\delta_{\zeta'})$.
Since $H(\eta',\mathcal{D}_{n}\times\mathcal{D}_{n})=H(\eta,\mathcal{D}_{n})$
we conclude that
\[
\left|H(\eta,\mathcal{D}_{n})-H(\theta',\mathcal{D}_{n}\times\mathcal{D}_{n})\right|=O((\varepsilon+\delta )n+H(\delta)).
\]
But also $H(\eta,\mathcal{D}_{n})=H(\mathbb{P},\xi^{-1}\mathcal{D}_{n})$
and $H(\theta',\mathcal{D}_{n}\times\mathcal{D}_{n})=H(\mathbb{P},\xi^{-1}\mathcal{D}_{n}\lor\zeta^{-1}\mathcal{D}_{n})$,
so the last bound is precisely the conditional entropy in (\ref{eq:6}).
\end{proof}

\begin{prop}
\label{prop:convergence-of-orbit-entropy-to-dimension}
For any $c>0$, for any $A_n \seq G$, with  $\mu^{\star n} (A_n) > c$, for any ${w}\in\RP$
we have 
\begin{equation} \label{conv1}
\frac{1}{2\chi n} H((\mu^{\star n})_{A_n}\conv{w}, \Dk_{2\chi n})=\alpha+o_c(1) \ \ \ \mbox{as}\  n\rightarrow\infty.
\end{equation}
\end{prop}
\begin{proof}
First we prove the inequality for $A_n \equiv G$, that is,
\begin{equation} \label{conv2}
\frac{1}{2\chi n} H(\mu^{\star n} \conv{w}, \Dk_{2\chi n})=\alpha+o(1) \ \ \ \mbox{as}\  n\rightarrow\infty.
\end{equation}
Fix ${w}\in\RP$. Let $X_{1},X_{2},\ldots$ be i.i.d. random
matrices with marginal $\mu$, so that $\mu^{\star n}$ is the distribution
of the random product $Z_{n}=X_{1}X_{2}\ldots X_{n}$, and let ${u}^{+}=\lim Z_n {u}_{Z_{n}}^{+}$,
so that 
\[
\nu=\mathbb{E}(\delta_{{u}^{+}}).
\]
Next, let ${w}_{n}=Z_{n}w$, so that 
$
\mu^{\star n} \conv {w}=\mathbb{E}(\delta_{w_n}).
$
By Proposition \ref{prop:quantitative-Furstenberg}, we have
\begin{equation} \label{conv3}
\mathbb{P}\left(d_{\RP}({u}^{+},{w}_{n})<2^{-2(\chi-\varepsilon_{n})n}\right)\rightarrow1\qquad\mbox{as }n\rightarrow\infty,
\end{equation}
for some $\eps_n\to 0$. 
By the previous lemma
(taking $\xi={u}^{+}$ and $\zeta={w}_{n}$ and
replacing $n$ by $2\chi n$, which is formally equivalent), we have
\[
\left|\frac{1}{2\chi n}H(\nu,\mathcal{D}_{2\chi n})-\frac{1}{2\chi n}H(\mu^{\star n} \conv{w},\mathcal{D}_{2\chi n})\right|\rightarrow0\qquad\mbox{as }n\rightarrow\infty,
\]
and (\ref{conv2}) follows.

Now, in order to prove the asymptotic equality (\ref{conv1}), it is enough to prove the inequality $\ge$. Indeed, by concavity and almost-convexity of entropy (Lemma~\ref{lem:entropy_prop}), the convex combination of the normalized entropy in the left-hand side of (\ref{conv1}) together with the analogous normalized entropy for the complements of $A_n$, is asymptotically equal to the normalized entropy of $\mu^{\star n}\conv w$, and we already know that the latter converges to $\alpha$ by (\ref{conv2}). Thus, knowing the inequality for both $A_n$ and $A_n^c$ (the complementary sets) proves the desired equality. Note that if the masses of the complements tend to zero, then their contribution to the convex combination also tends to zero, so we get equality in this case also.

To this end, define  \[
\nu_{n}:= \E(\delta_{u^+}\,|\,Z_n\in A_n),
\]
and note that, using $w_n = Z_n w$ as above, we have
\[
(\mu^{\star n})_{A_n} \conv w = \mathbb{E}(\delta_{w_n}\,|\,Z_n\in A_n).
\]
By (\ref{conv3}), 
\[
1 - \mathbb{P}\left(d_{\RP}({u}^{+},{w}_{n})<2^{-2(\chi-\varepsilon_{n})n}\,|\, Z_n \in A_n\right) = o_c(1)  \ \mbox{as}\  n\rightarrow\infty.
\]
Using the previous lemma again
(taking $\xi={u}^{+}$ and $\zeta={w}_{n}$ conditioned by $\{Z_n \in A_n\}$) yields
\[
\left|\frac{1}{2\chi n}H(\nu_n,\mathcal{D}_{2\chi n})-\frac{1}{2\chi n}H((\mu^{\star n})_{A_n} \conv{w},\mathcal{D}_{2\chi n})\right|\rightarrow0\qquad\mbox{as }n\rightarrow\infty.
\]
To complete the proof we must now show that $$
H(\nu_{n},\mathcal{D}_{2\chi n})\ge 2\chi\alpha n+o_c(n).
$$
For this, note that by the Radon-Nikodym theorem,  $\nu_n = \int f_n\,d\nu$ for some $0\le f_n \le 1/c$. By exact dimensionality of $\nu$ we have $\nu(B_r(x))=r^{\alpha+o_x(1)}$ for $\nu$-a.e.\ $x$, hence, as is well-known,
$\nu(\Dk_n(x)) = r^{\alpha+o_x(1)}$ for $\nu$-a.e.\ $x$.
For such $x$ the bound on $f_n$ gives $\nu_n(\mathcal{D}_n(x)) \le 2^{-n(\alpha+o_{x,c}(1))}$. Taking logarithms and integrating gives the desired entropy bound, and proves the proposition.
\end{proof}

\begin{cor}
  Let $c$ and $A_n\subseteq G$ be as in the previous proposition. Let $w=(w_1,w_2,w_3)\in\RP^3$, and write $\rho_w(g)=gw=(gw_1,gw_2,gw_3)$ and \[
  \mathcal{E}_n = \rho_w^{-1}\mathcal{D}_{2\chi n}
\]
(where $\mathcal{D}_{2\chi n}$ refers to the product partition in $\RP^3$). Then \[
  \frac{1}{2\chi n}H((\mu^{\star n})_{A_n},\mathcal{E}_n) = \alpha + o_{c,r}(1)
  \]
\end{cor}

\begin{proof}
Let $\rho_{w_i}(g)=gw_i$ and let \[
  \mathcal{E}_{n,i} = \rho_{w_i}^{-1}\mathcal{D}_{2\chi n}
\]
so that $\mathcal{E}_n=\bigvee_{i=1,2,3}\mathcal{E}_{n,i}$. Then
\[
H((\mu^{\star n})_{A_n}, \Ek_n)=H((\mu^{\star n})_{A_n},\mathcal{E}_{n,1})+H\Bigl((\mu^{\star n})_{A_n},\mathcal{E}_{n,2}\vee \mathcal{E}_{n,3}\,|\,\mathcal{E}_{n,1}\Bigr).
\]
By the previous proposition we know that the first summand is of order $2\chi \alpha n+o(n)$,
so we must show that the other summand is $o(n)$. For this it is enough to show that for each $i=2,3$.
\begin{equation}
H\bigl((\mu^{\star n})_{A_n},\mathcal{E}_{n,i}\,|\,\mathcal{E}_{n,1}\bigr)=o(n)\label{eq:1}
\end{equation}
This follows from Lemma \ref{lem:coupling-bound}. Indeed, adopting the notation of Proposition \ref{prop:quantitative-Furstenberg}, writing  ${w}_{n,i}=X_{1}X_{2}\ldots X_{n}{w}_{i}$, and assuming $\varepsilon_{n}\rightarrow0$
slowly enough, by that proposition
\[
\mathbb{P}\left(d_{\RP}({w}_{n,i},{u}^{+})<2^{-(2\chi-\varepsilon_{n})n}\right)\rightarrow1\qquad\mbox{as }n\rightarrow\infty,
\]
whereby by a slight increase of $\varepsilon_{n}$ we have
\[
\mathbb{P}\left(d_{\RP}({w}_{n,i},{w}_{n,1})<2^{-(2\chi-\varepsilon_{n})n}\right)\rightarrow1\qquad\mbox{as }n\rightarrow\infty.
\]
Since $\mu^{\star}(A_n)>c$, the last equation remains true if we condition on $X_1\ldots X_n\in A_n$, and then (\ref{eq:1}) follows from Lemma \ref{lem:coupling-bound}, equation (\ref{eq:6}).
\end{proof}

\begin{prop}
\label{prop:convergence-of-convolution-entropy-to-dimension}$\frac{1}{2\chi n}H(\mu^{\star  n},\mathcal{D}_{1}^{G})=\alpha +o(1)$
as $n\rightarrow\infty$.
\end{prop}
\begin{proof}
Fix $1/10$-separated ${w}_{1},{w}_{2},{w}_{3},{w}_{4}\in\RP$.
Let $r<1/20$ be small enough to apply Lemma \ref{lem:separation-of-G-by-3-orbits}
(with $\varepsilon=1/10$). For any $g_{0}\in G$ we certainly have
$u_{g_{0}}^{-}\notin B_{r}({w}_{i})$ for three out of the
four points. Thus for each $n$ we can partition $G$ into measurable sets $A_n^1,A_n^2,A_n^3,A_n^4$ such that for each $i=1,2,3,4$, \[
 g \in A_n^i \ \ \implies \ \ u_{g}^- \notin B_r(w_j)\ \ \ \textrm{ for } j\neq i,
\]
so that
\[
  p_{n,i}: = \mu^{\star n}(A_n^i) > c >0.
  \]
(In fact, we can take $c=1/5$, provided $r>0$ is sufficiently small. This follows from the claim that 
$$
\lim_{r\to 0} \lim_{n\to \infty}\mu^{\star n}(\{g\in G:\, u_g^-\in B_r(w)\})=0,
$$
uniformly in $w\in \RP$, which is a consequence of  Proposition~\ref{prop:quantitative-Furstenberg} and continuity of the stationary measure.)
Now observe that 
$$
\Bigl|H(\mu^{\star n}, \Dk_1^G) - \sum_{i=1}^4 p_{n,i}\cdot H((\mu^{\star n})_{A_n^i}, \Dk_1^G)\Bigr| \le 2,
$$
by Lemma~\ref{lem:entropy_prop}, using that $2=\log 4$ is the maximal entropy of $H(\mu^{\star n}, \{A_n^i\}_{i=1}^4)$. Thus
it is enough for us to show that for every $\eps>0$, for all large enough $n$, 
\begin{equation}
   |\frac{1}{2\chi n}H((\mu^{\star n})_{A_n^i}, \Dk_1^G) - \alpha | < \eps\;\;\;\textrm{ assuming } p_{n,i}>c. \label{eq:71}
\end{equation}
Without loss of generality we prove this for $i=4$. From now on we fix $i=4$ and a parameter $c>0$.
Define \[
\Ek_n:=\bigvee_{i=1,2,3}\rho_{w_i}^{-1}(\mathcal{D}_{2\chi n})
\]
(recall again that $\rho_w(g) = gw$).
Let $\varepsilon_{n}\rightarrow0$ and 
\[
\Gamma_{n}=\{g\in G\,:\,2^{-(2\chi+\varepsilon_{n})n}\leq\left\Vert g\right\Vert ^{-2}\leq2^{-(2\chi-\varepsilon_{n})n}\}.
\]
By Theorem \ref{thm:Oseledets} we can choose $\varepsilon_{n}$ so
that
\begin{equation}
\mu^{\star n}(\Gamma_{n})\rightarrow 1.\label{eq:4}
\end{equation}
By Lemma~\ref{lem:separation-of-G-by-3-orbits} and our choice of the sets $A_n^j$, for any $g_0\in A_n^4$ the map $B_r(g_0) \to \RP^3$,
$g\mapsto (gw_j)_{j\ne 4}$, scales by $\|g_0\|^{-2}$ with distortion $O(1)$.
It follows that if $2^{-(2\chi+\varepsilon)n}\leq\left\Vert g_{0}\right\Vert ^{-2}\leq2^{-(2\chi-\varepsilon)n}$,
then in an $O(1)$-neighborhood of $g_{0}\in A_n^4$ each atom of $\mathcal{D}_{1}^{G}$
can be covered by $O(2^{\varepsilon n})$ atoms of 
$\mathcal{E}_n$ and vice versa. It follows that
\[
\left|H((\mu^{\star n})_{\Gamma_{n}\cap A_n^4},\mathcal{D}^G_{1})-H((\mu^{\star n})_{\Gamma_{n}\cap A_n^4},\mathcal{E}_{n})\right|=o(n).
\]
By (\ref{eq:4}) and the fact that $\mu^{\star n}$ has only exponentially
many atoms, we also have
\begin{eqnarray*}
|H((\mu^{\star n})_{\Gamma_{n}\cap A_n^4},\mathcal{D}^G_{1}) - H((\mu^{\star n})_{A_n^4},\mathcal{D}^G_{1})| & = & o(n),\\
|H((\mu^{\star n})_{\Gamma_{n}\cap A_n^4},\mathcal{E}_{n}) - H((\mu^{\star n})_{A_n^4},\mathcal{E}_{n})| & = & o(n).
\end{eqnarray*}
Combining these estimates, (\ref{eq:71}) becomes \[
  |\frac{1}{2\chi n}H((\mu^{\star n})_{A_n^4}, \mathcal{E}_n) - \alpha | < \eps,\;\;\;\textrm{ assuming } \mu^{\star n}(A_n^4)>c,
\]
and this is the statement of the previous corollary.
\end{proof}

\subsection{\label{subsec:Component-measures}Component measures }

For $x\in\mathbb{R}$ recall that $\mathcal{D}_{n}(x)$ denotes the unique
element of $\mathcal{D}_{n}$ containing it, and for a measure $\eta$
on $[0,1)$ (or $\RP$), define the level-$n$ component of $\eta$
at $x$ to be the conditional measure on $\mathcal{D}_{n}(x)$:
\[
\eta_{x,n}=\frac{1}{\eta(\mathcal{D}_{n}(x))}\eta|_{\mathcal{D}_{n}(x)}.
\]

We define components of a measure $\theta\in\mathcal{P}(G)$ in the
same way, using the dyadic partitions $\mathcal{D}_{n}^{G}$, so $\theta_{g,n}=\frac{1}{\theta(\mathcal{D}_{n}^{G}(g))}\theta|_{\mathcal{D}_{n}^{G}(g)}$.

\subsection{\label{subsec:Random-component-measures}Random component measures }

We often view $\eta_{x,n}$  as random variables,
with $n$ chosen uniformly within some specified range, and $x$ chosen
independently of $n$ according to $\eta$. This is the intention
whenever $\eta_{x,n}$ appears in an expression $\mathbb{P}(\ldots)$
or $\mathbb{E}(\ldots)$. For example, if $\mathcal{U}$ is a set
of measures then $\mathbb{P}_{0\leq i\leq n}(\eta_{x,i}\in\mathcal{U})$
is the probability that $\eta_{x,n}\in\mathcal{U}$ when $0\leq i\leq n$
is chosen uniformly and $x$ is independently chosen according to
$\eta$. 

Similarly, $\mathbb{E}_{i=n}(H(\eta_{x,i}, \Dk_{i+m}))$ denotes the expected entropy of a component at level $n$ (note that we took $i=n$, so the level is deterministic), measured at scale $n+m$, so  by definition,
\begin{equation} \label{ent:nova}
H(\eta,\Dk_{n+m}|\Dk_n) = \mathbb{E}_{i=n}(H(\eta_{x,i}, \Dk_{n+m})).
\end{equation}
As another
example, for every $n$ we have the trivial identity
\[
\eta=\mathbb{E}_{i=n}(\eta_{x,i}).
\]

We view components of measures on $G$ as random variables in the
same way as above and adopt the same notational conventions. 

When several random components are involved, they are assumed to be
chosen independently unless otherwise specified. Thus $\theta_{g,i}\times\eta_{x,i}$
is obtained by choosing $g$ and $x$ independently according to $\theta$
and $\eta$, respectively.

The distribution on components has the convenient property that it
is almost invariant under repeated sampling, i.e.\ choosing components
of components. More precisely, for a probability measure $\eta\in\mathcal{P}(\RP)$
and $m,n\in\mathbb{N}$, let $\mathbb{P}_{n}^{\eta}$ denote the distribution
of components $\eta_{x,i}$, $0\leq i\leq n$, as defined above; and
let $\mathbb{Q}_{n,m}^{\eta}$ denote the distribution on components
obtained by first choosing a random component $\eta_{x,i}$, $0\leq i\leq n$,
as above, and then, conditionally on $\theta=\eta_{x,i}$, choosing
a component $\theta_{y,j}$, $i\leq j\leq i+m$ with the usual distribution
(note that $\theta_{y,j}=\eta_{y,j}$ is indeed a component of $\eta$). 
\begin{lem}
\label{lem:distribution-of-components-of-components}Given $\eta\in\mathcal{P}(\RP)$
and $m,n\in\mathbb{N}$, the total variation distance between $\mathbb{P}_{n}^{\eta}$
and $\mathbb{Q}_{n,m}^{\eta}$ satisfies 
\[
\left\Vert \mathbb{P}_{n}^{\eta}-\mathbb{Q}_{n,m}^{\eta}\right\Vert =O(\frac{m}{n}).
\]
In particular, let $\mathcal{A}_i,\mathcal{B}_i\subseteq\mathcal{P}([0,1)^{d})$, write $\alpha = \Prob_{0\le i \le n}(\eta_{x,i}\in \Ak_i)$, and suppose that $\theta \in \Ak_i$ implies $\Prob_{i\le j \le i+m}(\theta_{x,j} \in \Bk_j) \ge \beta$. Then
\[
\mathbb{P}_{0\leq i\leq n}(\eta_{x,i}\in\mathcal{B}_i)>\alpha \beta-O(\frac{m}{n}).
\]
\end{lem}

These are essentially applications of the law of total probability, for details see \cite[Lemma 2.7]{Hochman2015}.

Similar statements hold for  components  of measures
on $G$. We leave these generalizations to the reader.


\subsection{\label{subsec:Reduction}Reduction of main theorem to two entropy inequalities}

Although it is possible to give an effective proof of Theorem \ref{thm:main}, the proof is most transparently presented by contradiction. The following proposition says that if the main theorem fails, then, when the measure $\mu^{\star n}$ is conditioned on typical dyadic cells of diameter $O(1)$, the resulting measures have two important properties: first, their entropy is substantial (it grows linearly up to a suitably chosen scale), and second, when the measures are convolved with the stationary measure $\nu$, then the result does not have substantially more entropy than $\nu$ we started out with. These properties will be seen in the next section to be incompatible with the multi-scale regularity of $\nu$.

\begin{prop}
\label{prop:reduction}If $\mu\in\mathcal{P}(G)$ is as in Theorem \ref{thm:main}, with $\supp \mu$ Diophantine, and if
the conclusion of Theorem \ref{thm:main} fails, then there are constants
$c,c'>0$ such that as $n\rightarrow\infty$,
\begin{eqnarray*}
\mathbb{E}_{i=1}\left(H((\mu^{\star n})_{g,i},\mathcal{D}_{cn}^{G})\right) & \geq & c'\cdot n-o(n),\\
\mathbb{E}_{i=1}\left(H((\mu^{\star n})_{g,i}\conv\nu,\mathcal{D}_{2\chi n+cn})\right) & \leq & H(\nu,\mathcal{D}_{cn})+o(n).
\end{eqnarray*}
\end{prop}

\begin{proof} Write $\alpha=\dim\nu$. We have seen
that $H(\nu,\mathcal{D}_{n})=(\alpha+o(1))n$. Therefore, for every
$c>0$,
\begin{eqnarray*}
H(\nu,\mathcal{D}_{2\chi n+cn}|\mathcal{D}_{2\chi n}) & = & H(\nu,\mathcal{D}_{2\chi n+cn})-H(\nu,\mathcal{D}_{2\chi n})\\
 & = & (\alpha+o(1))\cdot(2\chi n+cn)-(\alpha+o(1))\cdot2\chi n\\
 & = & (\alpha+o(1))\cdot cn\qquad\mbox{as }n\rightarrow\infty.
\end{eqnarray*}
On the other hand, by the identity $\mu^{\star n}=\mathbb{E}_{i=1}\left((\mu^{\star n})_{g,i}\right)$ and linearity of convolution, we have
\[
\nu=\mathbb{E}_{i=1}\left((\mu^{\star n})_{g,i}\conv\nu\right),
\]
so by concavity of entropy,
\begin{eqnarray*}
H(\nu,\mathcal{D}_{2\chi n+cn}|\mathcal{D}_{2\chi n}) & = & H\left(\mathbb{E}_{i=1}\left((\mu^{\star n})_{g,i}\conv\nu\right),\mathcal{D}_{2\chi n+cn}|\mathcal{D}_{2\chi n}\right)\\
 & \geq & \mathbb{E}_{i=1}\left(H((\mu^{\star n})_{g,i}\conv\nu,\mathcal{D}_{2\chi n+cn}|\mathcal{D}_{2\chi n})\right).
\end{eqnarray*}
Now, for $g_{n}$ chosen randomly according to $\mu^{\star n}$ we have
$\frac{1}{n}\log\left\Vert g_{n}\right\Vert \rightarrow\chi$
in probability. Consequently, as $n\rightarrow\infty$, the random
level-$1$ component $(\mu^{\star n})_{g,1}\conv \nu$ is supported on a set
of diameter $O(2^{-(2\chi+o(1))n})$ with probability tending to $1$, hence it intersects $2^{o(n)}$ cells of the partition $\Dk_{2\chi n}$.
It follows that
\[
\mathbb{E}_{i=1}\left(H((\mu^{\star n})_{g,i}\conv \nu,\mathcal{D}_{2\chi n+cn}|\mathcal{D}_{2\chi n})\right)=\mathbb{E}_{i=1}\left(H((\mu^{\star n})_{g,i}\conv \nu,\mathcal{D}_{2\chi n+cn})\right)-o(n),
\]
as $n\to\infty$.
Putting this all together, we conclude that
\begin{eqnarray*}
\mathbb{E}_{i=1}\left(H((\mu^{\star n})_{g,i}\conv \nu,\mathcal{D}_{2\chi n+cn})\right) & \leq & (\alpha+o(1))\cdot cn\\
 & = & H(\nu,\mathcal{D}_{cn})+o(n)\qquad\mbox{as }n\rightarrow\infty.
\end{eqnarray*}
This is the second statement in the proposition.

Now suppose that $\supp \mu$ is  Diophantine, and choose the constant $c$
above to be a constant such that $d(g_{1}\ldots g_{n},g'_{1}\ldots g'_{n})>2^{-cn}$
for all  pairs of sequences $g_{1},\ldots,g_{n}\in\supp\mu$
and $g'_{1},\ldots,g'_{n}\in\supp\mu$, with $g_1\ldots g_n \ne g_1'\ldots g_n'$. Then
each atom of $\mathcal{D}^G_{cn}$ contains $O(1)$ atoms of $\mu^{\star n}$, and
we have 
\[
H(\mu^{\star n},\mathcal{D}^G_{cn})=H(\mu^{\star n})-O(1)\qquad\mbox{as }n\rightarrow\infty.
\]
On the other hand, by Proposition \ref{prop:convergence-of-convolution-entropy-to-dimension},
\begin{eqnarray*}
  H(\mu^{\star n},\mathcal{D}^G_{1}) & = & 2\alpha\chi n-o(n)
  \qquad\mbox{as }n\rightarrow\infty.
\end{eqnarray*}
By the definition of the random walk entropy $h_{\RW} = h_{\RW}(\mu)$, we have $H(\mu^{\star n}) = h_{\RW}\cdot n + o(n)$.
Thus, assuming for the sake of contradiction that $\dim\nu<h_{\RW}/2\chi$,
and writing $c'=h_{\RW}-2\chi\alpha>0$, we conclude that
\begin{eqnarray*}
\mathbb{E}_{i=1}\left(H((\mu^{\star n})_{g,i},\mathcal{D}_{cn}^{G})\right) & = & H(\mu^{\star n},{\mathcal D}_{cn}^{G}|\mathcal{D}_{1}^{G})\\
 & = & (h_{\RW}-2\chi\alpha)\cdot n-o(n)\\
 & = & c'\cdot n-o(n)\qquad\mbox{as }n\rightarrow\infty.
\end{eqnarray*}
This is the first statement in the proposition.
\end{proof}

\section{\label{sec:Inverse-theorem,-linearization-etc}Inverse theorem, linearization,
and completion of the proof}

\subsection{\label{subsec:Multiscale-formulas-for-entropy}Multiscale formulas
for entropy }

A simple property of scale-$n$ entropy of a measure is that when
$m\ll n$ it is roughly equal to the average of the scale-$m$ entropies
of its components, and for convolutions a related bound can be given.
The proofs are similar to e.g. \cite[Lemmas 3.4 and 3.5]{Hochman2014}.
Below we write $\eta*\eta'$ for the convolution of measures on
$\mathbb{R}$ or $\RP$.
\begin{lem}
  \label{lem:multiscale-formula-for-entropy}
  \begin{enumerate}
\item[{\rm (i)}] For any $\eta\in\mathcal{P}(\RP)$, for every $m,n\in\mathbb{N}$, 
\begin{eqnarray*}
\frac{1}{n}H(\eta,\Dk_n) & = & \mathbb{E}_{1\leq i\leq n}(\frac{1}{m} H(\eta_{x,i}, \Dk_{i+m}))+O(\frac{m}{n}).
\end{eqnarray*}
\item[{\rm (ii)}] For any $\eta\in \Pk(G)$ with support of diameter $O(1)$,
\begin{eqnarray*}
\frac{1}{n} H(\eta,\Dk_n^G) & = & \mathbb{E}_{1\leq i\leq n}(\frac{1}{m} H_{m}(\eta_{x,i}, \Dk_{i+m}^G))+O(\frac{m}{n}).
\end{eqnarray*}
\end{enumerate}
\end{lem}
For convolutions in $\RP$ (or $\mathbb{R}$), we have a lower bound:
\begin{lem}
\textup{\em \label{lem:multiscale-formula-for-entropy-of-additive-convolution}For
any $\eta,\theta\in\mathcal{P}(\RP)$, for every $m,n\in\mathbb{N}$,
\begin{eqnarray*}
\frac{1}{n}H(\theta*\eta,\Dk_n) & \geq & \mathbb{E}_{1\leq i\leq n}(\frac{1}{m}H(\theta_{y,i}*\eta_{x,i}, \Dk_{i+m}))-O(\frac{1}{m}+\frac{m}{n}).
\end{eqnarray*}
}
\end{lem}
\noindent In the expectations above, the random variables $\eta_{x,i}$ and
$\theta_{y,i}$ are independent.

Before we state the analogous formula for convolutions $\theta\conv\eta$
where $\theta\in\mathcal{P}(G)$ and $\eta\in\mathcal{P}(\RP)$, we
note that when $g\in G$ acts on a measure $\eta\in\mathcal{P}(\RP)$,
it typically (i.e. unless $\eta$ gives substantial mass to a small
neighborhood of $u_{g}^{-}$) scales most of $\eta$ by $\left\Vert g\right\Vert ^{-2}$.
This implies that for large $i$, 
\[
H(\eta,\mathcal{D}_{i})\approx H(g\eta,\mathcal{D}_{i+2\log\left\Vert g\right\Vert })\approx H(S_{-2\log\left\Vert g\right\Vert }g\eta,\mathcal{D}_{i}).
\]
Thus if $\theta$ is supported near $g$, then we should measure the
entropy of $\theta\conv\eta$ at resolution $2\log\left\Vert g\right\Vert $-scales
smaller than that at which we consider $\eta$.

\begin{lem}
\label{lem:multiscale-formula-for-entropy-of-action-convolution}
Let $\theta\in\mathcal{P}(G)$ and $\eta\in\mathcal{P}(\RP)$. Suppose
that $\theta$ is supported on a set of diameter $O(1)$, let $g_{0}\in\supp\theta$
and set $\ell=2\log\left\Vert g_{0}\right\Vert $. Given $\delta>0$
let
\begin{align*}
E_{\delta} & =\left\{ (g,x)\in G\times\RP\,:\,x\notin B_\delta(u_{g}^{-})\right\}; \\
p_{\delta} & =1-\theta\times\eta(E_{\delta}).
\end{align*}
Then for integers $m<n$, 
\begin{align*}
\frac{1}{n}H(\theta\conv \eta\,,\,\mathcal{D}_{n+\ell}) & \geq\mathbb{E}_{0\leq i\leq n}\left(\frac{1}{m}H(\theta_{g,i}\conv \eta_{x,i}\,,\,\mathcal{D}_{i+\ell+m})\;\Big|\;(g,x)\in E_\delta\right)\\
 & \qquad-\,O_{\delta}\left(\frac{1}{m}+\frac{m}{n}\right)-\frac{p_{\delta}\ell+H(p_{\delta})}{n}\,.
\end{align*}
In particular, if $\eta$ is continuous, then 
\[
\frac{1}{n}H(\theta\conv \eta\,,\,\mathcal{D}_{n+\ell})\geq\mathbb{E}_{0\leq i\leq n}\left(\frac{1}{m}H(\theta_{g,i}\conv \eta_{x,i}\,,\,\mathcal{D}_{i+\ell+m})\;\Big|\;(g,x)\in E_\delta\right)-\,o_{\eta}(1)
\]
as $m,n/m\rightarrow\infty$, and the error term is uniform in $\theta$
assuming that $\int\log\left\Vert g\right\Vert d\theta(g)$ (equivalently,
$\ell$) grows at most linearly in $n$.
\end{lem}
\begin{proof}
Using the conditional entropy formula, for any measure $\sigma\in\mathcal{P}(\RP)$
we have (see proof of Lemma 3.4 in \cite{Hochman2014}):
\begin{align*}
\frac{1}{n}H(\sigma,\mathcal{D}_{n+\ell}) & =\frac{1}{n}H(\sigma,\mathcal{D}_{\ell})+\frac{1}{n}\sum_{i=1}^{n}\frac{1}{m}H(\sigma,\mathcal{D}_{i+\ell+m}|\mathcal{D}_{i+\ell})+O(\frac{m}{n}).
\end{align*}
Now, given
$i$, we have the identities $\theta=\mathbb{E}_{j=i}(\theta_{g,j})$
and $\eta=\mathbb{E}_{j=i}(\eta_{g,j})$. By linearity of the convolution
operation over the measures, we have $\theta\conv \eta=\mathbb{E}_{j=i}(\theta_{g,j}\conv \eta_{x,j})$,
so by concavity of entropy,
\begin{align*}
H(\theta\conv \eta,\mathcal{D}_{i+\ell+m}|\mathcal{D}_{i+\ell}) & =H\left(\mathbb{E}_{j=i}(\theta_{g,j}\conv \eta_{x,j}),\mathcal{D}_{i+\ell+m}|\mathcal{D}_{i+\ell}\right)\\
 & \geq\mathbb{E}_{j=i}\left(H(\theta_{g,j}\conv \eta_{x,j},\mathcal{D}_{i+\ell+m}|\mathcal{D}_{i+\ell})\right).
\end{align*}
Using this and substituting $\sigma=\theta\conv \eta$ into the previous
identity we get
\begin{multline}
H(\theta\conv \eta\,,\,\mathcal{D}_{n+\ell})\geq\label{eq:first-splitting}\\
\geq\;H(\theta\conv \eta,\mathcal{D}_{\ell})+\sum_{i=1}^{n}\mathbb{E}_{j=i}\left(\frac{1}{m}H(\theta_{g,j}\conv \eta_{x,j},\mathcal{D}_{i+\ell+m}|\mathcal{D}_{i+\ell})\right)+O(m).
\end{multline}
Now, since $\theta$ has diameter $O(1)$, there is an interval $I\subseteq\RP$
of length $O(\left\Vert g_{0}\right\Vert ^{-2})=O(2^{-\ell})$ such
that if $g\in\supp\theta$ and $x\notin B_{\delta}(u_{g}^{-})$,
then $gx\in I$. Write
\begin{align*}
\theta\conv \eta & =\int\delta_{gx}\,d\, \eta(x)\,d\theta(g)\\
 & =\int_{E_{\delta}}\delta_{gx}\,d\, \eta(x)\,d\theta(g)+\int_{\RP\setminus E_{\delta}}\delta_{gx}\,d\, \eta(x)\,d\theta(g).
\end{align*}
The first term has total mass $1-p_{\delta}=\theta\times\eta(E_{\delta})$
and is supported on $I$, and the second has total mass $p_{\delta}$.
Re-writing the last line as a convex combination $(1-p_{\delta})\sigma'+p_{\delta}\sigma''$,
with the measure  $\sigma'$ is supported on $I$, we have $H(\sigma',\mathcal{D}_{\ell})=O(1)$
and trivially $H(\sigma'',\mathcal{D}_{\ell})\leq\ell$. It follows
from almost-concavity of entropy (Lemma \ref{lem:entropy_prop} (ii)) that
\begin{align*}
H(\theta\conv \eta,\mathcal{D}_{\ell}) & \leq O(1)+p_{\delta}\ell+H(p_{\delta}).
\end{align*}
Thus, after dividing by $n$, the term $H(\theta\conv \eta,\mathcal{D}_{\ell})$
in (\ref{eq:first-splitting}) is absorbed in the error term of the
inequality we are trying to prove. It remains to analyze the sum of
expectations in (\ref{eq:first-splitting}). For this, fix $i$ and
condition on $E_{\delta}$:
\begin{multline}
\mathbb{E}_{j=i}\left(H(\theta_{g,j}\conv \eta_{x,j},\mathcal{D}_{i+\ell+m}|\mathcal{D}_{i+\ell})\right)\geq\\
\begin{aligned}\geq\quad & (1-p_\delta)\mathbb{E}_{j=i}\left(H(\theta_{g,j}\conv \eta_{x,j},\mathcal{D}_{i+\ell+m}|\mathcal{D}_{i+\ell})\;\left|\;(g,x)\in E_{\delta}\right.\right)\\
 & \;+p_{\delta}\mathbb{E}_{j=i}\left(H(\theta_{g,j}\conv \eta_{x,j},\mathcal{D}_{i+\ell+m}|\mathcal{D}_{i+\ell})\;\left|\;(g,x)\notin E_{\delta}\right.\right)\\
\geq\quad & \mathbb{E}_{j=i}\left(H(\theta_{g,j}\conv \eta_{x,j},\mathcal{D}_{i+\ell+m}|\mathcal{D}_{i+\ell})\;\left|\;(g,x)\in E_{\delta}\right.\right)-mp_{\delta}.
\end{aligned}
\label{eq:expectation-estimate}
\end{multline}
Finally, fixing $(g,x)\in E_{\delta}$, and assuming $i$
large relative to $\delta$, the component $\eta_{x,i}$ is supported
on the complement of $B_{\delta/2}(u_{g}^{-})$ and so $g$ acts
on its support by contracting by $\left\Vert g\right\Vert ^{-2}$
with distortion $O_{\delta}(1)$. Since $\eta_{x,i}$ is supported
on a set of diameter $2^{-i}$ the set $\theta_{g,i}\conv \eta_{x,i}$
is supported on an interval of diameter $O_{\delta}(2^{-i}\cdot\left\Vert g\right\Vert ^{-2})=O_{\delta}(2^{-(i+\ell)})$
so $H(\theta_{g,i}\conv \eta_{x,i},\mathcal{D}_{i+\ell})=O_{\delta}(1)$
and we can remove the conditioning in the entropy in the last line of (\ref{eq:expectation-estimate})
at a cost of $O_{\delta}(1)$.

Inserting what we have got into (\ref{eq:first-splitting}), the inequality
in the lemma is proved.

For the second statement we only need to show that $p_{\delta}\rightarrow0$
as $\delta\rightarrow0$. When $\eta$ is continuous, we have $\eta(B_{\delta}(u_{g}^{-}))\rightarrow0$
as $\delta\rightarrow0$ uniformly in $g$. Therefore by Fubini,
\[
p_{\delta}=\theta\times\eta(\RP\setminus E_{\delta})=\int\eta(B_{\delta}(u_{g}^{-}))d\theta(g)\rightarrow0\qquad\text{as }\delta\rightarrow0,
\]
as desired.
\end{proof}

\subsection{\label{subsec:Essentially-bounded-component-component-entropy}Entropy porosity}

One of the important properties of the stationary measure $\nu$ is
that most of its components have essentially the same entropy when
measured at a suitable scale. We will make use of this mainly through
the common upper bound on the entropy of components. 

For the following discussion consider a general probability measure
$\eta\in\mathcal{P}(\RP)$. We say that \emph{$\eta$ is $(h,\delta,m)$-entropy porous 
from scale $n_{1}$
to $n_{2}$ }if 
\begin{equation}
\mathbb{P}_{n_{1}\leq i\leq n_{2}}\left(\frac{1}{m}H(\eta_{x,i}, \Dk_{i+m})\leq h+\delta\right)>1-\delta. \label{eq:bounded-component-entropy}
\end{equation}
We say that \emph{$\eta$ is $h$-entropy porous}  if\footnote{This implies that $\limsup \frac{1}{n}H(\eta,\Dk_n)\leq h$, but the converse is
false in general. Entropy porosity is closely related to the notion of uniform entropy dimension from \cite{Hochman2014}, but that definition made more requirements and did not specify some of the parameters, which we want to be explicit about here.} for every $\delta>0$, $m>m(\delta)$ and $n>n(\delta,m)$ the measure is $(h,\delta,m)$-entropy porous
from
scale $0$ to $n$.

It turns out that entropy porosity passes to components. The next lemma appears in \cite{Hochman2016} verbatim, but we repeat it here for the reader's convenience.
\begin{lem}
\label{lem:bounded-component-entropy-passes-to-components}Let $0<\delta<1$,
$k\in\mathbb{N}$ and $n>n(\delta,k)$. If $\eta\in\mathcal{P}(\RP)$ is $(h,\delta^2/2,m)$-porous from scale $0$ to $n$, then 
\begin{equation}
\mathbb{P}_{0\leq i\leq n}\left(\begin{array}{c}
\eta_{x,i}\mbox{ is }\mbox{$(h,\delta,m)$-entropy}\\
\mbox{ porous from scale }i\mbox{ to }i+k
\end{array}\right)>1-\delta.\label{eq:bounded-component-entropy-of-components}
\end{equation}
\end{lem}
\begin{proof}
By assumption, 
\begin{equation}
\mathbb{P}_{0\leq i\leq n}\left(\frac{1}{m} H(\eta_{x,i}, \Dk_{i+m})\leq h+\frac{\delta^{2}}{2}\right)>1-\frac{\delta^{2}}{2}.\label{eq:10}
\end{equation}
Let $\mathcal{B}_i\subseteq\mathcal{P}(\RP)$ denote the set of measures
$\theta$ with $\frac{1}{m}H(\theta, \Dk_{i+m})>h+\delta$, and $\mathcal{A}_i\subseteq\mathcal{P}(\RP)$
the set of $\theta$ such that $\mathbb{P}_{i\leq j\leq i+k}(\theta_{x,j}\in\mathcal{B}_j)>\delta$.
It suffices for us to show that $\mathbb{P}_{0\leq i\leq n}(\eta_{x,i}\in\mathcal{A}_i)\leq2\delta/3$.
Indeed, if we had $\mathbb{P}_{0\leq i\leq n}(\eta_{x,i}\in\mathcal{A}_i)>2\delta/3$,
then Lemma \ref{lem:distribution-of-components-of-components} would
imply $\mathbb{P}_{0\leq i\leq n}(\eta_{x,i}\in\mathcal{B}_i)=2\delta^{2}/3-O(k/n)$,
which, assuming as we may that $n$ large relative to $k,\delta$,
contradicts (\ref{eq:10}).
\end{proof}

We note that entropy porosity is stable under translation and re-scaling in the following sense: For every $h,\delta>0$ and $m$ there exist $\delta'>0$ such that if $\eta$ is $(h,\delta',m)$-entropy porous from scale $0$ to $n$, then for any non-singular affine map $A$, the measure $A\eta$  is $(h,\delta,m)$-entropy porous at these scales, as long as $n$ is large compared to $A$. We shall not need this, though, and omit the proof.

Returning to our stationary measure $\nu$, from its ``approximate
self-similarity'' we show that
\begin{prop} \label{prop:approx}
The measure $\nu$ is $\alpha$-entropy porous, where $\alpha=\dim\nu$,
and in particular $\nu$ satisfies (\ref{eq:bounded-component-entropy-of-components}).
\end{prop}
\begin{proof}
Let $\delta>0$ and set 
\[
\varepsilon=\varepsilon(\delta)=\sup\{\nu(B_{\delta}({x}))\,:\,x\in\RP\}.
\]
Given $g\in G$ and $I_{g}^{\delta}=\RP\setminus B_{\delta}({u}_{g}^{-})$,
we know by Lemma \ref{lem:g-contracts-on-most-of-RP} that $g|_{I_{g}^{\delta}}$
scales by $\left\Vert g\right\Vert ^{-2}$ with distortion $O_{\delta}(1)$.
Write $\nu=\nu(I_{g}^{\delta})\cdot\nu_{I_{g}^{\delta}}+(1-\nu(I_{g}^{\delta}))\nu_{\RP\setminus I_{g}^{\delta}}$,
and note that $\nu(I_{g}^\delta)>1-\varepsilon(\delta)$. By Lemma
\ref{lem:entropy_prop}(ii), we have 
\[
\frac{1}{m} H(\nu_{I_{g}^{\delta}},\Dk_m)\geq\frac{1}{1-\varepsilon}\Bigl(\frac{1}{m}H(\nu,\Dk_m)-\varepsilon-\frac{1}{m}H(\varepsilon)\Bigr),
\]
so, assuming $m$ large relative to $\delta$, we have
\[
\frac{1}{m}H(\nu_{I_{g}^{\delta}},\Dk_m)=(1-O(\varepsilon))\frac{1}{m} H(\nu,\Dk_m).
\]
Similarly, writing 
$$
g\nu = \nu(I_g^\delta)\cdot g(\nu_{I_g^\delta}) + (1 - \nu(I_g^\delta))\cdot g(\nu_{\RP\setminus I_g^\delta}),
$$
 by concavity of entropy, the previous discussion and Lemmas \ref{lem:entropy_perturb} and \ref{lem:entropy_distort},
\begin{eqnarray*}
\frac{1}{m}H(g\nu,\mathcal{D}_{2\log\left\Vert g\right\Vert +m}\,|\,\mathcal{D}_{2\log\left\Vert g\right\Vert }) & \geq & \frac{1-\varepsilon}{m}H(g(\nu_{I_{g}^{\delta}}),\mathcal{D}_{2\log\left\Vert g\right\Vert +m}|\mathcal{D}_{2\log\left\Vert g\right\Vert })\\
 & = & \frac{1-\varepsilon}{m}H(\nu_{I_{g}^{\delta}},\mathcal{D}_{m})-O_{\delta}(\frac{1}{m}).
\end{eqnarray*}
In the last line we used the fact that $g(\nu_{I_g^\delta})$ is supported on a set of size $O(\|g\|^{-2})$, hence intersects $O_\delta(1)$ atoms of the partition $\Dk_{2\log\|g\|}$, and therefore,
$$
H(g(\nu_{I_g^\delta}), \Dk_{2\log\|g\|}) = O_\delta(1).
$$
Now, if we again assume $m$ large relative to $\delta$ we conclude that 
\begin{eqnarray*}
\frac{1}{m}H(g\nu,\mathcal{D}_{m+2\log\left\Vert g\right\Vert }\,| \,\Dk_{2\log\|g\|}) & \geq & (1-O(\varepsilon)) \frac{1}{m} H(\nu,\Dk_m).
\end{eqnarray*}
Since $\frac{1}{m} H(\nu,\Dk_m)=\alpha+o(1)$ as $m\rightarrow\infty$, for $m$
large enough, we obtain
\begin{equation} \label{nova1}
\frac{1}{m}H(g\nu,\mathcal{D}_{m+2\log\left\Vert g\right\Vert }\, |\, \Dk_{2\log\|g\|})\geq(1-O(\varepsilon))\alpha.
\end{equation}

Fix an $m$ as above and consider a fixed $i\in\mathbb{N}$. Let $X_{1},X_{2},\ldots$
be i.i.d.\ with marginal $\mu$ and define the stopping time $\tau$
to be the minimal $\tau\in\mathbb{N}$ such that $\left\Vert X_{1}X_{2}\ldots X_{\tau}\right\Vert \geq2^{i/2}$.
Note that $\left\Vert X_{1}X_{2}\ldots X_{k}\right\Vert \rightarrow\infty$
a.s., so $\tau$ is almost surely finite. It is an easy consequence
of stationarity that 
\[
\nu=\mathbb{E}(X_{1}X_{2}\ldots X_{\tau}\nu).
\]
Thus, by concavity of entropy,
\begin{eqnarray*}
\frac{1}{m}H(\nu,\mathcal{D}_{i+m}|\mathcal{D}_{i}) & = & \frac{1}{m}H\left(\mathbb{E}(X_{1}\ldots X_{\tau}\nu),\mathcal{D}_{i+m}|\mathcal{D}_{i}\right)\\
 & \geq & \mathbb{E}\left(\frac{1}{m}H(X_{1}\ldots X_{\tau}\nu,\mathcal{D}_{i+m}|\mathcal{D}_{i})\right)\\
 & \geq & (1-O(\varepsilon))\alpha.
\end{eqnarray*}
In the last line we used (\ref{nova1}), with $g=X_{1}X_{2}\ldots X_{\tau}$, so that
$$2\log\|g\|=2\log\left\Vert X_{1}X_{2}\ldots X_{\tau}\right\Vert =i+O(1)$$ by
the definition of $\tau$.

To conclude the argument, by Lemma \ref{lem:multiscale-formula-for-entropy} and (\ref{ent:nova}), for large enough $n$,
\begin{eqnarray*}
\alpha & \geq & \frac{1}{n} H(\nu,\Dk_n)-\varepsilon\\
 & = & \mathbb{E}_{1\leq i\leq n}(\frac{1}{m} H(\nu_{x,i}, \Dk_{i+m}))-O(\frac{m}{n})-\varepsilon\\
 & = & \frac{1}{n}\sum_{i=1}^{n}\frac{1}{m}H(\nu,\mathcal{D}_{i+m}|\mathcal{D}_{i})-O(\frac{m}{n})-\varepsilon,
\end{eqnarray*}
so for large enough $n$,
\[
\frac{1}{n}\sum_{i=1}^{n}\frac{1}{m}H(\nu,\mathcal{D}_{i+m}|\mathcal{D}_{i})\leq\alpha+2\varepsilon.
\]
This is bounded above by $\alpha+2\varepsilon$, but each term
in the average is bounded below by $\alpha-O(\varepsilon)$. Thus by
Markov's inequality,
\[
\frac{1}{n}\#\left\{ 1\leq i\leq n\,:\,\frac{1}{m}H(\nu,\mathcal{D}_{i+m}|\mathcal{D}_{i})>\alpha+\varepsilon'\right\} <\varepsilon',
\]
where $\varepsilon'=O(\sqrt{\varepsilon})$. But this just means that
\[
\mathbb{P}_{1\leq i\leq n}\left(\frac{1}{m}H(\nu_{x,i},\Dk_{i+m})\leq\alpha+\varepsilon'\right)>1-\varepsilon'.
\]
Since $\varepsilon$, and hence $\varepsilon'$, can be made arbitrarily
small by making $\delta$ small, this is what was claimed.
\end{proof}

\subsection{\label{subsec:Entropy-growth-under-conv-Euclidean}Entropy growth
under convolution: Euclidean case}

Recall that $\theta*\eta$ denotes the convolution of measures
on $\mathbb{R}$. The entropy of a convolution is generally at least
as large as each of the convolved measures, although due to the discretization
involved there may be a small loss:
\begin{lem}
\label{lem:entropy-monotonicity-underconvolution}For every $\eta,\theta\in\mathcal{P}([0,1))$,
\[
\frac{1}{n}H(\eta,\mathcal{D}_{n})-O(\frac{1}{n})\leq\frac{1}{n}H(\theta*\eta,\mathcal{D}_{n})\leq\frac{1}{n}H(\eta,\mathcal{D}_{n})+\frac{1}{n}H(\theta,\mathcal{D}_{n})+O(\frac{1}{n}).
\]
\end{lem}
Typically one expects the upper bound to be the correct one, but in
general this is not the case and one cannot rule out that the lower
bound is achieved, i.e. there is no entropy growth at all. It is quite
non-trivial to give useful conditions under which the upper bound is
achieved, but Theorem 2.8 of \cite{Hochman2014} provides a verifiable
condition under which at least some entropy growth occurs.
\begin{thm}
\label{thm:entropy-growth-under-convolution}For every $\varepsilon>0$
there exists $\delta=\delta(\varepsilon)>0$ such that for every
$m>m(\varepsilon,\delta)$ and $n>n(\varepsilon,\delta,m)$, the following
holds: 

Let $\eta\in\mathcal{P}([0,1))$ and $\theta\in\mathcal{P}([0,1))$
be probability measures and suppose that $\eta$ is $(1-\eps,\delta,m)$-entropy porous from scale $1$
to $n$. Then 
\[
\frac{1}{n} H(\theta,\Dk_n)>\varepsilon\quad\implies\quad \frac{1}{n}H(\theta*\eta,\Dk_n)>\frac{1}{n} H(\eta,\Dk_n)+\delta.
\]
More generally, if $\eta$ is $(1-\eps,\delta,m)$-entropy porous from scale $n_1$ to $n_2$ and $n_2-n_1 > n(\eps,\delta,m)$, then 
\begin{multline}
 \frac{1}{n_2-n_1} H(\theta,\mathcal{D}_{n_{2}}|\mathcal{D}_{n_{1}})>\varepsilon\quad \\
  \implies\quad H(\theta*\eta,\mathcal{D}_{n_{2}}|\mathcal{D}_{n_{1}})>H(\eta,\mathcal{D}_{n_{2}}|\mathcal{D}_{n_{1}})+\delta\cdot(n_{2}-n_{1}).
\end{multline}
\end{thm}

\subsection{Linearization}

We require an analogous statement for convolutions $\theta\conv\eta$,
where $\eta\in\mathcal{P}(\RP)$ and $\theta\in\mathcal{P}(G)$. The
main idea, similar to that in \cite{Hochman2015}, is to reduce the
``non-linear convolution'' $\theta\conv\eta$ to an average of linear
convolutions involving (images of) components of $\theta$ and $\eta$,
to which, with an additional argument, we can apply the Euclidean
theorem above. Our argument is then based on linearization\footnote{The ``correct'' way to linearize is to use coordinates; concretely,
lift\textbf{ $B_{r}(g_{0})\times B_{r}(x_{0})$} to the tangent space
$T(G\times\RP)=TG\times T\RP$ via the logarithmic map, incurring
some distortion, apply the derivative map, and project
back down to $\RP$. However this introduces some unnecessary complications,
and instead we use a more elementary approach relying on the simple
angle-based coordinates in $\RP$.} of the action map $G\times\RP\rightarrow\RP$. 
We interpret sums and differences of elements in $\RP$, and also
their products with scalars, using the identification of $\RP$ with
$[0,1)\subseteq\mathbb{R}$. They are well defined as long as all
terms are nearby in $\RP$.
Assume that $B_{1}\times B_{2}\subseteq G\times\RP$
is a product of $r$-balls and $(g_{0},x_{0})\in B_{1}\times B_{2}$,
and  $x_{0}$ is far enough from $u_{g_{0}}^{-}$. 
For $x_{0}\in\RP$ write $\theta\conv x_{0}$ for the push-forward of $\theta$
by $g\mapsto gx_{0}$, and for $g_{0}\in G$ we write $g_{0}\eta$
for the push-forward of $\eta$ by $g_{0}$ (so $\theta\conv x_{0}=\theta\conv\delta_{x_{0}}$
and $g_{0}\eta=\delta_{g_{0}}\conv\eta$). Then some calculus shows that
on $B_{1}\times B_{2}$ the map $(g,x)\mapsto gx$ is very close, up to translation, to
the map $(g,x)\mapsto(gx_{0})+g_{0}x$, which in turn (using that
$x_{0}$ is far from $u_{g_{0}}^{-}$) is close to $(g,x)\mapsto gx_{0}+\left\Vert g_{0}\right\Vert ^{-2}x$.
In particular, since the convolution $\theta \conv \eta$ of $\theta\in\mathcal{P}(B_{1})$
and $\eta\in\mathcal{P}(B_{2})$ is the image of $\theta\times\eta$
under $(g,x)\mapsto gx$, we find that $\theta\conv \eta$ will be close
enough to the image of $\theta\times\eta$ under the map $(g,x)\mapsto(gx_{0})+\left\Vert g_{0}\right\Vert ^{-2}x$,
which is just $(\theta\conv x_{0})*(S_{-2\log\lambda_{g_{0}}}\eta)$.
This will imply that their entropies, at suitably small scales, are
close. 

\begin{lem}
\label{lem:linearization-at-the-identity}For $0<r<1$, $h\in B_{r}(1_{G})$
and $x\in B_{r}(x_{0})$ we have
\[
hx-hx_{0}=x-x_{0}+O(r^{2}).
\]
\end{lem}
\begin{proof}
For $h_{0}=1_{G}$ the map $\widehat{h}_{0}$ is the identity on $\RP$,
so its first derivative is identically $1$ and its second vanishes.
Thus for $h\in B_{r}(h_{0})$, writing $h$ in local coordinates around
$h_{0}=1_{G}$, we find that $\left\Vert \widehat{h}'\right\Vert _{\infty}=1+O(r)$
and $\left\Vert \widehat{h}''\right\Vert _{\infty}=O(r)$, and by
first-order approximation of $\widehat{h}$ at $x_{0}$, for $x\in B_{r}(x_{0})$
we have
\begin{eqnarray*}
hx & = & hx_{0}+\widehat{h}'(x_{0})(x-x_{0})+O((x-x_{0})^{2})\\
 & = & hx_{0}+(1+O(r))(x-x_{0})+O((x-x_{0})^{2})\\
 & = & hx_{0}+(x-x_{0})+O(r^{2}),
\end{eqnarray*}
as claimed.
\end{proof}
\begin{cor}
\label{cor:linearization-at-the-identity}For $0<r<1$, $h\in B_{r}(1_{G})$
and $x\in B_{r}(x_{0})$ we have $hx_{0}-x_{0}=O(r)$ and $hx-x_{0}=O(r)$.
\end{cor}
\begin{prop}
\label{prop:linearization}Let $0<\varepsilon<1$, let $0<r<r(\varepsilon)$,
and let $(g,x)\in B_{r}(g_{0})\times B_{r}(x_{0})\subset G\times\RP$.
Assume that $x_{0},x\notin B_{\varepsilon}(u_{g_{0}}^{-})$. Then
\begin{eqnarray*}
gx & = & gx_{0}+g_{0}x-g_{0}x_{0}+O_{\varepsilon}(\frac{r^{2}}{\lambda_{g_{0}}^{2}})\\
     & = & gx_{0}+\widehat{g}'_{0}(x_{0})(x-x_{0})+O_{\varepsilon}(\frac{r^{2}}{\lambda_{g_{0}}^{2}})
\end{eqnarray*}
where addition and scaling in $\RP$ is performed via the identification $\RP=[0,1)$.
\end{prop}
\begin{proof}
Write $g=g_{0}h$ with $h\in B_{r}(1_{G})$ and observe that since
$r$ is small relative to $\varepsilon$ and $x_{0}\notin B_{\varepsilon}(u_{g_{0}}^{-})$
we can assume that $hx_{0}\notin B_{\varepsilon/2}(u_{g_{0}}^{-})$;
and since $d_{\RP}(x,x_{0})<r$ we can also assume that $x,hx\notin B_{\varepsilon/2}(u_{g_{0}}^{-})$.
Using the first order approximation of $\widehat{g}_{0}$ at $x_{0}$
(Lemma \ref{lem:g-contracts-on-most-of-RP}) and using $hx-hx_{0}=x-x_{0}+O(r^{2})$
(Lemma \ref{lem:linearization-at-the-identity}) and $hx_{0}-x_{0}=O(r)$
and $hx-x_{0}=O(r)$ (Corollary \ref{cor:linearization-at-the-identity}),
we have
\begin{eqnarray*}
gx-gx_{0} & = & \widehat{g}_{0}(hx)-\widehat{g}_{0}(hx_{0})\\
 & = & \widehat{g}'_{0}(x_{0})(hx-hx_{0})+O_{\varepsilon}(\frac{(hx-x_{0})^{2}+(hx_{0}-x_{0})^{2}}{\lambda_{g_{0}}^{2}})\\
 & = & \widehat{g}'_{0}(x_{0})(x-x_{0})+O_{\varepsilon}(\frac{r^{2}}{\lambda_{g_{0}}^{2}}).
\end{eqnarray*}
This gives the second claim. Applying the calculation above with $g=g_{0}$
we conclude that $g_{0}x-g_{0}x_{0}=\widehat{g}'_{0}(x_{0})(x-x_{0})+O_{\varepsilon}(r^{2}/\lambda_{g_{0}}^{2})$.
Substituting this into the second equation in the statement gives
the first. 
\end{proof}
We endow $\mathcal{P}(\RP)$ with the compatible metric given by
\[
\varrho(\eta,\eta')=\sup\left\{ \left|\int fd\eta-\int fd\eta'\right|\,:\,f:\RP\rightarrow\mathbb{R}\mbox{ is }1\mbox{-Lipschitz}\right\}. 
\]

\begin{cor}
\label{cor:linearization-and-entropy}Let $0<\varepsilon<1$ and $0<r<r(\varepsilon)$.
Let $\theta\in\mathcal{P}(G)$ and $\eta\in\mathcal{P}(\RP)$ be supported
on sets of diameter $O(r)$ and let $g_{0}\in\supp\theta$ and $x_{0}\in\supp\eta$.
Assuming that $x_{0}\notin B_{\varepsilon}(u_{g_{0}}^{-})$, there
are $u,v\in\mathbb{R}$ such that 
\begin{eqnarray*}
\varrho\left(\theta\conv\eta\,,\,T_{u}((\theta\conv x_{0})*(g_{0}\conv\eta))\right) & = & O_{\varepsilon}(\frac{r^{2}}{\lambda_{g_{0}}^{2}}),\\
\varrho\left(\theta\conv\eta\,,\,T_{v}((\theta \conv x_{0})*(S_{\log\widehat{g}'_{0}(x_{0})}\eta))\right) & = & O_{\varepsilon}(\frac{r^{2}}{\lambda_{g_{0}}^{2}}).
\end{eqnarray*}
Moreover, given $n>n(\varepsilon)$, for 
$k=-\log r+2\log\lambda_{g_0}$, we have that $\theta\conv\eta$ is supported
on a set of diameter $O_\eps(2^{-k})$ and for $0< r < r(\eps,n)$,
\[
H(\theta\conv\eta,\mathcal{D}_{k+n})=H((\theta\conv x_{0})*(S_{\log\widehat{g}'_{0}(x_{0})}\eta),\mathcal{D}_{k+n})+O_{\varepsilon}(1).
\]
\end{cor}
\begin{proof}
The first two statements are immediate from the previous proposition. We interpret sums and differences of elements in $\RP$, and also
their products with scalars, using the identification of $\RP$ with
$[0,1)\subseteq\mathbb{R}$. They are well defined as long as all
terms are nearby in $\RP$, with $u=g_0x_0$ and $v=\widehat{g}_0'(x_0)x_0$.
Next, since $r$ is small relative to $\eps$ we can assume that $x\not\in B_{\eps/2}(u_g^-)$ for all $g\in \supp \theta$ and $x\in \supp\eta$. Thus, the action map $(g,x)\mapsto gx$ on the product of the supports scales by $\Theta_\eps(\lam_g^{-2})$, and since $\lam_g = \Theta(\lam_{g_0})$ for $g\in \supp\theta$, we conclude that $\theta\conv\eta$ is supported
on a set of diameter $O_\eps(2^{-k})$.

 For
the last statement, note that since $\theta\conv\eta$ is the image of
$\theta\times\eta$ under the action map $(g,x)\mapsto gx$, and $(\theta\conv x_{0})*(S_{\log\widehat{g}'_{0}(x_{0})}\eta)$
is its image under $(g,x)\mapsto gx_{0}+S_{\log\widehat{g}'_{0}(x_{0})}x$, so by
Lemma \ref{lem:entropy_perturb} it is enough to show that after a translation these
two maps differ uniformly (on the support of $\theta\times \eta$) by at most $O(2^{-(k+n)})$. In fact, by
the previous proposition, up to a translation, these maps differ from each other
 by
$O_{\varepsilon}(\widehat{g}'_{0}(x_{0})r^{2})$, and by Lemma \ref{lem:g-contracts-on-most-of-RP}
and the assumption $x_{0}\notin B_{\varepsilon}(u_{g_{0}}^{-})$
we know that $\widehat{g}'_{0}(x_{0})=\Theta_{\varepsilon}(\lambda_{g_{0}}^{-2})$.
Thus, after translation, these maps differ uniformly, on a set of diameter $O(r)$, by at most $$O_\eps\Bigl(\frac{r^2}{\lam_{g_0}^2}\Bigr)=O_\eps(r\cdot 2^{-k}) = O_\eps(2^{-(k+n)}),$$
assuming $r\le 2^{-n}$. This proves the claim.
\end{proof}

\subsection{Entropy growth under convolution: $G$-action }

We now study the growth of entropy for convolutions $\theta\conv\eta$.
We begin with some elementary observations about the scale at which
we should consider this measure. If the support of $\theta$ in $G$
is of diameter $O(1)$, then all elements $g\in\supp\theta$ have
the same norm up to a bounded multiplicative constant, and they contract
(most of) $\RP$ by a factor of roughly $O(\left\Vert g\right\Vert ^{2})$.
Therefore, if $\eta$ does not give too much mass to the exceptional
part of $\RP$ (which is automatic if $\eta$ is non-atomic and we
work at a small enough scale), then most of the mass of $\theta\conv\eta$ will be supported
on a set of diameter $O(\left\Vert g\right\Vert ^{-2})$. Thus properties
of $\theta,\eta$ that are observed at scale $n$ translate to properties
of $\theta\conv\eta$ at scale $2\log\left\Vert g\right\Vert +n$.
\begin{lem}
\label{lem:3-point-separation-and-entropy}Let $\eta\in\mathcal{P}(\RP)$
and let $\rho,\delta>0$ be such that $\eta(B)<\delta$ for every
ball $B\subseteq\RP$ of radius $\rho$. Then for $0<r<r(\rho)$ and
every measure $\theta\in\mathcal{P}(G)$ supported on a set of diameter
$r$, writing $\ell= 2\log\left\Vert g_{0}\right\Vert$ for some $g_{0}\in\supp\theta$,
for every $m$ we have 
\[
\eta\Bigl({x}\in\RP\,:\,H(\theta\conv{x},\mathcal{D}_{\ell+m})>\frac{1}{3}H(\theta,\mathcal{D}^G_{m})-O_{\rho}(1)\Bigr)>1-4\delta.
\]
\end{lem}
\begin{proof}
By Lemma \ref{lem:separation-of-G-by-3-orbits}, there are constants
$r(\rho)>0$ and $C=C(\rho)>0$ such that for every $0<r<r(\rho)$
and $B=B_{r}(g_{0})$, for every $\rho$-separated triple $x=( x_{1}, x_{2}, x_{3})\in(\RP\setminus B_{\rho}(u_{g_{0}}^{-}))^{3}$,
the restriction to $B$ of the map $f_{ x}:g\mapsto(g x_{1},g x_{2},g x_{3})$
scales by $2^{-\ell}$ with distortion $C$, and hence by Lemma 
\ref{lem:entropy_distort} there is a constant $C'$, depending only on $C$, such that 
\[
H(f_{ x}\theta,\mathcal{D}_{\ell+m})\geq H(\theta,\mathcal{D}^G_{m})-C'.
\]
Since the partition $\mathcal{D}_{\ell+m}$ on $(\RP)^{3}$ is just
$\bigvee\mathcal{D}_{\ell+m}^{i}$, where $\mathcal{D}_{\ell+m}^{i}=\pi_{i}^{-1}\mathcal{D}_{\ell+m}$
is the pullback of $\mathcal{D}_{\ell+m}$ by the projection $\pi_{i}:(\RP)^{3}\rightarrow\RP$
to the $i$-th coordinate, we have
\begin{eqnarray*}
  H(f_{ x}\theta,\mathcal{D}_{\ell+m})  & = & H(f_{ x}\theta,\mathcal{D}_{\ell+m}^{1}) + H(f_{ x}\theta,\mathcal{D}_{\ell+m}^{2}|\mathcal{D}_{\ell+m}^{1}) \\
    & & \qquad +\; H(f_{ x}\theta,\mathcal{D}_{\ell+m}^{3}|\mathcal{D}_{\ell+m}^{1}\lor\mathcal{D}_{\ell+m}^{2})\\
 & \leq & \sum_{i=1}^{3}H(f_{ x}\theta,\mathcal{D}_{\ell+m}^{i})\\
 & = & \sum_{i=1}^{3}H(\pi_{i}f_{ x}\theta,\mathcal{D}_{\ell+m}),
\end{eqnarray*}
and since $\pi_{i}f_{ x}\theta=\theta\conv  x_{i}$, this and the previous
inequality imply that
\begin{equation}
H(\theta\conv  x_{i},\mathcal{D}_{\ell+m})\geq\frac{1}{3}H(\theta,\mathcal{D}^G_{m})-C'\qquad\mbox{for some }i\in\{1,2,3\}.\label{eq:8}
\end{equation}
Now let
\[
E=\Bigl\{x\in\RP\,:\,H(\theta\conv  x,\mathcal{D}_{\ell+m})<\frac{1}{3} H(\theta,\mathcal{D}^G_{m})-C'\Bigr\}.
\]
We must show that $\eta(E)<4\delta$. Indeed, if this were not the case
then, since $B_{\rho}(u_{g_{0}}^{-})<\delta$ by the choice of $\rho,\delta$,
we would have $\eta(E\setminus B_{\rho}(u_{g_{0}}^{-}))\geq3\delta$.
Again by the properties of $\rho,\delta$ this means we can choose
$( x_{1}, x_{2},  x_{3})\in(E\setminus B_{\rho}(u_{g_{0}}^{-}))^{3}$
which are $\rho$-separated. But now these points are in $E$ but
satisfy (\ref{eq:8}), which is a contradiction. Thus $\eta(E)<4\delta$,
as claimed.
\end{proof}
We note that if $\eta$ is a fixed and non-atomic measure and $\delta>0$
is given, then there always exists a $\rho>0$ such that the hypothesis
of the previous lemma is satisfied. 

\begin{thm}
  \label{thm:entropy-growth-under-action}For every $\varepsilon>0$ there exists $\delta_{1}(\varepsilon)>0$, such that for every $0<\delta<\delta_1(\eps)$, $\rho>0$,  $m\geq m(\varepsilon,\delta,\rho)$ and $N\geq N(\varepsilon,\delta,\rho,m)$, the following holds.

  Suppose that $\eta\in\mathcal{P}(\RP)$ and $\theta\in\mathcal{P}(G)$ satisfy
\begin{enumerate}
\item[{\rm (i)}] The measure $\eta$ is $(1-\eps,\delta,m)$-entropy porous from scale $1$ to $N$;
\item[{\rm (ii)}] $\eta(B)<\delta^2$ for every ball $B\subseteq\RP$ of radius $\rho$;
\item[{\rm (iii)}] $\theta$ has support of diameter at most $1/\eps$.
\end{enumerate}
Let $g_0\in\supp\theta$ and set $\ell=2\log\left\Vert g_0\right\Vert$. Then
\[
\frac{1}{N}H(\theta,\mathcal{D}_{N}^{G})>\varepsilon\quad\implies\quad\frac{1}{N}H(\theta\conv\eta,\mathcal{D}_{N+\ell})>\frac{1}{N}H(\eta,\mathcal{D}_{N})+\delta-\frac{\ell}{N}\delta^2.
\]
\end{thm}
\begin{rem}
The role of $\rho$ in the theorem is to quantify the continuity of
the measure $\eta$. If we change the order of quantifiers and fix
a non-atomic measure $\eta$ in advance, then for given $\delta$
there is a $\rho>0$ satisfying (ii), so the conclusion holds (uniformly
in $\theta$) as soon as $m$, $N$ are suitably large (in a manner
that now depends on $\rho$, hence $\eta$).

The role of $\ell$ in the theorem is to control the amount of expansion/contraction of elements in the support of $\theta$. Note that as long as $\ell=O(N)$, the entropy increment $\delta-\delta^2\ell/N$ will be positive as long as $\delta$ is small enough (and $\rho,m,N$ corresponding to it). 
\end{rem}
\begin{proof}
Recall from Section \ref{subsec:Dyadic-partitions-and-entropy} that $M>0$ is a constant such that every level-$i$ dyadic cell in $G$ contains at most $M$ level-($i+1$) sub-cells. In particular, if $\theta\in\mathcal{P}(G)$ is supported on a level-$i$ cell then $\frac{1}{m}H(\theta,\mathcal{D}^G_{i+m})\leq \log M$. We may assume without loss of generality that $\log M\geq 1$. 

Let $\delta>0$ be small enough that \[
\delta'=(40\log M)\frac{\sqrt{\delta}}{\varepsilon}
\]
satisfies the conclusion of Theorem \ref{thm:entropy-growth-under-convolution}
when applied with parameter \[
  \varepsilon'=\frac{\varepsilon}{26}
  \]
(instead of $\varepsilon$ 
in the theorem). Let $m,n$ be large enough to satisfy the conclusion
of that theorem, as well as additional conditions we will see below.
Assume that $N$ is large with respect to the other parameters. 

Let $\eta,\theta$ satisfy (i)-(iii). 
Write $\mathbb{P}^{N}$ for the distribution on independently chosen
pairs of components $(\theta_{g,i},\eta_{x,i})$ with $0\leq i<N$,
and let
\[
\mathcal{E}_{1}=\left\{ (\theta_{g,i},\eta_{x,i})\,:\,x\notin B_{\rho}(u_{g}^{-})\right\}.
\]
We remark that we are slightly abusing notation here: The pair of measures $(\theta_{g,i},\eta_{x,i})$ in general does not determine $g,x$ and $i$. However, we extend our notion of component so that $\theta_{g,i}$ is viewed as the triple $(\theta,g,i)$ and similarly with $\eta_{x,i}$, and view $\mathbb{P}^N$ as a distribution on pairs of such triples. We will continue with this convention later and allow ourselves to refer to $g,x,i$ when given components as above.

\begin{lem} $$\mathbb{P}^{N}(\mathcal{E}_{1})>1-\delta^2.$$
\end{lem}

\begin{proof}
For every component $\theta_{g,i}$ of $\theta$ we have by assumption
$\eta(B_{\rho}(u_{g}^{-}))<\delta^2$, so the conclusion follows from Fubini.
\end{proof}

\begin{lem} \begin{eqnarray}
    \frac{1}{N}H(\theta\conv\eta,\mathcal{D}_{N+\ell}) & \geq & \mathbb{E}_{0\leq i\leq N}\left(\frac{1}{n}H(\theta_{g,i}\conv\eta_{x,i},\mathcal{D}_{\ell+i+n})\;\Big|\;\mathcal{E}_1\right) \nonumber \\
    & & \quad -\;O\Bigl(\frac{1}{n}+\frac{n}{N}\Bigr) - \frac{\ell}{N}\delta^2.
\label{eq:211}
\end{eqnarray}
\end{lem}
\begin{proof}
This is immediate from Lemma \ref{lem:multiscale-formula-for-entropy-of-action-convolution}. Using the notation of the lemma and setting \[
   p_\rho=\theta\times\eta({(g,x)\in G\times \RP\,:\,x\in B_\rho(u_g^-)})
\]
we have (by the argument in the previous lemma) that $p_\rho\leq \delta^2$, and the trivial bound gives $H(p_\rho)\leq 2$ so the error term $H(p_\rho)/N$ in that lemma is absorbed in the  error term of (\ref{eq:211}).
\end{proof}

For $g\in G$ and $x\in\RP$ define \[
  \ell(g,x) = \log\widehat{g}'(x).
\]

\begin{lem}
  Assuming that $n$ is large enough relative to $\varepsilon,\delta$,
  \begin{multline}
    \mathbb{E}_{0\leq i\leq N}\left(\frac{1}{n}H(\theta_{g,i}\conv \eta_{x,i},\mathcal{D}_{\ell+i+n})\;\Big|\;\mathcal{E}_1\right)\\
      \geq \mathbb{E}_{0\leq i\leq N}\left(\frac{1}{n}H\left((S_{-\ell(g,x)}\theta_{g,i}\conv x)*\eta_{x,i},\mathcal{D}_{i+n}\right)\;\Big|\;\mathcal{E}_1\right)-\delta^2.
  \end{multline}   
\end{lem}

\begin{proof}
It is enough to prove the inequality with an $O(1)$ error, but without the $1/n$ factors. For this, first observe that for all $g\in \supp\theta$ and $x\in \supp\eta$,
\begin{equation*}
H(\theta_{g,i}\conv\eta_{x,i},\mathcal{D}_{\ell+i+n}) = H(\theta_{g,i}\conv\eta_{x,i},\mathcal{D}_{2\log\|g\|+i+n})	+ O(1),
\end{equation*}
since $\|g\| = \Theta(\|g_0\|)$. 
By Corollary \ref{cor:linearization-and-entropy}, there is an $i_{0}\in\mathbb{N}$,
depending on $\delta$, such that for $i>i_{0}$ and $(\theta_{g,i},\eta_{x,i})\in\mathcal{E}_{1}$,
\begin{multline*}
  H(\theta_{g,i}\conv\eta_{x,i},\mathcal{D}_{2\log\|g\|+i+n})\;=\\
  =\;H\left((\theta_{g,i}\conv x)*(S_{\ell(g,x)}\eta_{x,i}),\mathcal{D}_{2\log\|g\|+i+n}\right)-O_{\varepsilon}(1)
\end{multline*}
(applying the corollary with $r=2^{-i}$). We can further replace the partition $\Dk_{2\log\|g\|+i+n}$ in the right-hand side with $\Dk_{\ell(g,x) + i + n}$, absorbing the difference into the error. Finally, by re-scaling we have \begin{multline*}
 H\left((\theta_{g,i}\conv x)*(S_{\ell(g,x)}\eta_{x,i}),\mathcal{D}_{\ell(g,x)+i+n}\right) \;=\\=\; H\left((S_{-\ell(g,x)}\theta_{g,i}\conv x)*\eta_{x,i},\mathcal{D}_{i+n}\right) - O(1).
\end{multline*}
Thus in the event $\mathcal{E}_1$, the desired inequality holds pointwise, and hence in expectation.
\end{proof}

In order to evaluate the lower bound obtained in the last lemma, define
\[
\mathcal{E}_{2}=\left\{ (\theta_{g,i},\eta_{x,i})\in\mathcal{E}_1\,:\,\begin{array}{c}
\frac{1}{n}H\left((S_{-\ell(g,x)}\theta_{g,i}\conv x)*\eta_{x,i},\mathcal{D}_{i+n}\right)\\
\qquad>\quad\frac{1}{n}H(\eta_{x,i},\mathcal{D}_{i+n})+\delta'
\end{array}\right\} 
\]

\begin{lem}
  For $n$ large enough and $N$ large enough relative to $n$,
\begin{multline*}
  \mathbb{E}_{0\leq i\leq N}\left(\frac{1}{n}H\left((S_{-\ell(g,x)}\theta_{g,i}\conv x)*\eta_{x,i},\mathcal{D}_{i+n}\right)\;\Big|\;\mathcal{E}_1\right)\\
\geq\; \frac{1}{N}H(\eta,\mathcal{D}_{N})+\delta'\cdot\mathbb{P}^{N}(\mathcal{E}_{2})-\delta^2.
\end{multline*}
\end{lem}

\begin{proof}
By Lemma \ref{lem:entropy-monotonicity-underconvolution}, for any pair of components
$(\theta_{g,i},\eta_{x,i})$ (even not in $\mathcal{E}_{2}$), we have the trivial
bound 
\begin{multline}
  \frac{1}{n}H\left((S_{-\ell(g,x)}\theta_{g,i}\conv x)*\eta_{x,i},\mathcal{D}_{i+n}\right)\;
  \geq\;\frac{1}{n}H(\eta_{x,i},\Dk_{i+n})-O(\frac{1}{n}).
\end{multline}
Thus, conditioning the expectation on $\mathcal{E}_2$ and its complement in $\mathcal{E}_1$, 
\begin{multline*}
\mathbb{E}_{0\leq i\leq N}\left(\frac{1}{n}H\left((S_{-\ell(g,x)}\theta_{g,i}\conv x)*\eta_{x,i},\mathcal{D}_{i+n}\right)\;\Big|\;\mathcal{E}_1\right)\\
\begin{aligned}
  = & \quad \mathbb{P}^N(\mathcal{E}_2\;\big|\;\mathcal{E}_1)\cdot\mathbb{E}_{0\leq i\leq N}\left(\frac{1}{n}H\left((S_{-\ell(g,x)}\theta_{g,i}\conv x)*\eta_{x,i},\mathcal{D}_{i+n}\right) \Big| \mathcal{E}_2\right)\\
  & \qquad + \; \mathbb{P}^N(\mathcal{E}_1\setminus\mathcal{E}_2\;\big|\;\mathcal{E}_1)\cdot\mathbb{E}_{0\leq i\leq N}\left(\frac{1}{n}H\left((S_{-\ell(g,x)}\theta_{g,i}\conv x)*\eta_{x,i},\mathcal{D}_{i+n}\right) \Big| \mathcal{E}_1\setminus\mathcal{E}_2\right)\\
  \geq & \quad \mathbb{P}^N(\mathcal{E}_2)\cdot\left(\mathbb{E}_{0\leq i\leq N}\left(\frac{1}{n}H\left(\eta_{x,i},\mathcal{D}_{i+n}\right)\Big| \mathcal{E}_2\right)+\delta'\right)\\
    & \qquad + \; \mathbb{P}^N(\mathcal{E}_1\setminus\mathcal{E}_2)\cdot\mathbb{E}_{0\leq i\leq N}\left(\frac{1}{n}H\left(\eta_{x,i},\mathcal{D}_{i+n}\right) \Big| \mathcal{E}_1\setminus\mathcal{E}_2\right)-O(\frac{1}{n})\\
  \geq & \quad\mathbb{E}_{0\leq i\leq N}\left(\frac{1}{n}H\left(\eta_{x,i},\mathcal{D}_{i+n}\right)\right) +\delta'\cdot\mathbb{P}^{N}(\mathcal{E}_{2})-O(\frac{1}{n})(1-\mathbb{P}^{N}(\mathcal{E}_{2}))\\
= & \quad\frac{1}{N}H(\eta,\mathcal{D}_{N})+\delta'\cdot\mathbb{P}^{N}(\mathcal{E}_{2})-O(\frac{1}{n}+\frac{n}{N}).
\end{aligned}
\end{multline*}
In the last line we used the multiscale entropy formula from Lemma~\ref{lem:multiscale-formula-for-entropy}(i). Assuming that $n$ and $N$ are suitably large, this proves the lemma.
\end{proof}

Combining all of the inequalities so far, for large $n$ and $N$ we have 
\begin{equation}
\frac{1}{N}H(\theta\conv \eta,\mathcal{D}_{N+\ell})\geq\frac{1}{N}H(\eta,\mathcal{D}_{N})+\delta'\cdot\mathbb{P}^{N}(\mathcal{E}_{2})-(3+\frac{\ell}{N})\delta^2.\label{eq:12}
\end{equation}
Our goal is now to bound $\mathbb{P}^{N}(\mathcal{E}_{2})$ from below by
proving that with non-negligible probability, pairs $(\theta_{g,i},\eta_{x,i})$
satisfy the hypotheses of the Euclidean inverse theorem. Specifically, recall that $\varepsilon'=\varepsilon/26$, and 
set 
\begin{eqnarray*}
\mathcal{E}_{3} & = & \left\{ (\theta_{g,i},\eta_{x,i})\,:\,\begin{array}{l}
\eta_{x,i}\mbox{ is }(1-\eps',\delta',m)\mbox{-entropy porous}\\
\mbox{from scale }i\mbox{ to scale }i+n
\end{array}\right\}, \\
\mathcal{E}_{4} & = & \left\{ (\theta_{g,i},\eta_{x,i})\,:\,\frac{1}{n}H(S_{-\ell(g,x)}\theta_{g,i}\conv x,\mathcal{D}_{i+n})>2\eps'\right\}. 
\end{eqnarray*}

\begin{lem}
For $n$ sufficiently large, we have
\begin{equation} \label{claim1}
\Ek_1\cap\mathcal{E}_{3}\cap\mathcal{E}_{4}\subseteq\mathcal{E}_{2}.
\end{equation}
\end{lem}

\begin{proof}
For $ (\theta_{g,i},\eta_{x,i})\in \Ek_1$, we have $x\not\in B_\rho(u_g^-)$, hence $\theta_{g,i}\conv x$ is supported on a set of diameter $O_\rho(2^{-(i+\ell(g,x))})$, so $S_{-\ell(g,x)}\theta_{g,i}\conv x$ is supported on a set of diameter $O_\rho(2^{-i})$. Thus, by 
(\ref{ent:supp1}) we have for $ (\theta_{g,i},\eta_{x,i})\in \Ek_1\cap \Ek_4$,
\begin{align*}
  \frac{1}{n}H(S_{-\ell(g,x)}\theta_{g,i}\conv x,\mathcal{D}_{i+n}\,|\,\Dk_{i}) & \geq \;\; \frac{1}{n}H(S_{-\ell(g,x)}\theta_{g,i}\conv x,\mathcal{D}_{i+n}) - \frac{O_\rho(1)}{n} \qquad \\
    &> \;\;  \frac{\eps}{26}=\eps',
\end{align*}
for $n$ sufficiently large. Now, by  our choice of parameters, Theorem \ref{thm:entropy-growth-under-convolution} applies for all $ (\theta_{g,i},\eta_{x,i})\in \Ek_1\cap\Ek_3\cap \Ek_4$,
with the conclusion that 
\begin{align*}
\frac{1}{n}H\left((S_{-\ell(g,x)}\theta_{g,i}\conv x)*\eta_{x,i},\mathcal{D}_{i+n}\,|\,\Dk_{i}\right) &
>\; \frac{1}{n}H(\eta_{x,i},\mathcal{D}_{i+n}|\mathcal{D}_{i})+\delta'\\
 &=\; \frac{1}{n}H(\eta_{x,i},\mathcal{D}_{i+n})+\delta'.
\end{align*}
Thus (\ref{claim1}) has been verified.
\end{proof}

\begin{lem}
  For $N$ large enough relative to $\delta,n$,
  \[
  \mathbb{P}^{N}\left(\mathcal{E}_{3}\right)>1-\sqrt{2\delta}.
  \]
\end{lem}

\begin{proof}
By hypothesis (i), $\eta$ is $(1-\eps,m,\delta)$-porous 
from scale $1$ to $N$, so by Lemma \ref{lem:bounded-component-entropy-passes-to-components}, if $N$ is large enough relative to $\delta$ and $n$,
\[
\mathbb{P}_{0\leq i<N}\left(\begin{array}{c}
\eta_{x,i}\mbox{ is } (1-\eps,m,\sqrt{2\delta})\mbox{-entropy}\\
\mbox{ porous from scale }i\mbox{ to }i+n
\end{array}\right)>1-\sqrt{2\delta}.
\]
Since  $\delta'>\sqrt{2\delta}$, the conclusion follows. 
\end{proof}

\begin{lem}
  For $\delta$ small enough, $n$ large relative to $\rho$ and  $N$ large enough,\[
  \mathbb{P}^{N}\left(\mathcal{E}_{4}\right)\geq \frac{\varepsilon}{5\log M}.
\]
\end{lem}

\begin{proof}
We are assuming $\frac{1}{N}H(\theta,\mathcal{D}_{N}^{G})>\varepsilon$,
and can assume $N$ is large relative to $\varepsilon,m$, so by Lemma
\ref{lem:multiscale-formula-for-entropy}(ii), 
\[
\mathbb{E}_{0\leq i\leq N}\left(\frac{1}{n}H(\theta_{g,i},\mathcal{D}^G
_{i+n})\right)>\frac{\varepsilon}{2}.
\]
Note that $\frac{1}{n} H(\theta_{g,i}, \Dk^G_{i+n}) \le \log M$ by the definition of dyadic partitions on $G$, 
hence
\begin{equation}
\mathbb{P}_{0\leq i\leq N}\left(\frac{1}{n}H(\theta_{g,i},\mathcal{D}^G_{i+n})>\frac{\varepsilon}{4}\right)>\frac{\varepsilon}{4\log M}.\label{eq:9}
\end{equation}
By hypothesis (ii) we can apply Lemma \ref{lem:3-point-separation-and-entropy}
to $\eta$. Since $\theta_{g,i}$ is supported on a set of diameter
$O(2^{-i})$, by that lemma there exists an $i_{1}=i_{1}(\rho)\in\mathbb{N}$,
such that for $i>i_{1}$, with $\eta$-probability at least $1-4\delta^2$,
a point $x\in\RP$ satisfies\footnote{Replacing $\ell$ in  Lemma \ref{lem:3-point-separation-and-entropy}  by $\ell(g,x)$ in the partition here costs $O(1/n)$, which is absorbed in the error term.} 
\begin{equation}
\frac{1}{n}H(\theta_{g,i}\conv x,\mathcal{D}_{\ell(g,x)+i+n})>\frac{1}{3}\cdot\frac{1}{n}H(\theta_{g,i},\mathcal{D}^G_{i+n})-O_{\rho}(\frac{1}{n}).\label{eq:13}
\end{equation}
Thus if $\theta_{g,i}$ belongs to the event in (\ref{eq:9}),
then with $\eta$-probability at least $1-4\delta^2$
over the choice of $x$, and assuming $n$ large enough relative to
$\rho$, we have 
\begin{eqnarray*}
  \frac{1}{n}H(S_{-\ell(g,x)}\theta_{g,i}\conv x,\mathcal{D}_{i+n}) & = &\frac{1}{n}H(\theta_{g,i}\conv x,\mathcal{D}_{\ell(g,x)+i+n}) + O(\frac{1}{n}) \\
  & > & \frac{\varepsilon}{12}-O_{\rho}(\frac{1}{n})\\
 & > & \frac{\varepsilon}{13}.
\end{eqnarray*}
Thus, assuming $N$ is large relative to $i_{1}$ (i.e. relative to
$\rho$), and $\delta$ is small relative to $\eps$,
\begin{eqnarray*}
\mathbb{P}^{N}\left(\mathcal{E}_{4}\right) & = & \mathbb{P}_{0\leq i\leq N}\left((\theta_{g,i},\eta_{x,i})\,:\,\frac{1}{n}H(\theta_{g,i}\conv x,\mathcal{D}_{i+\ell(g,x)+n})>\frac{\varepsilon}{13}\right)\\
 & \geq & \mathbb{P}_{i_{1}\leq i\leq N}\left((\theta_{g,i},\eta_{x,i})\,:\,\begin{array}{c}
\theta_{g,i}\mbox{ is in the event in \eqref{eq:9} }\\
\mbox{ and }x\mbox{ satisfies \eqref{eq:13} for }\theta_{g,i}
\end{array}\right)-\frac{i_{1}}{N}\\
 & \geq & \frac{\varepsilon}{4\log M}(1-4\delta^2)-\frac{i_{1}}{N}\\[1.2ex]
 & > & \frac{\varepsilon}{5\log M}.
\end{eqnarray*}
\end{proof}

We are now ready to complete the proof of the theorem. By (\ref{claim1}), for  $\delta$ sufficiently small, and $m$,$n$ and $N$ each large enough relative to the previous parameters, we have
\begin{eqnarray*}
\mathbb{P}^{N}(\mathcal{E}_{2}) & \geq & \mathbb{P}^{N}(\Ek_1\cap \mathcal{E}_{3}\cap\mathcal{E}_{4})\\
 & \geq & \mathbb{P}^{N}(\mathcal{E}_{4})-(1-\mathbb{P}^{N}(\mathcal{E}_{3}))-(1-\mathbb{P}^N(\Ek_1))\\
 & \geq & \frac{\varepsilon}{5\log M}-\sqrt{2\delta} - \delta\\
 & \geq & \frac{\varepsilon}{10\log M}.
\end{eqnarray*}
Plugging this into (\ref{eq:12}), we get
\begin{eqnarray*}
  \frac{1}{N}H(\theta\conv \eta,\mathcal{D}_{N+\ell}) & \geq & \frac{1}{N}H(\eta,\mathcal{D}_{N})+\frac{\varepsilon}{10\log M}\cdot\delta'-(3+\frac{\ell}{N})\delta^2\\
  & = & \frac{1}{N}H(\eta,\mathcal{D}_{N})+\frac{\varepsilon}{10\log M}\cdot\frac{40\log M \sqrt{\delta}}{\varepsilon}-(3+\frac{\ell}{N})\delta^2\\
  & = & \frac{1}{N}H(\eta,\mathcal{D}_{N})+4\sqrt{\delta}-(3+\frac{\ell}{N})\delta^2\\
  & \ge  & \frac{1}{N}H(\eta,\mathcal{D}_{N})+\delta-\frac{\ell}{N}\delta^2.
\end{eqnarray*}
This completes the proof.
\end{proof}

\subsection{Proof of Theorem \ref{thm:main}}

We continue to argue by contradiction, starting with Proposition \ref{prop:reduction}. Let $c,c'$ as in that proposition and fix a small  $0<\eps<1-\dim\nu$ satisfying $0<\eps<c'/2c$. By Proposition \ref{prop:approx} the stationary measure $\nu$ is $(1-\eps)$-entropy porous. Fix $\delta_1(\eps)$ provided by Theorem \ref{thm:entropy-growth-under-action}. We are aiming to 
  get a contradiction with conclusions of Theorem \ref{thm:entropy-growth-under-action}.
  For any $0<\delta<\delta_1(\eps)$ the measure $\eta=\nu$ satisfies condition (ii) for $\rho>0$ sufficiently small, since $\nu$ is non-atomic by assumption. Condition (i) holds by the definition of entropy porosity, for $m$ and $N$ sufficiently large.
Further, $\theta$ is going to be one of the raw components $ (\mu^{\star n})_{g,1}$, having support of diameter $O(1)$, by our choice of the dyadic partition on $G$. Thus, Theorem \ref{thm:entropy-growth-under-action} applies.
The first conclusion of Proposition \ref{prop:reduction} is
$$
\mathbb{E}_{i=1}\left(H((\mu^{\star n})_{g,i},\mathcal{D}_{cn}^{G})\right) \geq c'\cdot n-o(n).
$$
Recall that $H((\mu^{\star n})_{g,1}, \Dk_{cn}^G) \le \log M\cdot cn$ for all $g\in G$.
Let
$$
I_n = \{g\in G:\ H((\mu^{\star n})_{g,1}, \Dk_{cn}^G) >\eps cn\},
$$
then, conditioning the previous expectation on $I_n$ and its complement,
\begin{eqnarray*}
  c'n-o(n) & \le & \mu^{\star n}(I_n)\cdot \mathbb{E}_{i=1}\left(H((\mu^{\star n})_{g,i},\mathcal{D}_{cn}^{G})\Big|\,g\in I_n\right) + \\
  & & \;+\;\mu^{\star n}(G\setminus I_n)\cdot\mathbb{E}_{i=1}\left(H((\mu^{\star n})_{g,i},\mathcal{D}_{cn}^{G})\Big|\,g\notin I_n\right)\\
  & \le & \mu^{\star n}(I_n)\cdot \log M\cdot cn + \eps cn.
\end{eqnarray*}
Since we chose $\eps$ satisfying that $0 < \eps < c'/2c$, we obtain
\be \label{tutu}
\mu^{\star n}(I_n) \ge \gam:= \frac{c'}{2c\cdot \log M}
\ee
for $n$ sufficiently large. Further, let $\epsilon_n\to 0$ and define as in Proposition \ref{prop:convergence-of-convolution-entropy-to-dimension}
\[
\Gamma_{n}=\{g\in G\,:\,2^{-(2\chi+\epsilon_{n})n}\leq\left\Vert g\right\Vert ^{-2}\leq2^{-(2\chi-\epsilon_{n})n}\}.
\]
By Theorem \ref{thm:Oseledets} we can choose $\epsilon_{n}$ so
that
$$
\mu^{\star n}(\Gamma_{n})\rightarrow 1.
$$
Theorem \ref{thm:entropy-growth-under-action} then implies that for $g\in I_n\cap\Gamma_n$ we have
\be \label{tata}
H((\mu^{\star n})_{g,1}\conv \nu, \Dk_{cn+2\log\|g\|}) > H(\nu, \Dk_{cn}) + cn\delta- 3\chi n\delta^2,
\ee
for $n$ sufficiently large. Observe that for $g\in \Gamma_n$,
\be \label{trata}
H((\mu^{\star n})_{g,1}\conv \nu, \Dk_{cn+2\chi n}) \ge H((\mu^{\star n})_{g,1}\conv \nu, \Dk_{cn+2\log\|g\|}) -o(n),
\ee
by Lemma~\ref{lem:entropy_perturb}.
We will need the following lemma.

\begin{lem}\label{lem:entropy_conv}
For $g\in \Gamma_n$ we have
$$
H((\mu^{\star n})_{g,1}\conv \nu, \Dk_{cn+2\chi n}) \ge H(\nu,\Dk_{cn})-\delta^2O(n) -o(n).
$$
\end{lem}

First we finish the proof of the theorem, and then give the proof of the lemma.
Our goal is to get a contradiction with the second conclusion of Proposition \ref{prop:reduction}, which asserts
\begin{equation} \label{assert}
\E_{i=1}(H((\mu^{\star n})_{g,i}\conv \nu, \Dk_{cn+2\chi n})) \le H(\nu, \Dk_{cn}) + o(n).
\end{equation}
Conditioning on $I_n$ and using (\ref{tata}), (\ref{trata}), and then conditioning on $\Gamma_n\setminus I_n$ and using the last lemma, we obtain
\begin{eqnarray*}
\E_{i=1}(H((\mu^{\star n})_{g,i}\conv \nu, \Dk_{cn+2\chi n})) & \ge & \mu^{\star n}(I_n)\cdot [H(\nu, \Dk_{cn}) + cn\delta- 3\chi n\delta^2] \\[1.2ex]
&  & \;+\;\mu^{\star n}(\Gamma_n\setminus I_n)\cdot [H(\nu,\Dk_{cn})-\delta^2O(n) -o(n)] \\[1.2ex]
& \ge & \mu^{\star n}(\Gamma_n)\cdot H(\nu, \Dk_{cn}) + \\[1.2ex] & & \;+\;\gam[cn\delta- 3\chi n\delta^2]  - \delta^2O(n)-o(n) \\[1.2ex]
& = & H(\nu,\Dk_{cn}) \\[1.2ex] & & \;+\; \gam[cn\delta- 3\chi n\delta^2] - \delta^2O(n)-o(n). 
\end{eqnarray*}
In the second inequality we used that 
$$
(1- \mu^{\star n}(\Gamma_n)) H(\nu, \Dk_{cn}) \le cn (1- \mu^{\star n}(\Gamma_n)) = o(n).
$$
For $\delta$ sufficiently small and $n$ sufficiently large we get a contradiction with (\ref{assert}), as desired. It remains to prove Lemma~\ref{lem:entropy_conv}.

\begin{proof}[Proof of the lemma] Let $\theta = (\mu^{\star n})_{g,1}$. We have
$$
(\mu^{\star n})_{g,i}\conv \nu = \theta \conv \nu = \int g\nu\,d\theta(g),
$$
hence by concavity of the entropy,
\begin{equation} \label{concav1}
H((\mu^{\star n})_{g,1}\conv \nu, \Dk_{cn+2\chi n})) \ge \int H(g\nu,\Dk_{cn+2\chi n})\,d\theta(g).
\end{equation}
Consider the conditional measures $\nu_g^-= \nu_{B_\rho(u_g^-)}$ and $\nu_g^+:= \nu_{\RP\setminus B_\rho(u_g^-)}$; we have
$$
\nu = c_g\nu_g^- + (1-c_g)\nu_g^+,\ \ \mbox{where}\ \ c_g = \nu(B_\rho(u_g^-)) \le \delta^2,
$$
by the choice of $\rho$. Again using concavity of the entropy, we obtain
\begin{eqnarray*}
 H(g\nu,\Dk_{cn+2\chi n}) & \ge & (1-c_g)\cdot H(g\nu_g^+ ,\Dk_{cn+2\chi n})\\[1.2ex] & \ge &  (1-\delta^2)\cdot H(g\nu_g^+,\Dk_{cn+2\chi n})\\[1.2ex]
& \ge &  H(g\nu_g^+,\Dk_{cn+2\chi n}) - \delta^2 O(n).
\end{eqnarray*}
Recall that the action of $g$ scales $\RP\setminus B_\rho(u_g^-)\supseteq \supp\nu_g^+$ by $\|g\|^{-2}$ with distortion $O_\rho(1)$, hence 
\begin{eqnarray*}
H(g\nu_g^+, \Dk_{cn+2\chi n}) & = & H(\nu_g^+, \Dk_{cn + 2\chi n - 2\log\|g\|}) + O_\rho(1) \\[1.2ex]
& \ge & H(\nu_g^+, \Dk_{cn}) - \eps_n  n  - O_\rho(1),
\end{eqnarray*}
by Lemma \ref{lem:entropy_distort} and the definition of $\Gamma_n$. Finally,
$$
H(\nu_g^+, \Dk_{cn}) \ge H(\nu, \Dk_{cn}) -c_g H(\nu_g^-, \Dk_{cn}) -H(\delta^2) \ge H(\nu,\Dk_{cn}) - \delta^2 O(n),
$$
by the almost-convexity of entropy. Combining this with the last two inequalities and substituting into (\ref{concav1}) yields the desired claim. This concludes the proof of the lemma and of the main theorem.
\end{proof}


\section{Applications and examples} \label{section:appl}
 
 Denote by $G^+_\Ak$ the semigroup generated by $\Ak$ and by $G_\Ak$ the group generated by $\Ak$.
In order to apply Theorem~\ref{thm:main}, we need to check that $\supp\mu$ is 
Diophantine and
$G^+_\mu$ is free. If the latter is not the case, computing $h_{\RW}(\mu)$ exactly is usually impossible, but one may be able to obtain lower bounds yielding lower bounds for $\dim \nu$.

 \subsection{Diophantine property}
 
 We start with a few general comments. Recently there has been interest in the Diophantine property for groups. Following \cite{ABSR2015},
we say that a finitely generated metric group is Diophantine if every nontrivial element of the word ball $B_n(1)$ in the group is separated from 1 by at least $|B_n(1)|^{-\beta}$ for some $\beta>0$ independent of $n$.  Note that our condition that $\Ak$ is Diophantine, for a finite set $\Ak$,  is weaker, at least formally, than the condition that the group $G_\Ak$  is Diophantine and has exponential growth (the conditions are equivalent when $\Ak$ is symmetric).
 It is mentioned in \cite{ABSR2015} that very little is known about the Diophantine property in semi-simple Lie groups, although it is conjectured that a random $k$-tuple has this property. See \cite{ABSR2015} for further references.
 
 The following lemma is standard; we provide a proof for completeness.
 
\begin{lem} \label{alg-dioph}
Suppose that all the entries of the matrices in $\supp\mu $ are algebraic. Then $\supp\mu$ is  Diophantine.
\end{lem}
 
 \begin{proof}
 
 Before starting the proof, note that in the case of rational entries the argument is immediate.

Let $a_1,\ldots,a_k$ be all the entries of the matrices in $\Ak:=\supp\mu\subset G$. It is easy to see by induction that for any product $A_1\ldots A_n$, with $A_i\in \Ak$, its entries are
integer polynomials in $a_1,\ldots,a_k$ of degree $n$ and coefficients bounded by $2^n$.   Let $f(x_1,..,x_k)$ be an integer polynomial of degree $n$ and coefficients bounded by $H$ in absolute value. Assuming $f(a_1,...,a_k)$ is not zero, it suffices to bound it from below in absolute value by $c^n/H^u$ for some $c,u>0$ depending only on the $\{a_i\}$.
Let $F = \Q(a_1,\ldots,a_k)$ be the field over $\Q$ generated by $\{a_i\}$.

{\sc Claim:} We may assume that $a_i$ are algebraic integers. This is because we can choose positive integers $p_1,...,p_k$ such that $b_i = p_i\cdot a_i$ is an algebraic integer. Let $p=p_1\cdot\ldots\cdot p_k$ (note that this depends only on the $a_i$). Then
$$p^n \cdot f(a_1,...,a_k) = g(b_1,...,b_k),$$
and $g$ is an integer polynomial of degree $n$ with coefficients bounded by $H\cdot p^n$. So if we have $c=c(b_1,...,b_k)>0$ such that $g(b_1,...,b_k)>c^n/(Hp^n)^u$, then $f(a_1,...,a_k)>c^n/(H\cdot p^{(u+1)n})$, which is what we wanted (using the constant $c/p^{u+1}$ instead of $c$).

Assuming now that $a_i$ are algebraic integers, let $F'$ be the normal closure of $F=\Q(a_1,...,a_k)$ and $\Gam=Gal(F'/\Q)$, so the fixed field of $\Gam$ is $\Q$. Note that $F'$, hence $\Gam$, depends only on the $a_i$, and 
$\Gam$ is finite.

Now we do the usual thing: if $f(x_1,...,a_k)$ is not zero then also $\prod_{s \in \Gam} s(f(x))$ is non-zero, but it is both an algebraic integer and rational, so its absolute value is at least 1. Hence
 $$1 \le \prod_{s \in \Gam} |f(sx)|
     = | f(x)| \cdot  \prod_{s \in \Gam\setminus \{id\}} |f(sx)|.$$
\noindent The last product has $|\Gam|-1$ factors $|f(sx)|$, each of size at most \\ $H\cdot \max{|\Gam\mbox{-conjugates of}\ a_i|}^n$. Dividing gives the bound that we want. 
 \end{proof}
 
 As already mentioned, the Diophantine property in groups is hard to check, especially for non-amenable groups; for example, it is an open problem whether almost every pair of elements and their inverses in $SU(2)$ is Diophantine (see \cite{GJS1999}).
 However, in some cases (rather special) we can obtain results using {\em transversality methods}.
 
 \begin{defi}
 Suppose that the set $\Ak^{(\lam)}=\{A_1^{(\lam)},\ldots,A_m^{(\lam)}\} \subset PSL_2(\R)$, $\lam\in I$,
 depends on the parameter $\lam$ continuously, where $I\subset \R$ is an interval. We say that the family 
 $\{\Ak^{(\lam)}, \lam\in I\}$ satisfies a transversality condition of order $k\ge 1$ if there exists $C>0$ such that for any $n\in \N$ and any $i_1,\ldots,i_n$ and $j_1,\ldots,j_n$ in $\{1,\ldots,m\}^n$, with
 $i_1\ne j_1$, either 
 $$
 A_{i_1}^{(\lam)}A^{(\lam)}_{i_2}\cdots A^{(\lam)}_{i_n} - A_{j_1}^{(\lam)}A^{(\lam)}_{j_2}\cdots A^{(\lam)}_{j_n} \equiv 0\ \ \mbox{for all}\ \lam\in I,
 $$
 or 
 $$
 \left|\bigl\{\lam\in I:\ \|A_{i_1}^{(\lam)}A^{(\lam)}_{i_2}\cdots A^{(\lam)}_{i_n} - A_{j_1}^{(\lam)}A^{(\lam)}_{j_2}\cdots A^{(\lam)}_{j_n}\|\le r\bigr\}\right| \le Cr^{1/k}\ \ \ \mbox{for all}\ 
 r>0.
 $$
 \end{defi}
 
 The following lemma is standard and easy, see \cite[Section 5.4]{Hochman2014} for details.
 
 \begin{lem}
 Suppose that $\{\Ak^{(\lam)}, \lam\in I\}$ satisfies the transversality condition of some order $k\ge 1$. Then the set
 $$
 \{\lam\in I:\ \Ak^{(\lam)} \mbox{\ is not Diophantine}\}
 $$
 has packing dimension zero.
 \end{lem}

 \begin{example}
 $\Ak^{(\lam)} = \left\{\left[\begin{array}{cc} 1 & 0 \\ 1 & 1 \end{array} \right], \left[\begin{array}{cc} 1 & \lam \\ 1 & 1+\lam \end{array} \right]\right\}$
 \end{example}
 
 The action of $\Ak^{(\lam)}$ on the projective line can be expressed using a linear fraction representation, which corresponds to a specific choice of charts, yielding an iterated function system (IFS)
 $\{\frac{x}{x+1}, \frac{x+\lam}{x+1+\lam}\}$. The transversality condition of order one for this parametrized IFS has been verified in \cite{SSU2001} for $\lam\in [0.215,0.5]$ (for $\lam\ge 0.5$ the IFS satisfies the Open Set  Condition, which corresponds to the case of an obviously free semigroup with a Diophantine property). This implies transversality for $\Ak^{(\lam)}$. 
 It is clear that the group generated by $\Ak^{(\lam)}$ is free when $\lam$ is transcendental. Let $\nu_\lam^p$ be the stationary measure for the $(p,1-p)$ Bernoulli measure supported on $\Ak^{(\lam)}$. Our theorem implies that the formula $\dim_H(\nu_\lam^p) = \min\{H(p)/2\chi_\lam^p,1\}$, where $2\chi_\lam^p$ is the corresponding Lyapunov exponent, holds for all $\lam$ outside a set of packing dimension zero in the transversality interval, uniformly in $p$. In particular, $\dim_H(\nu_\lam^{1/2})=1$ for all $\lam\in [0.215,0.268]$  outside a set of packing dimension zero, see \cite[Corollary 6.3]{SSU2001}. For comparison, the exceptional set in \cite{SSU2001} is of Lebesgue measure zero and depends on $p$; on the other hand, in \cite{SSU2001}  absolute continuity is proved a.e.\ in the ``super-critical'' region
 $\{\lam:\, H(p) > 2\chi_\lam^p\}$, which we do not handle in this paper.
 
 Other examples of a similar kind may be found in \cite{BaranyPollicottSimon2012}.
 We do not know how to verify higher order transversality for such families. 
 
 \subsection{Free semigroups}
 There are many papers on freeness and non-freeness of specific subgroups and semigroups of $SL_2(\C)$; see e.g.\  \cite{Bamberg2000} and references therein. Many of the papers focus on the set
 \be \label{Slam}
 S_\lam = \left\{\left[\begin{array}{cc} 1 & 0 \\ \lam & 1\end{array}\right], \left[\begin{array}{cc} 1 & \lam \\ 0 & 1\end{array}\right]\right\}=:\{A_\lam, B_\lam\}.
 \ee
 It is known that the semigroup $G^+_{S_\lam}$ is free when $|\lam|\ge 1,\ |\arg(\lam)|\le \pi/4$, see \cite{BrennerCharnow1978}. (For $\lam\ge 1$ real, this is easy to see from the fact that $A_\lam(\R^2_{++})\cap B_\lam(\R^2_{++}) = \es$, where
 $\R^2_{++}$ is the open 1st quadrant $x>0, y>0$.) It is also clear that if $\lam$ is algebraic, then $G^+_{S_\lam}$ is free if and only if $G^+_{S_{\lam'}}$ is free, for $\lam'$ a Galois conjugate of $\lam$. It is known that
the semigroup $G^+_{S_\lam}$ may be free when the group $G_{S_\lam}$ is nonfree. Observe that the Lyapunov exponent  of a measure $\mu$ supported on $\Ak$ may be estimated by
\be \label{Lyapex}
2\chi(\mu) \le 2\max\{\log\|A\|:\,A\in \Ak\},
\ee
which in the case of $\Ak = S_\lam$ yields $2\chi\le 2\log(1+|\lam|)$.
We thus obtain the following:

\begin{cor}
Suppose that $\lam\in \R$ is algebraic, $|\lam|\le \sqrt{2}-1$, and one of the Galois conjugates $\lam'$ of $\lam$ satisfies $|\lam'|\ge 1,\ |\arg(\lam')|\le \pi/4$. Then the stationary measure $\nu_\lam$, corresponding to the uniform  measure $\{1/2,1/2\}$ on
$S_\lam$, has $\dim \nu_\lam=1$.
\end{cor}
 
 \subsection{Strong Tits alternative and applications} The following result is stated in the special case of $GL_2(\C)$. Below all norms are 2-norms, unless otherwise stated.
 
 \begin{thm}[{\cite[Cor. 1.5]{Breuillard2011}}] \label{th-Tits}
 There exists $\varepsilon>0$ such that if $F$ is a finite subset of $GL_2(\C)$ containing $1_G$ and generating a non-amenable subgroup $\Gamma$, and if $f\in\ell^2(\Gamma)$, then there exists $g\in F$ such that $\Vert f-f\circ g^{-1}\Vert \geq \eps \Vert f\Vert$.
 \end{thm}

We now demonstrate how this can be used to prove Theorem \ref{thm:appl1}. Recall that $\star$ denotes the convolution operation in $G$.
 
\begin{lem}
Suppose that $\mu\in\mathcal{P}(G)$. Assume that $1_{G}\in\supp\mu$,
that $\mu$ is purely atomic and all atoms have mass at least $\delta$,
and that for every $f\in\ell^{2}(G_\mu)$ there is a $g\in\supp\mu$ with
$\left\Vert f-f\circ g^{-1}\right\Vert \ge \delta\left\Vert f\right\Vert $.
Then there exists an $\varepsilon=\varepsilon(\delta)>0$ such that
$\left\Vert \mu\star{}f\right\Vert \leq(1-\varepsilon)\left\Vert f\right\Vert $
for all $f\in\ell^{2}(G_\mu)$, and $h_{\RW}(\mu)\geq -\log(1-\varepsilon)$.\end{lem}
\begin{proof}
First, let $f\in\ell^{2}(G_\mu)$ and choose $g\in\supp\mu$
as guaranteed by the hypothesis. Then $\mu\star{}f$ is a convex combination
of $\ell^{2}$ functions of norm $\left\Vert f\right\Vert $, among
which appear the functions $f$ (corresponding to the action of $1_{G}$)
and $f\circ g^{-1}$. By choice of $g$ these functions are at least
$(1-\delta)\left\Vert f\right\Vert $ apart, and their weights
in the convex combination are at least $\delta$, so by uniform convexity
of the norm in $\ell^{2}$ there exists an $\varepsilon>0$ such that
$\left\Vert \mu\star{}f\right\Vert <(1-\varepsilon)\left\Vert f\right\Vert $.

For the second statement note that by the above, $\left\Vert \mu^{\star{}n}\star{}f\right\Vert <(1-\varepsilon)^{n}\left\Vert f\right\Vert $,
and that the $\ell^{2}$-norm dominates the $\ell^{\infty}$ norm
for functions on $G_\mu$. Thus for $f=1_{\{1_{G}\}}$ we conclude that
\begin{eqnarray*}
  \sup_{g\in G}\mu^{\star{}n}(g) & = & \left\Vert \mu^{\star{}n}\star{}1_{\{1_{G}\}}\right\Vert _{\infty}\\
  & \leq &\left\Vert \mu^{\star{}n}\star{}1_{\{1_{G}\}}\right\Vert \\
  & \leq &(1-\varepsilon)^{n}\left\Vert 1_{\{1_{G}\}}\right\Vert \\
  & = & (1-\varepsilon)^{n},
\end{eqnarray*}
which, substituting the bound into the definition of Shannon entropy, immediately gives that
\begin{eqnarray*}
\frac{1}{n}H(\mu^{\star{}n}) & > & -\log(1-\varepsilon).
\end{eqnarray*}
Letting
$n\to\infty$ proves the second claim.
\end{proof}
Let $\mu\in\mathcal{P}(G)$ be finitely supported, and for $\varepsilon>0$
let 
\[
\mu_{\varepsilon}=\varepsilon\delta_{1_{G}}+(1-\varepsilon)\mu.
\]

\begin{lem}
$h_{\RW}(\mu)=h_{\RW}(\mu_\eps)/(1-\varepsilon)$.\end{lem}
\begin{proof}
Since $\nu\star{}\delta_{1_{G}}=\nu$ for every $\nu\in\mathcal{P}(G)$
and convolution is bilinear, we have
\[
(\mu_{\varepsilon})^{\star{}n}=\sum_{k=1}^{n}\binom{n}{k}\varepsilon^{k}(1-\varepsilon)^{n-k}\mu^{\star{}n-k},
\]
so $(\mu_{\varepsilon})^{\star{}n}$ is a convex combination of $\mu^{\star{}k}$,
$k=0,\ldots,n$, with weights $p_{n,k}=\binom{n}{k}\varepsilon^{k}(1-\varepsilon)^{n-k}$
which tend to zero uniformly as $n\rightarrow\infty$. By concavity
and almost convexity of entropy, we have 
\[
\sum_{k=0}^{n}p_{n,k}\cdot\frac{1}{n}H(\mu^{\star{}k})\leq\frac{1}{n}H((\mu_{\varepsilon})^{\star{}n})\leq\sum_{k=0}^{n}p_{n,k}\cdot\frac{1}{n}H(\mu^{\star{}k})+\frac{1}{n}H(p_{n}).
\]
Since $p_{n}$ is supported on $0,\ldots,n$, we have $H(p_{n})\leq\log(n+1)$
and we conclude that
\begin{eqnarray*}
\lim_{n\rightarrow\infty}\frac{1}{n}H((\mu_{\varepsilon})^{\star{}n}) & = & \lim_{n\rightarrow\infty}\sum_{k=0}^{n}p_{n,k}\cdot\frac{1}{n}H(\mu^{\star{}k})\\
 & = & \lim_{n\rightarrow\infty}\sum_{k=0}^{n}p_{n,k}\cdot\frac{1}{n}k(h_{\RW}(\mu)+o(k))\\
 & = & h_{\RW}(\mu)\sum_{k=0}^{n}p_{n,k}\cdot\frac{k}{n}\\
 & = & (1-\varepsilon)h_{\RW}(\mu),
\end{eqnarray*}
where we used that $H(\mu^{\star{}k})=k(h_{\RW}(\mu)+o(1))$ as $k\rightarrow\infty$
and that the distribution $p_{n}$ is binomial with parameters $(n,1-\varepsilon)$
and so the mean value of $k$ under the distribution $p_{n,k}$ is
$1-\varepsilon$.\end{proof}
\begin{cor}
For every $\delta>0$ there exists  $\rho>0$ such that if
$\mu\in\mathcal{P}(G)$ is purely atomic, all its atoms have mass
at least $\delta$, and $\supp \mu$ generates a non-amenable group,
then $h_{\RW}(\mu)>\rho$.\end{cor}
\begin{proof}
By Theorem \ref{th-Tits}, there is a $\delta>0$ such that if $S\subseteq \Pk(G)$
contains the identity and generates a non-amenable group, then for
every $f\in\ell^{2}(G)$ there exists some $g\in S$ such that
$\left\Vert f-f\circ g^{-1}\right\Vert \ge \delta\left\Vert f\right\Vert $.
Now suppose that $\mu\in\Pk(G)$ is purely atomic with atoms of mass
at least $\delta$, and $\supp\mu$ generates a non-amenable group.
Then by the first lemma above (with $\delta(1-\delta)$ instead of $\delta$), $h_{\RW}(\mu_{\delta})>\rho>0$
for some $\rho=\rho(\delta)>0$, and by the second lemma, the same is true
for $\mu$ (with $\rho/(1-\delta)$ instead of $\rho$).
This is what we wanted.\end{proof}

\begin{proof}[Proof of  Theorem~\ref{thm:appl1}] 
Fix $\rho>0$ and $\delta>0$ as in the last corollary. If $\mu$ is supported close enough to the identity of $G$, we will have $\chi(\mu)<\rho/2$, and therefore, if $G_\mu$ is non-amenable, we have $h_{\RW} (\mu)/2\chi(\mu)>1$. If, furthermore, $G_\mu$ is unbounded and totally irreducible, Theorem~\ref{thm:main} implies that $\dim\nu=1$. This is exactly Theorem~\ref{thm:appl1} except that in the statement above we have an additional non-amenability assumption. But non-amenability follows from the unboundedness and total irreducibility of $G_\mu$. This implication is  standard, but we sketch a proof for the reader's convenience.

Let $H=G_\mu$. The assumptions are that $H$ is unbounded and totally irreducible. Then $\mu$ has positive Lyapunov exponent and non-atomic stationary measure $\nu$. If $H$ is amenable, then there is an invariant measure on $\RP$ for the action of $H$. Since invariant measure is stationary, from uniqueness of the latter it follows that $\nu$ is invariant. Then $\nu$ is invariant under the action of every element of $H$. Positive Lyapunov exponent implies that
$H$ contains hyperbolic matrices, for which the invariant measure must be supported on the fixed points, that is, the two eigendirections. However, in that case all the eigendirections for such elements of $H$ coincide and the
measure $\nu$ is atomic, a contradiction.
\end{proof} 
 
 \begin{example}
Theorem~\ref{thm:appl1} applies to measures $\mu$ supported on $S_\lam^{\pm}:=S_\lam \cup S_{-\lam}$, where $S_\lam$ is from (\ref{Slam}), with $\lam$ algebraic. 
 \end{example}

\bibliographystyle{plain}
\bibliography{bib}

\footnotesize{

\noindent  Einstein Institute of Mathematics, 
Edmond J. Safra Campus (Givat Ram), 
The Hebrew University, 
Jerusalem 91904, 
Israel.\\
\texttt{mhochman@math.huji.ac.il}

\bigskip
\noindent Department of Mathematics, 
University of Bar-Ilan 
Ramat-Gan, 5290002, Israel \\
\texttt{bsolom3@gmail.com}
}

\end{document}